\documentclass[a4paper,12pt]{amsart}
\RequirePackage{plautopatch}
\RequirePackage[l2tabu, orthodox]{nag}
\usepackage{bxtexlogo}
\usepackage{amsfonts}
\usepackage{bbm}
\usepackage{amssymb, amsmath,stmaryrd,mathrsfs,mathtools}
\usepackage{tikz,float}
\usepackage{tikz-cd}
\usepackage{amsthm}
\usepackage{enumerate}
\usepackage{enumitem}

\usetikzlibrary{calc}
\usepackage{url}
\usepackage[hidelinks]{hyperref}
\usepackage{cleveref}
\usepackage{graphicx}
\usepackage {diagbox}

\usepackage{fullpage}

\usepackage[
 defernumbers=true,
 maxbibnames=6,
 backend=biber,
 style=alphabetic,
 giveninits=false,
 sortcites=true,
 maxalphanames=4,
 minalphanames=3,
 doi=true,
 isbn=false,
 url=false,
 backend=biber,
 safeinputenc
]{biblatex}

\usepackage[T1]{fontenc}
\usepackage[utf8]{inputenc}
\usepackage{bxtexlogo}

\newtheorem{thm}{Theorem}
\newtheorem{prop}[thm]{Proposition}
\newtheorem{lem}[thm]{Lemma}
\newtheorem{cor}[thm]{Corollary}
\newtheorem{fac}[thm]{Fact}

\newtheorem{claim}[thm]{Claim}

\newcounter{enuAlph}

\newtheorem{teorema}[enuAlph]{Theorem}

\theoremstyle{definition}

\newtheorem{dfn}[thm]{Definition}
\newtheorem{rem}[thm]{Remark}
\newtheorem{que}[thm]{Question}

\crefname{dfn}{Definition}{Definitions}
\crefname{thm}{Theorem}{Theorems}
\crefname{lem}{Lemma}{Lemmas}
\crefname{cor}{Corollary}{Corollaries}
\crefname{prop}{Proposition}{Propositions}

\numberwithin{thm}{section} 
\numberwithin{equation}{section}

\renewcommand{\a}{\alpha}
\renewcommand{\b}{\beta}

\newcommand{\p}{\mathbb{P}}
\newcommand{\q}{\mathbb{Q}}
\newcommand{\br}{\mathbb{R}}

\newcommand{\qd}{\dot{\mathbb{Q}}}

\DeclareMathAlphabet{\mymathbb}{U}{BOONDOX-ds}{m}{n}

\newcommand{\oo}{{\omega}^\omega}
\newcommand{\ooo}{[{\omega}]^\omega}

\newcommand{\fsq}{\omega^{<\omega}}
\newcommand{\sq}{2^{<\omega}}

\newcommand{\on}{\mathpunct{\upharpoonright}}

\newcommand{\bb}{\mathfrak{b}}
\newcommand{\cc}{\mathfrak{c}}
\newcommand{\dd}{\mathfrak{d}}
\newcommand{\ee}{\mathfrak{e}}

\newcommand{\hh}{\mathfrak{h}}
\newcommand{\mm}{\mathfrak{m}}
\newcommand{\pp}{\mathfrak{p}}
\newcommand{\rr}{\mathfrak{r}}

\newcommand{\uu}{\mathfrak{u}}

\newcommand{\m}{\mathcal{M}}
\newcommand{\n}{\mathcal{N}}

\DeclareMathOperator{\non}{non}
\DeclareMathOperator{\cov}{cov}
\DeclareMathOperator{\add}{add}
\DeclareMathOperator{\cof}{cof}

\newcommand{\covm}{\cov(\mathcal{M})}
\newcommand{\nonm}{\non(\mathcal{M})}
\newcommand{\addm}{\add(\mathcal{M})}
\newcommand{\cofm}{\cof(\mathcal{M})}
\newcommand{\covn}{\cov(\mathcal{N})}
\newcommand{\nonn}{\non(\mathcal{N})}
\newcommand{\addn}{\add(\mathcal{N})}
\newcommand{\cofn}{\cof(\mathcal{N})}
\newcommand{\covi}{\cov(I)}
\newcommand{\noni}{\non(I)}
\newcommand{\addi}{\add(I)}
\newcommand{\cofi}{\cof(I)}

\newcommand{\R}{\mathbf{R}}

\renewcommand{\lq}{\preceq_T}

\DeclareMathOperator{\dom}{dom}
\DeclareMathOperator{\ran}{ran}

\DeclareMathOperator{\cf}{cf}

\newcommand{\forces}{\Vdash}

\makeatletter
\renewcommand{\p@enumi}{}
\makeatother

\newcommand{\ed}{\mathcal{ED}}

\newcommand{\lebb}{\mathbb{LE}}

\newcommand{\ind}{i}

\newcommand{\none}{\non(\mathcal{E})}
\newcommand{\cove}{\cov(\mathcal{E})}

\DeclareMathOperator{\Pred}{Pred}

\newcommand{\I}{\mathscr{I}}
\newcommand{\J}{\mathscr{J}}
\DeclareMathOperator{\Exh}{Exh}
\DeclareMathOperator{\Fin}{Fin}

\newcommand{\dif}{\mathbb{P}^{\I_L^f}}

\newcommand{\edfin}{\mathcal{ED}_\mathrm{fin}}
\newcommand{\finfin}{\mathrm{Fin}\otimes\mathrm{Fin}}

\newcommand{\cal}{\mathcal}

\newcommand{\rb}[1]{(#1)}    

\newcommand{\set}[2]
{\left\{{#1}:{#2}\right\}}  

\newcommand{\card}[1]{\left|#1\right|}   

\newcommand{\calA}{\mathcal{A}}

\newcommand{\ki}{K_\I}
\newcommand{\mzb}{M}
\newcommand{\mzbi}{\mzb_{\I}}

\newcommand{\sbd}{\trianglelefteq^*}
\newcommand{\Z}{\mathcal{Z}}

\DeclareMathOperator{\stem}{stem}

\DeclareMathOperator{\suc}{succ}

\newcommand{\nwd}{\mathrm{nwd}}
\newcommand{\conv}{\mathrm{conv}}

\newcommand{\covomn}{\operatorname{cov}_\omega(\mathcal{N})}

\DeclareMathOperator{\Spec}{Spec}

\renewcommand{\subset}{\subseteq}

\newcommand{\baire}{\omega^\omega}

\newcommand{\cantor}{2^{\omega}}
\newcommand{\cantortree}{2^{<\omega}}

\newcommand{\addo}[1]{\add_\omega(#1)}
\newcommand{\covo}[1]{\cov_\omega(#1)}

\newcommand{\addostar}[1]{\add_\omega^*(#1)}

\newcommand{\continuum}{\mathfrak{c}}

\newcommand{\Meager}{\mathcal{M}}
\newcommand{\Null}{\mathcal{N}}

\newcommand{\Baire}{\omega^{\omega}}

\newcommand{\RandomGraph}{\mathcal{R}}
\newcommand{\Solecki}{\mathcal{S}}

\newcommand{\cano}{(\cantortree)^\omega}
\newcommand{\bbmz}{\mathbb{M}_{\Z}}

\newcommand{\IL}{\I_{L}}
\newcommand{\ILf}{\I_{L}^f}

\newcommand{\KS}{\mathbb{K}_{\Solecki}}

\newcommand{\posetEDfin}{\mathbb{P}^{\edfin}}
\newcommand{\posetED}{\mathbb{P}^{\ed}}
\newcommand{\posetJL}{\mathbb{P}^{\J_L}}
\newcommand{\eec}{\ee^{\mathrm{const}}}
\newcommand{\eect}{\ee^{\mathrm{const}}(2)}
\newcommand{\eeck}{\ee^{\mathrm{const}}(k)}
\newcommand{\vv}{\mathfrak{v}}
\newcommand{\vvc}{\vv^{\mathrm{const}}}
\newcommand{\vvct}{\vv^{\mathrm{const}}(2)}
\newcommand{\vvck}{\vv^{\mathrm{const}}(k)}
\newcommand{\pc}{\sqsubset^{\mathrm{pc}}}
\newcommand{\pck}{\sqsubset^{\mathrm{pc},k}}

\newcommand{\Polynomial}{\mathscr{I}_P}
\newcommand{\Linear}{\mathscr{I}_L}

\newcommand{\ctbl}{[\mathbb{R}]^{\leq{\aleph_0}}}
\usepackage{caption}
\title{Cardinal invariants of idealized Miller null sets} 

\author{Aleksander Cieślak}
\address{Faculty of Pure and Applied Mathematics, Wrocław University of Science and Technology, Wybrzeże Stanisława Wyspiańskiego 27, 50-370 Wrocław, Poland}
\email{aleksander.cieslak@pwr.edu.pl}

\author{Takehiko Gappo}
\address{Institut f\"{u}r Diskrete Mathematik und Geometrie, TU Wien, Wiedner Hauptstra\ss e 8-10/104, 1040 Wien, Austria}
\email{takehiko.gappo@tuwien.ac.at}

\author{Arturo Martínez-Celis}
\address{Instytut Matematyczny, Uniwersytet Wrocławski, pl.\ Grunwaldzki 2, 50-384 Wrocław, Poland}
\email{arturo.martinez-celis@math.uni.wroc.pl}

\author{Takashi Yamazoe}
\address{Graduate School of System Informatics, Kobe University,
	Rokko--dai 1--1, Nada--ku, 657--8501 Kobe, Japan}
\email{212x502x@cloud.kobe-u.jp}

\begin{document}
	
    \begin{abstract}
        This paper provides an extensive study of the $\I$-Miller null ideals $M_\I$, $\sigma$-ideals on the Baire space parametrized by ideals $\I$ on countable sets. These $\sigma$-ideals are associated to the idealized versions of Miller forcing in the same way that the meager ideal is associated to Cohen forcing. We compute the cardinal invariants of $M_\I$ for typical examples of Borel ideals $\I$ and show that Cicho\'{n}'s Maximum can be extended by adding the uniformity and covering numbers of $M_\I$ for different ideals $\I$.
    \end{abstract}
	
    \maketitle

    \tableofcontents

    \section{Introduction}

    Cicho\'n's diagram (\Cref{figure:Cd_Intro}) illustrates the relationship of well-known cardinal invariants of the continuum, which we call \emph{Cicho\'n characteristics}.

    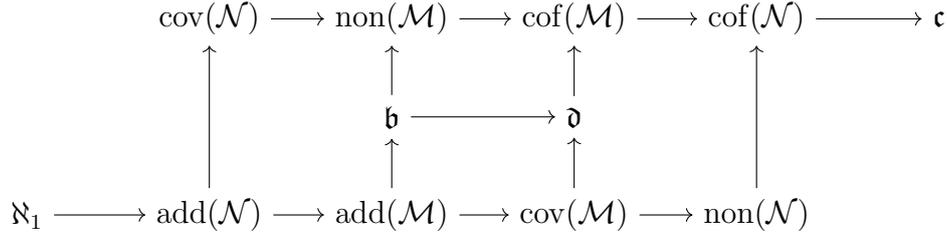
\begin{figure}[ht]	
		\centering
		\begin{tikzpicture}
			\tikzset{
				textnode/.style={text=black},
			}
			\tikzset{
				edge/.style={color=black,thin},
			}
			\newcommand{\w}{2.4}
			\newcommand{\h}{1.3}
			
			\node[textnode] (addN) at (0, 0) {$\addn$};
			\node[textnode] (covN) at (0, \h*2) {$\covn$};
			
			\node[textnode] (addM) at (\w, 0) {$\addm$};
			\node[textnode] (b) at (\w, \h) {$\bb$};
			\node[textnode] (nonM) at (\w, \h*2) {$\nonm$};
			
			\node[textnode] (covM) at (\w*2, 0) {$\covm$};
			\node[textnode] (d) at (\w*2, \h) {$\dd$};
			\node[textnode] (cofM) at (\w*2, \h*2) {$\cofm$};

			\node[textnode] (nonN) at (\w*3, 0) {$\nonn$};
			\node[textnode] (cofN) at (\w*3, \h*2) {$\cofn$};
			
			\node[textnode] (aleph1) at (-\w, 0) {$\aleph_1$};
			\node[textnode] (c) at (\w*4, \h*2) {$\cc$};
			
			\draw[->,edge] (addN) to (covN);
			\draw[->,edge] (addN) to (addM);
			\draw[->,edge] (covN) to (nonM);	
			\draw[->,edge] (addM) to (b);
			\draw[->,edge] (b) to (nonM);
			\draw[->,edge] (addM) to (covM);
			\draw[->,edge] (nonM) to (cofM);
			\draw[->,edge] (covM) to (d);
			\draw[->,edge] (d) to (cofM);
			\draw[->,edge] (b) to (d);
			\draw[->,edge] (covM) to (nonN);
			\draw[->,edge] (cofM) to (cofN);
			\draw[->,edge] (nonN) to (cofN);
			\draw[->,edge] (aleph1) to (addN);
			\draw[->,edge] (cofN) to (c);
		\end{tikzpicture}
		\caption{Cicho\'n's diagram. An arrow $\mathfrak{x}\to\mathfrak{y}$ denotes that $\mathfrak{x}\leq\mathfrak{y}$ holds.}\label{figure:Cd_Intro}
    \end{figure}

    All of them are cardinal invariants associated with $\sigma$-ideals on Polish spaces; eight of them are defined for the Lebesgue null ideal $\n$ or the meager ideal $\m$. The bounding number $\bb$ and the dominating number $\dd$ are related to the ideal $K_\sigma$ of $\sigma$-compact subsets of the Baire space:
    \[
    \bb=\add(K_\sigma)=\non(K_\sigma) \text{ and }\dd=\cov(K_\sigma)=\cof(K_\sigma).
    \]
    In addition, $\aleph_1$ and $\cc$ correspond to the $\sigma$-ideal $[\mathbb{R}]^{\leq{\aleph_0}}$ of countable sets of the reals:
    \[
    \aleph_1=\add(\ctbl)=\non(\ctbl)\text{ and }\cc=\cov(\ctbl)=\cof(\ctbl).
    \]
    Thus, Cicho\'n characteristics are the cardinal invariants of the four $\sigma$-ideals $\n, \m, K_\sigma$, and $\ctbl$, which we call \emph{Cicho\'n $\sigma$-ideals}.
    
    \emph{Cicho\'{n}'s maximum} is a maximal separation constellation of Cicho\'n's diagram, where all Cicho\'n characteristics have distinct values except for the two dependent numbers
    \[
    \addm=\min\{\bb,\covm\}\text{ and }\cofm=\max\{\dd,\nonm\}.
    \]
    Goldstern, Kellner and Shelah \cite{GKS19} constructed a model of Cicho\'{n}'s maximum assuming four strongly compact cardinals. Later, with Mej\'{\i}a \cite{GKMS22}, the large cardinal assumption was eliminated.
    From the perspective of $\sigma$-ideals, their result would mean that the four Cicho\'n $\sigma$-ideals are so different that their cardinal invariants can have distinct values at the same time.
    Moreover, it is possible to force Cicho\'n's maximum together with other cardinal invariants taking pairwise different values. Several such examples are known: Goldstern, Kellner, Mej{\'i}a, and Shelah added to Cicho\'n's maximum the cardinal invariants $\mm, \pp, \hh$ in \cite{GKMS21a} and $\mathfrak{s},\mathfrak{r}$ in \cite{GKMS21b}. The fourth author forced Cicho\'n's maximum with the evasion and prediction numbers in \cite{Yam25}. Cardona, Repick{\'y}, and Shelah \cite{CRS25} forced Cicho\'n's maximum together with the constant evasion and constant prediction numbers. The fourth author forced Cicho\'n's maximum together with the covering and uniformity number of the closed null ideal $\mathcal{E}$ in \cite{Yam24}.

    Our main aim is to systematically extend Cicho\'n's maximum to include cardinal invariants associated with other $\sigma$-ideals. In this paper, we focus on the following family of $\sigma$-ideals on Polish spaces parametrized with ideals on countable sets:

    \begin{dfn}
        Let $\I$ be an ideal on a countable set $X$. We define the \emph{$\I$-Miller null ideal} $M_\I$ as the $\sigma$-ideal on $X^\omega$ generated by sets of the form
        \[
        M_\phi = \{x\in X^\omega:\forall^{\infty}n<\omega~x(n)\in\phi(x\upharpoonright n)\},
        \]
        where $\phi\colon X^{<\omega}\to\I$.
    \end{dfn}
    
    \begin{dfn}
        Let $\I$ be an ideal on a coutable set $X$. We define $K_\I$ as the $\sigma$-ideal on $X^\omega$ generated by sets of the form
        \[
        K_\phi = \{x\in X^\omega:\forall^{\infty}n<\omega~x(n)\in\phi(n)\},
        \]
        where $\phi\colon\omega\to\I$.
    \end{dfn}
    
    Through the study of cardinal invariants of $M_\I$ and $K_\I$ for various Borel ideals $\I$, we will see how much these $\sigma$-ideals $M_\I$ and $K_\I$ are similar to/different from Cicho\'n $\sigma$-ideals.
    It turns out that for some ideals $\I$, the cardinal invariants of $M_\I$ and $K_\I$ are totally characterized using Cicho\'n characteristics, while for some other ideals $\I$, it is possible to extend Cicho\'{n}'s maximum to include some of the cardinal invariants of $M_\I$ and $K_\I$.

    Recall that the poset of all nonmeager Borel sets is forcing equivalent to Cohen forcing and that the poset of all Borel sets with a positive Lebesgue measure  is forcing equivalent to Random forcing.
    The ideal $M_\I$ is related to some natural tree forcing in the same way: Given an ideal $\I$ on a countable set $X$, we say that a tree $T\subset X^{<\omega}$ is an \emph{$\I$-Miller tree} if for every $\sigma\in T$ there exists $\tau\in T$ with $\sigma\subset\tau$ and
    \[
    \mathrm{succ}_T(\tau):=\{i\in X: \tau^{\frown}\langle i\rangle\in T\}\in\I^+.
    \]
    The \emph{$\I$-Miller forcing}, denoted by $\mathbb{M}_{\I}$, is the forcing poset of all $\I$-Miller trees ordered by inclusion. Note that $\mathbb{M}_{\Fin}=\mathbb{M}$ is the standard Miller forcing, where $\Fin=[\omega]^{<\omega}$ denotes the Fr\'{e}chet ideal. The forcings $\mathbb{M}_\I$ were studied by Sabok and Zapletal \cite{SZ11}, where they showed the relationship between $\I$-Miller trees and the $\I$-Miller null ideal.
    
    \begin{lem}[{Sabok--Zapletal \cite{SZ11}}]
        Let $X$ be a countable set. For every analytic subset $A\subset X^\omega$, either $A \in M_\I$ or there is $T\in\mathbb{M}_\I$ such that $[T]\subset A$.
    \end{lem}
    
    This lemma justifies our terminology. It also implies that the poset of all $M_{\I}$-positive Borel subsets of $\omega^\omega$ is forcing equivalent to $\mathbb{M}_{\I}$. The ideal $K_\I$ does not seem to correspond to any standard forcing notion. We introduce a tree forcing $\mathbb{K}_\I$ that is forcing equivalent to the poset of all $K_\I$-positive Borel sets, and this forcing notion seems new. This is why we did not give a particular name to $K_\I$; the letter $K$ was chosen as it is a direct generalization of the ideal $K_\sigma$. 
    It appears harder to study forcing properties of $\mathbb{K}_\I$ compared to $\mathbb{M}_\I$, and this fact sometimes prevents us to compute the exact value of $K_\I$, even when we could get the exact value of $M_\I$.

    A systematic study of the cardinal invariants of $M_\I$ was initiated in \cite{CM25} by the first and third authors.
    However, the cardinal invariants of $M_\I$ and $K_\I$ implicitly appeared in previous literature in different contexts. Forcing results in the aforementioned paper by Sabok--Zapletal \cite{SZ11} have direct corollaries on cardinal invariants of $M_\I$, e.g.\
    \[
    \non(M_\nwd) = \non(\m)\text{ and }\non(M_\Solecki)\geq\cov(\n),
    \]
    where $\nwd$ is the nowhere dense ideal and $\Solecki$ is the Solecki ideal (cf.\ \Cref{cor:Mnwd_nonm}, \Cref{cor:nonMSolecki}). In \cite{Paw00}, Pawlikowski studied the ideal $K_\Z$, where $\Z$ is the asymptotic density zero ideal on $\omega$, and proved that a poset adding a perfect set of random reals forces $\omega^\omega\cap V \in K_\Z$. Another result in the paper immediately implies that
    \[
    \non(K_\Z)\leq\max\{\bb, \non(\mathcal{E})\} \text{ and }\cov(K_\Z)\geq\min\{\dd, \cov(\mathcal{E})\},
    \]
    where $\mathcal{E}$ denotes the $\sigma$-ideal on $2^\omega$ generated by closed null sets (cf.\ \Cref{cor:paw00_ineq}). Also, $\non(K_\I)$ and $\cov(K_\I)$ appear as one of the slalom numbers in \cite{Sup23} and \cite{CGMRS24}. According to the notation in \cite{CGMRS24}, we have
    \[
    \non(K_\I) = \mathfrak{sl}_\mathrm{t}^{\bot}(\I, \mathrm{Fin})\text{ and }\cov(K_\I) = \mathfrak{sl}_\mathrm{t}(\I, \mathrm{Fin}).
    \]
    Slalom numbers are related to topological selection principles. For example, in \cite{Sup23}, \v{S}upina showed that $\cov(K_\I)$ is the least size of non-$S_1(\I\mathchar`-\Gamma, \mathcal{O})$ Hausdorff spaces. He also proved
    \[
    \cov(K_\I) \geq \min\{\dd, \cov^*(\I)\}.
    \]
    Moreover, one can see uniformity and covering numbers of $M_\I$ and $K_\I$ as idealized version of evasion/prediction numbers. Blass studied many variants of the evasion number $\ee$ in \cite{Bla10} and, according to his terminologies, $\non(M_\I)$ is the evasion number for global adaptive predictions with values in $\I$, while $\non(K_\I)$ is the evasion number for global non-adaptive predictions with values in $\I$.
    
    The first step to compute the cardinal invariants of the $\sigma$-ideals $M_\I$ and $K_\I$ is to clarify their connections with those of the original ideal $\I$. Such a connection was already noticed in \cite{CM25} and we refine their results to get the following.

    \begin{teorema}[\Cref{thm:exact_values_of_add}, \Cref{lem:b_nonM}, \Cref{thm:nonmi_leq_max}] \label{mainthm:basic_ineq}
        Let $\I$ be an ideal on a countable set. Then the following hold:
        \begin{enumerate}
            \item $\add(M_\I) = \add(K_\I) = \min\{\bb, \add^*_\omega(\I)\}$.
            \item\label{item:thmA_two} $\bb\leq\non(K_\I)\leq\non(M_\I)\leq\max\{\bb, \non^*_\omega(\I)\}$.
            \item $\min\{\dd, \cov^*(\I)\}\leq \cov(M_\I)\leq\cov(K_\I)\leq\dd$.
            \item $\cof(M_\I) = \cof(K_\I) = \max\{\dd, \cof^*_\omega(\I)\}$.
        \end{enumerate}
    \end{teorema}
    
    Here, $\add^*_\omega(\I), \non^*_\omega(\I)$ and $\cof^*_\omega(\I)$ are the ``$\omega$-versions'' of $\add^*(\I), \non^*(\I)$ and $\cof^*(\I)$. See \Cref{dfn:omega_star_numbers}. These cardinal invariants (for ultrafilters) were introduced by Brendle--Shelah \cite{BS99} to compute the cardinal invariants of null ideals of Laver and Mathias type forcings associated to ultrafilters. While the usual $*$-numbers have been extensively studied as surveyed in \cite{HH07}, these $\omega$-versions have been ignored for a long time. The only papers discussing the $\omega$-versions for Borel ideals are \cite{CM25} and \cite{FK25}.\footnote{In \cite{FK25}, they studied $\add^*_\omega(\I)$, but not $\non^*_\omega(\I)$ and $\cof^*_\omega(\I)$.} In this work, we continue this line of research, and compute additional values of these $\omega$-versions of $*$-numbers for even more examples of Borel ideals, see \Cref{table_values}.

    \begin{table}[t]
		\centering
		\begin{tabular}{c|ccccc}
			\hline
			ideal &  $\add^*_\omega(\I)$  & $\non^*_\omega(\I)$  & $\cof^*_\omega(\I)$ & $\non(\ki)$ & $\non(\mzbi)$\\
			
			\hline
			
			$\mathcal{R}$   & $\omega_1$  & $\omega_1$ &   $\cc$ & $\bb$ & $\bb$\\
			$\mathcal{S}$   & $\omega_1$  & $\covomn$ &   $\cc$ & $\max\{\bb,\covomn\}$ & $\max\{\bb,\covomn\}$\\
			$\nwd$   & $\addm$  & $\nonm$ &   $\cofm$ & $?$ & $\nonm$\\
			$\conv$   & $\omega_1$  & $\omega_1$ &    $\cc$ & $\bb$ & $\bb$\\
			$\finfin$   & $\bb$  & $\dd$ &   $\dd$ & $\bb$ & $\ee^{\mathrm{const}}_\leq(2)$\\
			$\ed$   & $\omega_1$  & $\covm$  & $\cc$ & $\bb$ & $\max\{\bb, \ee^{\mathrm{const}}(2)\} \leq \mathord{?}$\\
			$\edfin$   & $\omega_1$  & $?$  & $\cc$ & $?$ & $\max\{\bb, \ee^{\mathrm{const}}_b(2)\}\leq\mathord{?}$\\
            $\I_L$ & $\omega_1$ & $?$ & $\cc$ & $?$ & $?$\\
            $\I_P$ & $ \add_t(\n)\leq\mathord{?}$ & $?$ & $\mathfrak{l}\leq \mathord{?}$ & $?$ & $?$ \\
			$\Z$   & $\addn$  & $\non^*(\Z)$ &   $\cofn$ & $?$ & $\max\{\bb,\none\}$\\
			\hline 
		\end{tabular}
        \caption{Values of cardinal invariants associated with Borel ideals.}\label{table_values} 
	\end{table}

    The $\omega$-versions of $*$-numbers can be seen as improved versions of the original ones. Recall that the original $*$-numbers do not necessarily behave well for non-P-ideals: $\add^*(\I) = \omega$ for every non-P-ideal $\I$ and $\non^*(\I) = \omega$ for every Borel ideal $\I$ that is not Kat\v{e}tov reducible to $\edfin$. While $\add^*_\omega(\I) = \add^*(\I)$ and $\non^*_\omega(\I) = \non^*(\I)$ hold for all P-ideals $\I$, the $\omega$-versions take nontrivial values for non-P-ideals. For example, we show that $\add^*_\omega(\I)\geq\add_t(\n)$ for some $F_\sigma$ non-P-ideal $\I$, where $\add_t(\n)$ denotes the transitive additivity of the null ideal (\Cref{prop:addstaromega_IP}, \Cref{thm:addstaromega_grad_frag}). As a corollary, we get the consistency of $\add^*_\omega(\I)>\dd$ for such an ideal $\I$. This result should be compared to the fact that there are only three known values of $\add^*(\I)$ so far, which are $\omega, \add(\n)$ and $\bb$. We also note that the consistency of $\cof^*_\omega(\I)<\bb$ for some $F_\sigma$ non-P-ideal $\I$ was essentially shown in \cite{HRRZ14}. Even though $\add^*_\omega(\I)$ and $\cof^*_\omega(\I)$ do not have dual definitions, our proof looks like dual to the argument in \cite{HRRZ14}. So, by considering $\omega$-versions of $*$-numbers, we can naturally extend duality between additivity and cofinality even for non-P-ideals. A similar duality also appears between $\non^*_\omega(\I)$ and $\cov^*(\I)$ in our computation of $\non^*_\omega(\Solecki)$ and $\non^*_\omega(\ed)$.

    While the additivity and cofinality of $M_\I$ and $K_\I$ can be computed from $\add^*_\omega(\I)$ and $\cof^*_\omega(\I)$, the uniformity and covering numbers of $M_\I$ and $K_\I$ have very complicated patterns. Our $\mathsf{ZFC}$-provable results on $\non(M_\I)$ and $\non(K_\I)$ are summarized in \Cref{table_values}, but let us highlight some of them by stating here:
    
    \begin{teorema}\label{mainthm:comp_non_MI}\leavevmode
        \begin{enumerate}
            \item \textnormal{(\Cref{prop:non_M_finfin,prop:K_finfin_b})} $\non(M_{\finfin}) = \ee^{\mathrm{const}}_{\leq}(2)$ and $\cov(M_{\finfin}) = \vv^{\mathrm{const}}_{\leq}(2)$.
            On the other hand, $\non(K_{\finfin}) = \bb$ and $\cov(K_{\finfin}) = \dd$.
            \item \textnormal{(\Cref{thm:M_S_exact}, \Cref{thm:K_S_exact})} $\non(M_\Solecki) = \non(K_\Solecki) = \max\{\bb, \cov_\omega(\n)\}$ and $\cov(M_\Solecki) = \cov(K_\Solecki) = \min\{\dd, \non(\n)\}$.
            \item \textnormal{(\Cref{thm:nonMZ_exact_value})} $\non(M_\Z) = \max\{\bb, \non(\mathcal{E})\}$ and $\cov(M_\Z) = \min\{\dd, \cov(\mathcal{E})\}$.
        \end{enumerate}
    \end{teorema}
    
    In \Cref{mainthm:comp_non_MI}(1), $\ee^{\mathrm{const}}_{\leq}(2)$ and $\vv^{\mathrm{const}}_{\leq}(2)$ denote variants of constant evasion and prediction numbers introduced by Kamo \cite{Kam00}. Such variations are considered in \cite{CR25}. It is known that $\ee^{\mathrm{const}}_{\leq}(2)$ and $\vv^{\mathrm{const}}_{\leq}(2)$ are consistently different from $\bb$ and $\dd$, respectively. Thus, $\finfin$ is an example of ideals $\I$ such that $\non(M_\I)$ and $\non(K_\I)$ are consistently different. Since we also know that $\non^*_\omega(\finfin) = \dd$, $\finfin$ is also an example of ideals $\I$ such that $\non(M_\I) < \max\{\bb, \non^*(\I)\}$ is consistent, which should be compared to \Cref{mainthm:basic_ineq}\eqref{item:thmA_two}. In \Cref{mainthm:comp_non_MI}(2), $\cov_\omega(\n)$ denotes the least size of a family $\mathcal{F}$ of Lebesgue null sets such that any countable set $A$ of reals is included by some $N\in\mathcal{F}$, which turns out to be equal to $\non^*_\omega(\Solecki)$. So, $\mathcal{S}$ is an example of ideals $\I$ such that $\non(M_\I) = \non(K_\I) = \non\{\bb, \non^*_\omega(\I)\}$ is provable. Moreover, combining \Cref{mainthm:comp_non_MI}(3) with results of Raghavan \cite{Rag20}, we get new bounds for $\non^*(\Z)$ and $\cov^*(\Z)$:
    \[
    \non(\mathcal{E})\leq\non^*(\Z), \cov^*(\Z)\leq\cov(\mathcal{E}) \text{ and }\cov^*(\Z)\leq\non^*(\Z).
    \]
    See \cref{thm:new_bounds}. It is still unknown whether $\non(K_\Z) = \max\{\bb, \non(\mathcal{E})\}$ and $\cov(K_\Z) = \min\{\dd, \cov(\mathcal{E})\}$ hold or not.

    We also get numerous consistency results that are not direct consequences of $\mathsf{ZFC}$-provable inequalities. In a similar way as \Cref{mainthm:comp_non_MI}(1), we show that
    \[
    \max\{\bb, \eect\}\leq\non(M_{\ed})\text{ and }\max\{\bb,\eec_b(2)\}\leq\non(M_{\edfin})
    \]
    for any increasing function $b\in\oo$ (\Cref{prop:eect_nonMed}, \Cref{prop:edfin_const}). The converses of these inequalities are not $\mathsf{ZFC}$-provable by the following results.
    
    \begin{teorema}\label{mainthm:sep_into_two_values}
        The following (and their duals) consistently hold:
        \begin{enumerate}
            \item \textnormal{(\Cref{thm:Con_max_b_eect_nonMED})} $\max\{\bb,\eect\}<\non(M_{\ed})$.
            \item \textnormal{(\Cref{thm:Con_nonMfinfin_nonKEDfin})} $\non(M_{\finfin})<\non(K_{\edfin})$.
            \item \textnormal{(\Cref{thm:Con_max_nonKEDfin})} $\max\{\bb,\eec_b(2)\}<\non(K_{\edfin})$ for any increasing $b\in\oo$.
        \end{enumerate}
    \end{teorema}
    
    These results suggest that $\non(M_{\ed})$ and $\non(M_{\edfin})$ cannot be characterized by constant evasion numbers, unlike $\non(M_{\finfin})=\eec_{\leq}(2)$.
    The main method to prove (1) is the \emph{Fr\'{e}chet-limit (Fr-limit)} method, introduced by Mej\'{\i}a \cite{Mej19}, which preserves $\bb$ small through the forcing iteration. To show (2), we introduce a variant of this method, called \emph{closed-Fr-limit} method, which preserves $\non(M_{\finfin})$ small. (3) is obtained by combining the arguments of (1) and (2).
    
    In addition to the separations into two values in \Cref{mainthm:sep_into_two_values}, we could construct a model of extended Cicho\'n's maximum with uniformity and covering numbers of $M_\I$ and $K_\I$ for some ideal $\I$. Consequently, we encounter the following ideals, which are closely related to the linear growth ideal $\I_L$ (cf.\ {\cite[Page 56]{Hru11}, \cite[Page 3]{BM14}}):
    \[
    \J_L \coloneqq \{A\subset\omega\times\omega: \exists k<\omega~\forall^{\infty} n<\omega~\lvert\{m<\omega\colon\langle n, m\rangle\in A\}\rvert \leq k\cdot i\},
    \]
    and for $f\in\oo$ with $\textstyle\lim_{n\to\infty}f(n)/2^n=0$,
    \[
    \IL^f\coloneqq\{A\subset\omega\colon\exists k<\omega~\forall^\infty n<\omega~\lvert A\cap [2^n-1, 2^{n+1}-1)\rvert\leq k\cdot f(n)\}.
    \]
    These ideals provide the classes of ideals $\I_0, \I_1$ such that cardinal invariants of $M_{\I_0}$ and $K_{\I_1}$ can be added to a model of Cicho\'n's maximum with distinct values: 

    \begin{teorema}[\Cref{thm:CM}] \label{mainthm:ext_CM}
        Let $\aleph_1\leq\theta_1\leq\cdots\leq\theta_{12}$ be regular cardinals and $\theta_\cc$ an infinite cardinal with $\theta_\cc\geq\theta_{12}$ and $\theta_\cc^{\aleph_0}=\theta_\cc$. Then there exists a ccc poset that forces the separation constellation described in \Cref{fig:ext_CM}.
        Moreover, in the forcing extension, for ideals $\I_0, \I_1$ on countable sets such that $\J_L \leq_{\mathrm{K}} \I_0 \leq_{\mathrm{K}} \finfin$ and $\I_L^f \leq_{\mathrm{K}} \I_1 \leq_{\mathrm{K}} \Z$ for some $f\in\omega^\omega$ such that $\textstyle\lim_{n\to\infty}f(n)=\infty$ and $\textstyle\lim_{n\to\infty}f(n)/2^n=0$, we have:
        \[\non(M_{\I_0})=\theta_4,~\cov(M_{\I_0})=\theta_9,\]
        \[\non(K_{\I_1})=\non(M_{\I_1})=\theta_5,~\cov(K_{\I_1})=\cov(M_{\I_1})=\theta_8.\]
    \end{teorema}
    
    This result implies a substantial difference of the $\sigma$-ideals $M_{\I_0}$ and $K_{\I_1}$ for such ideals $\I_0,\I_1$ from Cicho\'n $\sigma$-ideals.
    The framework to construct a model of the extended Cicho\'n's maximum in \Cref{mainthm:ext_CM} is based on the fourth author's work \cite{Yam25}, \cite{Yam24}. To construct such a model, we need to use a stronger method called \emph{Ultrafilter-limit (UF-limit)} method than the Fr-limit method, introduced by Goldstern--Mej{\'\i}a--Shelah \cite{GMS16}, to control $\bb$ in a finer way. We also need the \emph{closed-UF-limit} method, introduced by the fourth author \cite{Yam24}, to control $\non(M_{\finfin})$. According to these necessities, we find the ideals $\J_L$ and $\I^L_f$ and the classes of ideals for $\I_0,\I_1$ in \Cref{mainthm:ext_CM}.
    These classes of ideals seem interesting in themselves. It turns out $\J_L$ and $\IL^f$ are critical for some selective properties. The class of ideals for $\I_0$ is considered in \cite{DFGT21} and every tall non-pathological analytic P-ideal belongs to the class of ideals for $\I_1$. See \Cref{sec:more_growth}.

    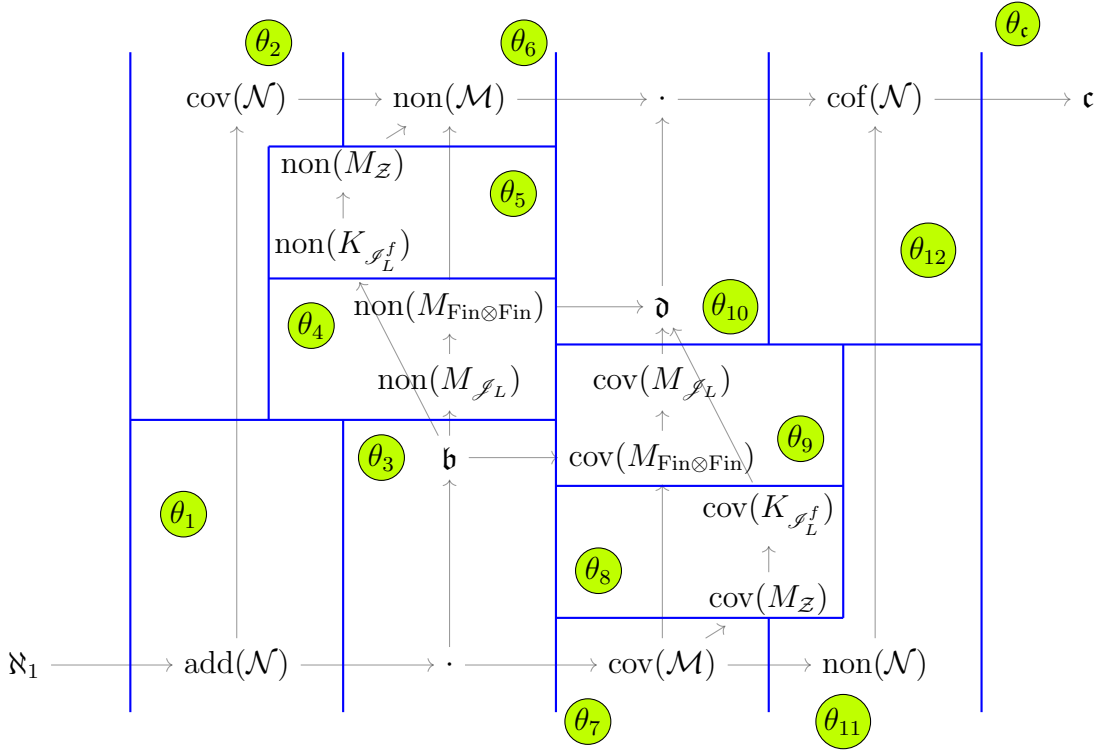
\begin{figure}[h!]
        \centering	
		\begin{tikzpicture}
			\tikzset{
				textnode/.style={text=black}, 
			}
			\tikzset{
				edge/.style={color=black, thin, opacity=0.4}, 
			}
			\newcommand{\w}{2.8}
			\newcommand{\h}{2.5}
			
			\node[textnode] (addN) at (0,  0) {$\addn$};
			\node (t1) [fill=lime, draw, text=black, circle,inner sep=1.0pt] at (-0.25*\w, 0.8*\h) {$\theta_1$};
			
			\node[textnode] (covN) at (0,  \h*3) {$\covn$};
			\node (t2) [fill=lime, draw, text=black, circle,inner sep=1.0pt] at (0.15*\w, 3.3*\h) {$\theta_2$};

			\node[textnode] (addM) at (\w,  0) {$\cdot$};
			
			\node (t3) [fill=lime, draw, text=black, circle,inner sep=1.0pt] at (0.68*\w, 1.1*\h) {$\theta_3$};
			
			\node[textnode] (nonM) at (\w,  \h*3) {$\nonm$};
			\node (t6) [fill=lime, draw, text=black, circle,inner sep=1.0pt] at (1.35*\w, 3.3*\h) {$\theta_6$};
			
			\node[textnode] (covM) at (\w*2,  0) {$\covm$};
			\node (t7) [fill=lime, draw, text=black, circle,inner sep=1.0pt] at (1.65*\w, -0.3*\h) {$\theta_7$};
			
			\node[textnode] (d) at (\w*2,  1.9*\h) {$\dd$};
			\node (t10) [fill=lime, draw, text=black, circle,inner sep=1.0pt] at (2.32*\w, 1.9*\h) {$\theta_{10}$};
			\node[textnode] (cofM) at (\w*2,  \h*3) {$\cdot$};

			\node[textnode] (nonN) at (\w*3,  0) {$\nonn$};
			\node (t11) [fill=lime, draw, text=black, circle,inner sep=1.0pt] at (2.85*\w, -0.3*\h) {$\theta_{11}$};
			
			\node[textnode] (cofN) at (\w*3,  \h*3) {$\cofn$};
			\node (t12) [fill=lime, draw, text=black, circle,inner sep=1.0pt] at (3.25*\w, 2.2*\h) {$\theta_{12}$};
			
			\node[textnode] (aleph1) at (-\w,  0) {$\aleph_1$};
			\node[textnode] (c) at (\w*4,  \h*3) {$\continuum$};
			\node (t10) [fill=lime, draw, text=black, circle,inner sep=1.0pt] at (3.67*\w, 3.4*\h) {$\theta_\cc$};

			\node (t4) [fill=lime, draw, text=black, circle,inner sep=1.0pt] at (0.35*\w, 1.8*\h) {$\theta_4$};

			\node (t9) [fill=lime, draw, text=black, circle,inner sep=1.0pt] at (2.65*\w, 1.2*\h) {$\theta_9$};

			\node (t5) [fill=lime, draw, text=black, circle,inner sep=1.0pt] at (1.3*\w, 2.5*\h) {$\theta_5$};

			\node (t8) [fill=lime, draw, text=black, circle,inner sep=1.0pt] at (1.7*\w, 0.5*\h) {$\theta_8$};

            \node[textnode] (b) at (\w,  1.1*\h) {$\bb$};

            \node[textnode] (nonMFinFin) at (\w,  1.9*\h) {$\non(M_{\finfin})$}; 
            \node[textnode] (covMFinFin) at (2*\w,  1.1*\h) {$\cov(M_{\finfin})$};

            \node[textnode] (nonMZ) at (\w*0.5,  \h*2.65) {$\non(M_{\Z})$};
            
			\node[textnode] (covMZ) at (\w*2.5,  \h*0.35) {$\cov(M_{\Z})$};
            \node[textnode] (nonMILf) at (\w*0.5,  \h*2.2) {$\non(K_{\I_L^f})$};
            \node[textnode] (covMILf) at (\w*2.5,  \h*0.8) {$\cov(K_{\I_L^f})$};
			
			\node[textnode] (nonMJL) at (\w,  1.5*\h) {$\non(M_{\J_L})$};
            \node[textnode] (covMJL) at (2.0*\w,  1.5*\h) {$\cov(M_{\J_L})$};

			\draw[->, edge] (addN) to (covN);
			\draw[->, edge] (addN) to (addM);
			\draw[->, edge] (covN) to (nonM);	
			\draw[->, edge] (addM) to (b);
			\draw[->, edge] (nonMFinFin) to (nonM);
			\draw[->, edge] (addM) to (covM);
			\draw[->, edge] (nonM) to (cofM);
			\draw[->, edge] (covM) to (covMFinFin);
			\draw[->, edge] (d) to (cofM);

			\draw[->, edge] (covM) to (nonN);
			\draw[->, edge] (cofM) to (cofN);
			\draw[->, edge] (nonN) to (cofN);
			\draw[->, edge] (aleph1) to (addN);
			\draw[->, edge] (cofN) to (c);

			\draw[->, edge] (nonMFinFin) to (d);
			\draw[->, edge] (nonMJL) to (nonMFinFin);

			\draw[->, edge] (b) to (covMFinFin);
			\draw[->, edge] (covMFinFin) to (covMJL);
			
			\draw[->, edge] (nonMZ) to (nonM);
			
			\draw[->, edge] (covM) to (covMZ);

			\draw[->, edge] (b) to (nonMILf);
			\draw[->, edge] (nonMILf) to (nonMZ);

            \draw[->, edge] (b) to (nonMJL);

            \draw[->, edge] (covMILf) to (d);
			\draw[->, edge] (covMZ) to (covMILf);

            \draw[->, edge] (covMJL) to (d);

			\draw[blue,thick] (-0.5*\w,1.3*\h)--(1.5*\w,1.3*\h);
			\draw[blue,thick] (3.5*\w,1.7*\h)--(1.5*\w,1.7*\h);
			\draw[blue,thick] (1.5*\w,-0.25*\h)--(1.5*\w,3.25*\h);
			
			\draw[blue,thick] (-0.5*\w,-0.25*\h)--(-0.5*\w,3.25*\h);
			\draw[blue,thick] (3.5*\w,-0.25*\h)--(3.5*\w,3.25*\h);
			
			\draw[blue,thick] (0.5*\w,-0.25*\h)--(0.5*\w,1.3*\h);
			\draw[blue,thick] (2.5*\w,1.7*\h)--(2.5*\w,3.25*\h);
			
			\draw[blue,thick] (0.15*\w,2.05*\h)--(1.5*\w,2.05*\h);
			\draw[blue,thick] (0.15*\w,2.75*\h)--(1.5*\w,2.75*\h);
			\draw[blue,thick] (0.15*\w,1.3*\h)--(0.15*\w,2.75*\h);
			\draw[blue,thick] (0.5*\w,2.75*\h)--(0.5*\w,3.25*\h);
			
			\draw[blue,thick] (2.85*\w,0.95*\h)--(1.5*\w,0.95*\h);
			\draw[blue,thick] (2.85*\w,0.25*\h)--(1.5*\w,0.25*\h);
			\draw[blue,thick] (2.85*\w,1.7*\h)--(2.85*\w,0.25*\h);
			\draw[blue,thick] (2.5*\w,0.25*\h)--(2.5*\w,-0.25*\h);
		\end{tikzpicture}
        \caption{Extended Cicho\'{n}'s Maximum.   $f$ represents any function $f\in\oo$ such that $\textstyle\lim_{n\to\infty}f(n)=\infty$ and  $\textstyle\lim_{n\to\infty}f(n)/2^n=0$.}\label{fig:ext_CM}
    \end{figure}

    \subsection*{Acknowledgments}
    
    The second author's research was funded by the Austrian Science Fund (FWF) [10.55776/I6087, 10.55776/ESP5842424]. He thanks Piotr Borodulin-Nadzieja for inviting him to Wrocław and for supporting his stay.
    The fourth author thanks his supervisor J\"{o}rg Brendle for his support throughout this work. He also thanks Osvaldo Guzm\'{a}n and Yurii Khomskii, who formulated and proved \Cref{lem:quasi}. He was supported by JSPS KAKENHI Grant Number JP25KJ1818.
    For the purpose of open access, the authors have applied a CC BY public copyright license to any Author Accepted Manuscript version arising from this submission.

    \section{Preliminaries}\label{sec:Preliminaries}

    Given a set $X$, $\I\subset\mathcal{P}(X)$ is an \emph{ideal} if it is closed under taking finite unions and subsets. An ideal $\I$ is called a \emph{$\sigma$-ideal} if it is closed under taking countable unions. We always assume that an ideal $\I$ on $X$ is \emph{proper} in the sense that $X\notin\I$ and $\I$ contains every finite subset of $X$. An ideal $\I$ on $X$ is \emph{tall} if every infinite subset of $X$ contains an infinite set in $\I$.
    
    Let $\I$ be an ideal on $X$. We say that $A\subset X$ is \emph{$\I$-positive} if $A\notin \I$. We write $\I^+ = \mathcal{P}(X)\setminus\I$. The \emph{dual filter} of $\I$ is defined by $\I^* \coloneq \{A\subset X: X\setminus A\in\I\}$.
    
    The following are well-known relations among ideals and partial orders.
    
    \begin{dfn}
        Let $\I, \J$ be ideals on $X, Y$ respectively.
        \begin{enumerate}
            \item $\I$ is \emph{Kat\v{e}tov reducible} to $\J$, denoted by $\I\leq_\mathrm{K}\J$, if there is a function $f\colon Y\to X$ such that for all $A\in\I$, $f^{-1}[A]\in\J$.
            \item $\I$ is \emph{Kat\v{e}tov--Blass reducible} to $\J$, denoted by $\I\leq_{\mathrm{KB}}\J$, if there is a finite-to-one function $f\colon Y\to X$ such that for all $A\in\I$, $f^{-1}[A]\in\J$.
        \end{enumerate}
    \end{dfn}
    
    \begin{dfn}
        Let $\langle P, \leq\rangle$ be a partially ordered set. A subset $X\subset P$ is \emph{bounded} if there is $a\in P$ such that $\forall x\in X\,(x\leq a)$ holds. We also say that $X$ is \emph{$\sigma$-bounded} if there is a countable subset $A\subset P$ such that $\forall x\in X\, \exists a\in A\,(x\leq a)$ holds.
    
        Let $P, Q$ be partially ordered sets.
        \begin{enumerate}
            \item $P$ is \emph{Tukey reducible} to $Q$, denoted by $P\preceq_T Q$, if there is $f\colon P\to Q$ such that for any unbounded subset $X\subset P$, $f[X]$ is a unbounded subset of $Q$.
            \item $P$ is \emph{$\omega$-Tukey reducible} to $Q$, denoted by $P\preceq_{\omega T} Q$, if there is $f\colon P\to Q$ such that for any $\sigma$-unbounded subset $X\subset P$, $f[X]$ is a $\sigma$-unbounded subset of $Q$.
        \end{enumerate}
        Trivially, $P\preceq_T Q$ implies $P\preceq_{\omega T} Q$.
    \end{dfn}
    
    For any countable set $X$, its power set $\mathcal{P}(X)$, equipped with the product topology, is a Polish space homeomorphic to $2^\omega$. Thus, it makes sense to classify an ideal $\I \subseteq \mathcal{P}(X)$ by its complexity, i.e. $F_\sigma$, Borel, or analytic. We now list the definitions of the Borel ideals that appear in \Cref{table_values}.
    \begin{itemize}
        \item The random graph ideal $\mathcal{R}$ is the ideal on $\omega$ generated by the homogeneous sets in Rado's random graph.
        \item The Solecki ideal $\mathcal{S}$ is the ideal on the countable set
        \[
        \Omega\coloneqq\{U\in\operatorname{Clopen}(2^\omega):U\text{ has Lebesgue measure }1/2\}
        \]
        generated by subsets $A\subseteq\Omega$ with non-empty intersection.
        \item The nowhere dense ideal, denoted by $\nwd$, is the ideal on the rational numbers $\mathbb{Q}\cong2^{<\omega}$ of nowhere dense subsets of $2^{<\omega}$.
        \item The convergent ideal, denoted by $\conv$, is the ideal on $\mathbb{Q}\cap[0, 1]$ generated by convergent sequences in $\mathbb{Q}\cap [0,1]$.
        \item $\Fin\otimes\Fin$ is the Fubini product of two $\Fin$'s, where $\Fin$ is the ideal of finite sets. In general, for ideals $\I,\J$ on $X, Y$ respectively, their Fubini product $\I\otimes\J$ is the ideal on $X\times Y$ defined by
        \[
        \I\otimes\J = \{A\subset X\times Y:\{x\in X:(A)_x\notin\J\}\in\I\},
        \]
        where $(A)_x \coloneqq \{y\in Y:\langle x, y\rangle\in A\}$.
        \item The eventually different ideal is defined by
        \[
        \ed\coloneqq\{A\subseteq\omega\times\omega:\exists k<\omega~\forall^\infty n<\omega~|(A)_n|\leq k\}.
        \]
        Also, we define $\edfin\coloneqq\ed \on \Delta$, where $\Delta\coloneqq\{(n,m)\in\omega\times\omega:m\leq n\}$.
        \item  Fix the interval partition $\bar{P}^\mathrm{exp}=(P_i^\mathrm{exp})_{i<\omega}$ of $\omega$ such that $\lvert P^\mathrm{exp}_i\rvert = 2^i$. 
        The linear growth ideal $\I_L$ is defined by
        \[
        \I_L = \{A\subset\omega: \exists k<\omega\, \forall^\infty i <\omega\, (\lvert A \cap P^\mathrm{exp}_i\rvert \leq k\cdot i)\}
        \]
        The polynomial growth ideal $\I_P$ is
        \[
        \I_P = \{A\subset\omega: \exists k<\omega\, \forall^\infty i <\omega\, (\lvert A \cap P^\mathrm{exp}_i\rvert \leq i^k)\}
        \]
        \item The asymptotic density zero ideal $\Z$ is defined by
        \[
        \Z = \left\{A\subseteq\omega:\lim_{n\to\infty}\frac{\lvert A\cap n\rvert}{n}=0\right\}.
        \]
    \end{itemize}
    
    \Cref{figure:Hrusaks_heart} illustrates the Kat\v{e}tov-Blass orders among the Borel ideals listed above (including $\Fin$). The details can be consulted in \cite{Hru17} ($\edfin\leq_{\mathrm{KB}}\I_L$ is easy to see and the KB-reductions $\I_L\leq_{\mathrm{KB}}\I_P\leq_{\mathrm{KB}}\Z$ follow from the inclusions $\I_L\subseteq\I_P\subseteq\Z$).
	\begin{figure}[h]
		\centering
		\begin{tikzpicture}[]
			\matrix[matrix of math nodes,column sep={50pt,between origins},row
			sep={30pt,between origins},nodes={asymmetrical rectangle}] 
			{
				|[name=nwd]| \nwd &&|[name=finfin]| \finfin &&|[name=Z]| \Z &\\
				&&&&&|[name=Pol]| \Polynomial \\
				&&&&&|[name=Lin]| \Linear \\
				&&&&&|[name=edfin]| \edfin \\
				|[name=sol]| \Solecki &&|[name=conv]| \conv &&|[name=ed]| \ed  \\
				\\
				&&|[name=rand]| \RandomGraph \\
				&&|[name=fin]| \Fin \\
			};
			\draw[->] 
			(conv) edge (finfin)
			(ed) edge (finfin)
			(Pol) edge (Z)
			(conv) edge (nwd)
			(conv) edge (Z)
			(rand) edge (ed)
			(ed) edge (edfin)
			(edfin) edge (Lin)
			(Lin) edge (Pol)
			(rand) edge (conv)
			(fin) edge (rand)
			(fin) edge (sol)
			(sol) edge (nwd);
			
		\end{tikzpicture}
		\caption{Diagram of Borel ideals and Kat\v{e}tov-Blass orders. An arrow $\I\to\J$ denotes that $\I\leq_{\mathrm{KB}}\J$ holds.}\label{figure:Hrusaks_heart}
	\end{figure}
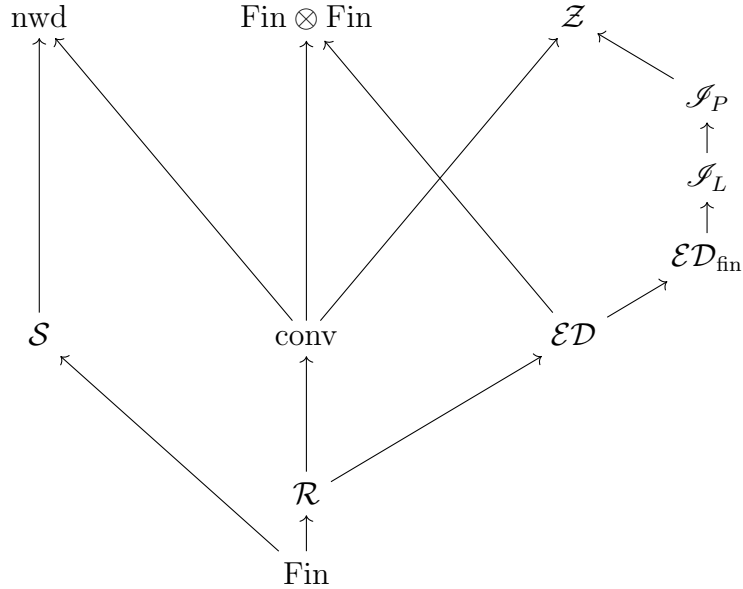
    
    Recall that an ideal $\I$ is called a \emph{P-ideal} if for any $\overline{A}\in[\I]^\omega$, there is an $B\in\I$ such that $\forall A\in\overline{A}\,(A\subset^* B)$. The classes of $F_\sigma$ ideals and analytic P-ideals (on countable sets) are especially important subclasses of Borel ideals, so let us state relevant definitions and basic results:
    
    \begin{dfn}
        A function $\varphi\colon\mathcal{P}(\omega)\to\left[0,\infty\right]$ is called a \emph{submeasure} on $\omega$ if:
    	\begin{itemize}
    		\item $\varphi(\emptyset)=0$,
    		\item $A\subseteq B\Rightarrow\varphi(A)\leq\varphi(B)$,
    		\item $\varphi(A\cup B)\leq\varphi(A)+\varphi(B)$.
    	\end{itemize}
    	A submeasure $\varphi$ is \emph{lower semi-continuous} if $\varphi(A)=\displaystyle\lim_{n\to\infty}\varphi(A\cap n)$ for all $A\subseteq\omega$.
    \end{dfn}
    
    \begin{fac}\label{fac:submeasure_chara}
        Let $\I$ be an ideal on $\omega$.
        \begin{itemize}
            \item \textnormal{(Mazur \cite{Maz91})} $\I$ is an $F_\sigma$ ideal if and only if there is a lower semi-continuous submeasure $\varphi$ such that
            \[
            \I=\Fin(\varphi)\coloneqq\{A\subseteq\omega:\varphi(A)<\infty\}.
            \]
            \item \textnormal{(Solecki \cite{Sol99})} $\I$ is an analytic P-ideal if and only if there is a lower semi-continuous submeasure $\varphi$ such that
            \[
            \I=\Exh(\varphi)\coloneqq\{A\subseteq\omega:\lim_{n\to\infty}\varphi(A\setminus n)=0\}.
            \]
        \end{itemize}
    \end{fac}
    
    The following is a natural subclass of $F_\sigma$ non-P-ideals, introduced in \cite{HRRZ14}.
    
    \begin{dfn}
        An ideal $\I$ on $\omega$ is \emph{fragmented} if there is a partition $\langle P_n\rangle_{n<\omega}$ of $\omega$ into nonempty finite sets and a sequence $\langle\varphi_n\rangle_{n<\omega}$ such that each $\varphi_n$ is a submeasure on $P_n$ such that
        \[
        X\in\I \iff \sup_{n<\omega}\varphi_n(X\cap P_n) < \infty.
        \]
        A fragmented ideal associated to $\langle P_n, \varphi_n\rangle_{n<\omega}$ is \emph{gradually fragmented} if there is a function $f\colon\omega\to\omega$ such that for all $n<\omega$, the following holds:
        \[
        \forall i<\omega\,\forall^{\infty} j<\omega\,\forall \overline{B}\in[\mathcal{P}(P_j)]^{\leq i}\,\left(\forall B\in\overline{B}\,(\varphi_j(B)\leq n)\Rightarrow\varphi_j\left(\bigcup\overline{B}\right)\leq f(n)\right).
        \]
    \end{dfn}

    Let $\I$ be a fragmented ideal on $\omega$ associated to $\langle P_n, \varphi_n\rangle_{n<\omega}$. Then a function $\varphi\colon\mathcal{P}(\omega)\to[0, \infty]$ defined by $\varphi(A) = \sup_{n<\omega}\varphi_n(A\cap P_n)$ is a lower semi-continuous submeasure on $\omega$ such that $\I = \Fin(\varphi)$. By \cref{fac:submeasure_chara} (1), $\I$ is $F_\sigma$. It was shown in \cite[Corollary 2.9]{BM14} that any tall (proper) fragmented ideal is not a P-ideal. One can find more information about fragmented ideals in \cite{BM14}.

    Next, we introduce the standard cardinal invariants associated to ideals.

    \begin{dfn}
        Let $\mathcal{I}$ be a $\sigma$-ideal on a set $X$. Then we define
        \begin{align*}
            \add(\mathcal{I}) &= \min\{\lvert\mathcal{A}\rvert:\mathcal{A}\subset\mathcal{I}\land\bigcup\mathcal{A}\notin\mathcal{I}\}\\
            \non(\mathcal{I}) &= \min\{\lvert A\rvert: A\subset X\land A\notin \mathcal{I}\}\\
            \cov(\mathcal{I}) &= \min\{\lvert\mathcal{A}\rvert:\mathcal{A}\subset\mathcal{I}\land\bigcup\mathcal{A} = X\}\\
            \cof(\mathcal{I}) &= \min\{\lvert\mathcal{A}\rvert:\mathcal{A}\subset\mathcal{I}\land\forall B\in\mathcal{I}\,\exists A\in\mathcal{A}\,(B\subset A)\}
        \end{align*}
    \end{dfn}

    \begin{dfn}\label{dfn:star_numbers}
        Let $\I$ be an ideal on a countable set $X$. Then we define
        \begin{align*}
            \add^*(\I) &= \min\{\lvert\mathcal{A}\rvert: \mathcal{A}\subset\I\land\forall B\in\I\,\exists A\in\mathcal{A}\,(A \not\subset^* B)\}\\
            \non^*(\I) &= \min\{\lvert\mathcal{A}\rvert: \mathcal{A}\subset[X]^\omega\land\forall B\in\I\,\exists A\in\mathcal{A}\,(\lvert A\cap B \rvert<\omega)\}\\
            \cov^*(\I) &= \min\{\lvert\mathcal{A}\rvert: \mathcal{A}\subset\I\land\forall B\in [X]^{\omega}\,\exists A\in\mathcal{A}\,(\lvert A\cap B\rvert = \omega)\}\\
            \cof^*(\I) &= \min\{\lvert\mathcal{A}\rvert: \mathcal{A}\subset\I\land\forall B\in\I\,\exists A\in\mathcal{A}\,(B \subset^* A)\}
        \end{align*}
        If $\I$ is not tall, $\cov^*(\I)$ is ill-defined. In this case, we set $\cov^*(\I) = \infty$ for convenience.\footnote{Recall that our implicit assumption that an ideal is always proper. This guarantees that $\add^*(\I)$ and $\non^*(\I)$ are well-defined. $\cof^*(\I)$ is well-defined even without properness.}
    \end{dfn}

    The notation in \cref{dfn:star_numbers} come from \cite{HH07}.
    Note that if $\I$ is tall, then $\cov^*(\I)$ and $\cof^*(\I)$ are always uncountable, but $\add^*(\I)$ and $\non^*(\I)$ may not. There is an easy criteria for $\add^*(\I)$ and $\non^*(\I)$ to be countable as follows.

    \begin{lem}\label{lem:addstar_nonstar}
        Let $\I$ be a tall ideal on a countable set. Then the following hold:
        \begin{enumerate}
            \item $\add^*(\I) > \omega \iff \I$ is a P-ideal.
            \item \textnormal{(\cite{HMM10})} If $\I$ is Borel, then $\non^*(\I) > \omega \iff \I\geq_{\mathrm{KB}}\edfin$.
        \end{enumerate}
    \end{lem}

    We recall the role of idealized forcings for computing cardinal invariants. Let $\mathcal{I}$ be a $\sigma$-ideal on a Polish space $X$. The poset of $\mathcal{I}$-positive Borel sets ordered by inclusion is denoted by $\mathbb{P}_{\mathcal{I}}$. The forcing $\mathbb{P}_\mathcal{I}$ adds a \emph{canonical real} $\dot{r}^{\mathcal{I}}_{\text{gen}}$ such that, for every Borel $B \subseteq X$ coded in the ground model, $B$ is in the generic filter if and only if $\dot{r}^{\mathcal{I}}_{\text{gen}} \in B$ \cite[Proposition 2.1.2]{Zapl}. 
    
    When computing $\non(\mathcal{I})$ and $\cov(\mathcal{I})$, it is natural to look at the forcing $\mathbb{P}_\mathcal{I}$ because of the following connection between the cardinal invariants and the generic reals added by $\mathbb{P}_\mathcal{I}$.
    
    \begin{dfn}[{\cite[Definition 2.3.1]{Kho11}}]
        Let $\mathcal{I}$ be a $\sigma$-ideal on a Polish space $X$ and $M$ a transitive model of set theory. A real $x$ is called \emph{$\mathcal{I}$-quasi-generic} over $M$ if for every Borel set $B\in\mathcal{I}$ whose Borel code lies in $M$, $x\notin B$.
    \end{dfn}

    \begin{lem}\label{lem:quasi}
        Let $\mathcal{I}, \mathcal{J}$ be $\sigma$-ideals on Polish spaces $X$ and $Y$ generated by Borel sets. Assume that there is a Borel function $f\colon Y\to X$ such that
        \[
        \mathbb{P}_\mathcal{J}\Vdash f(\dot{r}^\mathcal{J}_\mathrm{gen})\text{ is }\mathcal{I}\text{-quasi-generic}.
        \]
        Then, for any $A\in\mathcal{I}$, $f^{-1}[A]\in\mathcal{J}$ holds (in other words, $f$ witnesses $\mathcal{I} \leq_{\mathrm{K}} \mathcal{J}$). In particular, $\non(\mathcal{I})\leq\non(\mathcal{J})$ and $\cov(\mathcal{I})\geq\cov(\mathcal{J})$ hold.
    \end{lem}
    \begin{proof}
        Let $A\in\mathcal{I}$ and we may assume $A$ is Borel. Assume $f^{-1}[A]\notin\mathcal{J}$. Then $f^{-1}[A]$ is Borel so $f^{-1}[A]\in \mathbb{P}_\mathcal{J}$. $f^{-1}[A]$ forces $\dot{r}^\mathcal{J}_\mathrm{gen}\in f^{-1}[A]$, and hence $ f(\dot{r}^\mathcal{J}_\mathrm{gen})\in A$, which contradicts that $f(\dot{r}^\mathcal{J}_\mathrm{gen})$ is forced to be $\mathcal{I}$-quasi-generic. 
    \end{proof}

    \section{$\omega$-versions of cardinal invariants of $\I$}\label{sec:omega}

    To compute the cardinal invariants of $M_\I$, the following ``$\omega$-versions'' of the cardinal invariants in \cref{dfn:star_numbers} will be useful.

    \begin{dfn}\label{dfn:omega_star_numbers}
        Let $\I$ be an ideal on a countable set $X$. Then we define
        \begin{align*}
            \add^*_\omega(\I) &= \min\{\lvert\mathcal{A}\rvert:\mathcal{A}\subset\I\land\forall\overline{B}\in[\I]^\omega\,\exists A\in\mathcal{A}\,\forall B\in\overline{B}\, (A \not\subset^* B)\}\\
            \non^*_\omega(\I) &= \min\{\lvert\mathcal{A}\rvert:\mathcal{A}\subset[X]^\omega\land\forall\overline{B}\in[\I]^\omega\,\exists A\in\mathcal{A}\,\forall B\in\overline{B}\, (\lvert A \cap B \rvert<\omega)\}\\
            \cof^*_\omega(\I) &= \min\{\lvert\mathcal{A}\rvert:\mathcal{A}\subset[\I]^\omega\land\forall\overline{B}\in[\I]^\omega\,\exists\overline{A}\in\mathcal{A}\,\forall B\in\overline{B}\,\exists A\in\overline{A}\, (B \subset^* A)\}
        \end{align*}
        If $\I$ is countably generated, then $\add^*_\omega(\I)$ is ill-defined. In this case, we set $\add^*_\omega(\I) = \infty$ for convenience.\footnote{Properness of $\I$ guarantees that $\non^*_\omega(\I)$ is well-defined. Even without properness, $\cof^*_\omega(\I)$ is well-defined.}
    \end{dfn}
    
    One might think that we could naturally define the $\omega$-version of $\cov^*(\I)$ by
    \[
    \cov^*_\omega(\I) = \min\{\lvert\mathcal{A}\rvert:\mathcal{A}\subset\I\land\forall\overline{B}\in[[X]^\omega]^\omega\,\exists A\in\mathcal{A}\,\forall B\in\overline{B}\,(\lvert A\cap B\rvert = \omega)\},
    \]
    but obviously $\cov^*_\omega(\I) = \cov^*(\I)$ holds.

    Except for $\cof^*_\omega(\I)$, the notions in \cref{dfn:omega_star_numbers} come from \cite{CM25}, but the cardinal invariants themselves were introduced in \cite{BS99} for ultrafilters. Using the notation in \cite{BS99}, we have
    \[
    \mathfrak{p}'(\mathcal{U}) = \add^*_\omega(\mathcal{U}^*), \pi\chi_\sigma(\mathcal{U}) = \non^*_\omega(\mathcal{U}^*), \chi_\sigma(\mathcal{U}) = \cof^*_\omega(\mathcal{U}^*)
    \]
    for any ultrafilter $\mathcal{U}$ on $\omega$. In \cite{LV99}, $\add_\omega(P)$ was defined for partial orders $P$ and $\add^*_\omega(\I) = \add_\omega(\langle\mathcal{I}, \subset^*\rangle)$ holds. Also, $\add^*_\omega(\I)$ was studied in \cite{FK25}, but denoted by $\add_\omega(\I)$.

    It is natural to ask when these $\omega$-versions of $*$-numbers are different from the original $*$-numbers. The following observations cover the case of $P$-ideals.

    \begin{lem}
        Let $\I$ be a tall ideal on a countable set. Then the following hold:
        \begin{enumerate}
            \item $\omega_1\leq\add^*_\omega(\I)\leq\non^*_\omega(\I)\leq\cof^*_\omega(\I)$.
            \item Let $\mathrm{inv}$ denote either $\add, \non$, or $\cov$. Then $\mathrm{inv}^*(\I) \leq \mathrm{inv}^*_\omega(\I)$ and if $\I$ is a P-ideal, then $\mathrm{inv}^*(\I) = \mathrm{inv}^*_\omega(\I)$.
        \end{enumerate}
    \end{lem}
    
    This lemma implies that $\add^*(\I) = \add^*_\omega(\I) \iff \I$ is a P-ideal (If $\add^*(\I) = \add^*_\omega(\I)$ then $\add^*(\I)$ is uncountable by (1) and thus $\I$ has to be a P-ideal. The other direction is (2)).
    By \cref{lem:addstar_nonstar}(2), there are many tall ideals $\I$ such that $\omega = \non^*(\I) < \non^*_\omega(\I)$ holds (in $\mathsf{ZFC}$). Also, it is not hard to find a Borel ideal $\I$ such that $\omega_1\leq \non^*(\I) < \non^*_\omega(\I)$ is consistent (\cref{cor:nonstar_vs_nonstar_omega}).
    In addition, it is consistent that $\cof^*(\I) < \cof^*_\omega(\I)$ for some ideal $\I$ (see \Cref{rem:aleph_omega_Cohen}).

    \begin{lem}\label{lem:reductions_and_invstaromega}
        Let $\I, \J$ be tall ideals on $\omega$.
        \begin{enumerate}
            \item If $\I \preceq_{\omega T} \J$ then $\add^*_\omega(\I)\geq\add^*_\omega(\J)$ and $\cof^*_\omega(\I)\leq\cof^*_\omega(\J)$ hold.
            \item \textnormal{(\cite[Proposition 4.6]{CM25})} $\I \leq_{\mathrm{KB}} \J$ implies $\non^*_\omega(\I) \leq \non^*_\omega(\J)$. 
        \end{enumerate}
    \end{lem}
    \begin{proof}
        To show (1), let $f\colon\I\to\J$ be a witness of $\I \preceq_{\omega T} \J$. 
        
        To see $\add^*_\omega(\I)\geq\add^*_\omega(\J)$, let $\mathcal{A}\subset\I$ be such that $\lvert\mathcal{A}\rvert < \add^*_\omega(\J)$. Then $f[\mathcal{A}]$ is $\sigma$-bounded in $\J$ and thus $\mathcal{A}$ is $\sigma$-bounded in $\I$. So, $\lvert\mathcal{A}\rvert < \add^*_\omega(\I)$.
    
        To see $\cof^*_\omega(\I)\leq\cof^*_\omega(\J)$, let $\mathcal{A}\subseteq[\J]^\omega$ witness $|\mathcal{A}|=\cof^*_\omega(\J)$. 
        We may assume:
        \[\forall\overline{X}\in[\J]^\omega\exists\overline{A}\in\mathcal{A}\forall X\in\overline{X}\exists A\in\overline{A}\, (X \subset A).\]
        For each $\overline{A}\in\mathcal{A}$, there is $\mathcal{A}_{\overline{A}}\in[\I]^\omega$ witnessing $\overline{I}_{\overline{A}}\coloneqq\{X:f(X)\subseteq A\text{ for some }A\in\overline{A}\}\subseteq\I$ is $\sigma$-bounded, since $\overline{A}\subseteq\J$ is $\sigma$-bounded. To see $\{\overline{I}_{\overline{A}}:A\in\mathcal{A}\}$ witnesses $\cof^*_\omega(\J)\leq|\mathcal{A}|$, let $\overline{X}\in[\I]^\omega$ be arbitrary. Take $\overline{A}\in\mathcal{A}$ such that for all $X\in\overline{X}$ there is $A\in\overline{A}$ such that $f(X) \subseteq A$. Thus $X\in\overline{I}_{\overline{A}}$ and hence there is $A^\prime\in\mathcal{A}_{\overline{A}}$ such that $X\subseteq A^\prime$, which finishes the proof.
    \end{proof}

    \subsection{Computation of $\add^*_\omega(\I)$ and $\cof^*_\omega(\I)$}
    
    We compute $\add^*_\omega(\I)$ for various Borel non-P ideals $\I$. According to \cite{Hru11} (the paragraph before Question 6.21), there are only three known distinct values of $\add^*(\I)$ for analytic ideals $\I$, which are $\omega$, $\add(\mathcal{N})$ and $\bb$. In contrast, we will see that $\add^*_\omega(\I)$ can take values other than $\omega_1, \add(\mathcal{N})$ and $\bb$. Indeed, $\add^*_\omega(\nwd) = \add(\m)$ and $\add^*_\omega(\I) > \bb$ is consistent for some $F_\sigma$ non-P ideal $\I$. In \cite{CM25} and \cite{FK25}, the following was already shown.

    \begin{prop}[Cieślak--Martínez-Celis \cite{CM25} and Filip\'{o}w--Kwela \cite{FK25}\footnotemark]\label{prop:addstar_omega_in_CM25} \leavevmode
        \begin{enumerate}
            \item $\add^*_\omega(\finfin)=\bb$.
            \item $\add^*_\omega(\I) = \omega_1$ for $\I = \Solecki, \RandomGraph, \mathrm{conv}, \ed, \edfin$.
        \end{enumerate}
    \end{prop}
    \footnotetext{All the proofs can be found in Section 4 of \cite{CM25}. In \cite{FK25}, only $\add^*_\omega(\finfin)$ and $\add^*_\omega(\Solecki)$ are computed as Example 5.12 and Theorem 5.13(7).}

    Theorem 5.11(3) of \cite{FK25} can be improved as follows:
    
	\begin{prop}\label{prop:addstart_omega_of_Fubini_prod}
	   	For any ideals $\I, \J$ on countable sets,
	 	\[
	  	\add^*_\omega(\I\otimes\J) = \min\{\bb, \add^*_\omega(\I), \add^*_\omega(\J)\}.
	   	\]
	\end{prop}
	\begin{proof}
        We may assume that the ideals $\I$ and $\J$ are both on $\omega$.
        
		($\leq$) The inequality $\add^*_\omega(\I\otimes\J) \leq \min \{ \add^*_\omega(\I), \add^*_\omega(\J)\}$ was shown in \cite{FK25} as Theorem 5.11(3), so it suffices to prove that $\add^*_\omega(\I\otimes\J) \leq \bb$.
            
        Let $\{f_\alpha : \alpha < \bb\} \subseteq \baire$ be an unbounded family of strictly increasing functions and let
        \[
        A_\alpha = \{ \langle i,j \rangle\in\omega\times\omega : j \leq f_\alpha(i)\}.
        \]
        Clearly $A_\alpha \in \finfin \subset \I\otimes\J$. We claim that $\{A_\alpha : \alpha<\bb\}$ is a witness for $\add^*_\omega(\I\otimes\J)$. To show this, let $\{B_n: n<\omega\} \subseteq \I\otimes\J$. For each $n<\omega$, there is a function $g_n\in\omega^\omega$ such that for all $m<\omega$,
        \[
        (\omega \setminus m \times g_n(m))\setminus B_n \neq\emptyset,
        \]
        since otherwise $(\omega\setminus m)\times\omega\subset B_n$ for some $m<\omega$, which contradicts $B_n\in\I\otimes\J$. We also choose $\alpha<\bb$ such that $f_\alpha \not\leq^* g_n$ for all $n<\omega$. Now we show that $A_\alpha\not\subset^* B_n$ for all $n<\omega$: For each $n<\omega$, there is $m<\omega$ such that $f_\alpha(m) > g_n(m)$. By the choice of $g_n$, we can find $\langle i, j \rangle \in \omega\times\omega$ such that $i\geq m, j\leq g_n(m)$, and $\langle i, j\rangle\notin B_n$. Since
        \[
        f_\alpha(i) \geq f_\alpha(m) > g_n(m) \geq j,
        \]
        so $\langle i, j\rangle \in A_\alpha \setminus B_n$.
			
		($\geq$) Let $\kappa < \min\{\bb, \add^*_\omega(\I), \add^*_\omega(\J)\}$ and let $\mathcal{A}\coloneqq\{A_\alpha : \alpha<\kappa\} \subseteq \I\otimes\J$. We want to show that $\mathcal{A}$ is not a witness for $\add^*_\omega(\I\otimes\J)$.
        
        For each $\alpha<\kappa$, let $J_\alpha = \{n<\omega : (A_\alpha)_n \notin \I\}$. Clearly $J_\alpha \in \J$. Since $\kappa<\add^*_\omega(\I)$, there must be $\{I_n: n< \omega\}\subset\I$ such that for any $\alpha<\kappa$ and any $n \notin J_\alpha$, there is $m<\omega$ such that $(A_\alpha)_n \subseteq I_m$. We may assume that $I_n\subset I_{n+1}$ for all $n<\omega$. For each $\alpha<\kappa$, let $g_\alpha\in\omega^\omega$ be such that for all $n<\omega$,
        \[
        n\notin J_\alpha \Rightarrow (A_\alpha)_n \subseteq I_{g_\alpha(n)}.
        \]
        As $\kappa < \bb$, there is a function $f\in\omega^\omega$ such that $f \geq^* g_\alpha$ for all $\alpha<\kappa$. Moreover, since $\kappa < \add^*_\omega(\J)$, there is $\{J'_n :n<\omega\}\subset\J$ such that for every $\alpha < \kappa$, there is $n<\omega$ such that $J_\alpha \subseteq J'_n$. We may also assume that $n\subset J'_n\subset J'_{n+1}$ for all $n<\omega$.
        
        For each $n<\omega$, we define $B_n\subset\omega\times\omega$ by
        \[
        B_n = \{\langle i, j\rangle: i\in J'_n \lor j\in I_{f(i)}\}.
        \]
        Clearly $B_n \in \I \otimes \J$. We claim that for each $\alpha < \kappa$ there is $n<\omega$ such that $A_\alpha \subseteq B_n$, which implies $\mathcal{A}$ is not a witness for $\add^*_\omega(\I\otimes\J)$. To show this, fix $\alpha < \kappa$. Then we can find $n<\omega$ such that $f(m)\geq g_\alpha(m)$ for all $m\geq n$ and $J_\alpha\subset J'_n$. It is easy to see $A_\alpha\subset B_n$ as follows: let $\langle i, j\rangle\in A_\alpha$. If $i \in J'_m$, then $\langle i, j\rangle\in B_n$ by definition. Otherwise, since $i \notin J_\alpha$ and $i\geq n$, we have $j\in (A_\alpha)_i \subset I_{g_\alpha (i)} \subset I_{f(i)}$ and thus $\langle i, j\rangle\in B_n$.
	\end{proof}

    The following proposition was independently proved by Osvaldo Guzm\'an and Francisco Santiago Nieto de la Rosa. It gives an affirmative answer to \cite[Question 4.10]{CM25}.

    \begin{prop}[Independently, Guzm\'an--Nieto de la Rosa]\label{prop:nwd_addm}
        $\add^*_\omega(\nwd) = \addm$.
    \end{prop}
    \begin{proof}
        In \cite[Theorem 4.9]{CM25}, $\add^*_\omega(\nwd) \leq \addm$ was proved. To show the converse inequality, let $\kappa<\addm$ and assume that $\{A_\alpha:\alpha<\kappa\}\subset\nwd$. It suffices to find a countable family $\{D_k:k<\omega\}$ of open dense subsets of $2^{<\omega}$ such that for each $\alpha<\kappa$ there is $k<\omega$ such that $A_\alpha \cap D_k=\emptyset$. Since $\kappa<\addm$, there is a countable family $\{T_n:n<\omega\}$ of nowhere dense subtrees of $2^{<\omega}$ (i.e.\ for any $s\in 2^{<\omega}$, there is $t\supseteq s$ with $t\notin T_n$) such that $\bigcup_{\alpha<\kappa}[A_\alpha]\subset\bigcup_{n<\omega}[T_n]$. 

        Fix $t\in 2^{<\omega}$ for now. Since each $T_n$ is nowhere dense, we can inductively take $t\subsetneq s_0\subsetneq s_1\subsetneq\cdots\in 2^{<\omega}$ such that $s_n\notin T_n$. Let $x_t = \bigcup_{n<\omega}s_n\in2^\omega$. Since $x_t\notin  \bigcup_{n<\omega}[T_n]$, for each $\alpha<\kappa$ there is $m_{t,\alpha}<\omega$ such that $t\subset x_t\upharpoonright m_{t,\alpha}\notin A_\alpha$. Since $\kappa<\addm\leq\bb$, there is $m_t<\omega$ such that for all $\alpha<\kappa$ and for all but finitely many $t\in 2^{<\omega}$, $m_{t,\alpha}\leq m_t$.
        
        Fix some bijection $e\colon\omega\to 2^{<\omega}$ and for each $k<\omega$, let
        \[
        D_k = \{s\in 2^{<\omega}: \exists n\geq k\,(s\supseteq x_{e(n)}\upharpoonright m_{e(n)})\}.
        \]
        Clearly $D_k$ is open dense. For each $\alpha<\kappa$, there is $k<\omega$ such that for all $n\geq k$, $m_{e(n),\alpha}\leq m_{e(n)}$. To show $A_\alpha\cap D_k=\emptyset$, let $s \in D_k$ be arbitrary. Then $s\supseteq x_{e(n)}\upharpoonright m_{e(n)}$ for some $n\geq k$. Since $x_{e(n)}\upharpoonright m_{e(n)}\supseteq x_{e(n)}\upharpoonright m_{e(n),\alpha}\notin A_\alpha$, we have $s\notin A_\alpha$.
    \end{proof}

    Now we aim to find Borel ideals $\I$ such that $\add^*_\omega(\I) > \bb$ is consistent. By the following classical results, such $\I$ must be a $F_\sigma$ non-P-ideal.

    \begin{thm}[Todor\v{c}evi\'{c}, \cite{Tod96}]\label{thm:Todorcevic}
        For every tall analytic P-ideal $\I$,
        \[
        \add(\mathcal{N}) \leq \add^*(\I) \leq \bb\text{ and }\dd\leq\cof^*(\I) \leq \cof(\mathcal{N}).
        \]
    \end{thm}

    \begin{thm}[Louveau--Velickovic, \cite{LV99}]\label{thm:Louveau-Velickovic}
        Let $\I$ be a tall analytic ideal on $\omega$. Then either $\I$ is $F_\sigma$ or $(\omega^\omega, \leq)$ is Tukey reducible to $(\mathcal{I}, \subset)$.
    \end{thm}

    \begin{cor}
        If $\I$ is a tall non-$F_\sigma$ ideal on a countable set, then
        \[
        \add^*_\omega(\I) \leq \bb \text{ and } \cof^*(\I)\geq\dd.
        \]
    \end{cor}

    Even for $F_\sigma$ non-P ideals $\I$, it is often the case that $\add^*(\I) = \omega_1$, which follows from the following results.

    \begin{dfn}
        Given an ideal $\I$, a subset $X \subset \I$ of $\I$ is \emph{strongly unbounded} if for every countable $A \subset X$ we have that $\bigcup A \notin \I$.
    \end{dfn}

    One can easily show that $\mathscr{I}_L$ has an uncountable perfect strongly unbounded set.
    
    \begin{prop}\label{prop:strongly_unbounded}
        If an ideal $\I$ has a strongly unbounded uncountable subset, then $\add^*_\omega(\I) = \omega_1$.
    \end{prop}
    \begin{proof}
        Observe that, if $X$ is a strongly unbounded subset of size $\omega_1$, then every $I \in \I$ can only contain finitely many elements of $X$. So if $\{I_n :n\in\omega\} \subset \I$, then there is $A \in X$ not contained in any of the $I_n$'s, thus $X$ is a witness for $\addo{\I}$, i.e.\ $\add^*_\omega(\I) = \omega_1$.
    \end{proof}

    \begin{cor}
        $\add^*_\omega(\I_L) = \omega_1$.
    \end{cor}
    
    The ideal $\I_P$ is the typical example of a fragmented ideal that has no strongly unbounded subsets (see \cite{HRRZ14}). We will see that $\add^*_\omega(\I_P) > \bb$ is consistent. To prove this, we need to recall the following cardinal invariant introduced by Pawlikowski \cite{Paw85}.

    \begin{dfn}
        Let $\mathcal{I}$ be a $\sigma$-ideal on $2^\omega$. Note that $2^\omega$ can be regarded as a topological group with the coordinate-wise addition modulo 2. Then the transitive additivity of $\mathcal{I}$, denoted by $\add_t(\mathcal{I})$, is defined by
        \[
        \add_t(\mathcal{I}) = \min\{\lvert A\rvert: A\subset 2^\omega\land \exists X\in\mathcal{I}\,(A + X\notin\mathcal{I})\}.
        \]
    \end{dfn}

    We use Pawlikowski's characterization of $\add_t{(\mathcal{N})}$, where $\n$ is the null ideal on $2^\omega$.\footnote{There is a slight abuse of notation here, because $\n$ also denotes the null ideal on $\omega^\omega$ in this paper.}
    
    \begin{lem}[Pawlikowski, {\cite[Lemma 2.2]{Paw85}}]\label{lem:addtN}
		\[
		\add_t(\mathcal{N}) = \min \left\{\lvert\mathcal{A}\rvert : \exists g\in(\omega\setminus 1)^\omega\,\left(\mathcal{A}\subset\prod_{i<\omega}g(i) \land \forall S\in\prod_{i<\omega}[g(i)]^{\leq i}\,\exists f \in\mathcal{A}\,(f\not\in^* S)\right)\right\},
		\]
        where $f\in^* S$ means that $\forall^\infty i<\omega\,(f(i)\in S(i))$ holds.
	\end{lem}

    In \cite{Paw85}, Pawlikowski also showed that
    \[
    \add(\mathcal{N}) = \min\{\bb, \add_t(\mathcal{N})\}
    \]
    holds and that $\add(\mathcal{N}) = \bb < \add_t(\mathcal{N})$ is consistent. By Shelah \cite{She92} (see also \cite[Section 3.4, Theorem 2]{Bre95}), it is also known that $\dd<\add_t(\mathcal{N})$ is consistent. One can find other results on $\add_t(\mathcal{N})$ in \cite[Section 2.7]{BJ95} and more recently in \cite{CMRM25}.

    \begin{prop} \label{prop:addstaromega_IP}
		$\add^*_\omega(\I_P) \geq \add_t(\mathcal{N})$.
	\end{prop}
	\begin{proof}
		Assume that $\kappa < \add_t(\mathcal{N})$ and let $\mathcal{A} \subseteq \I_P$ be of size $\kappa$. It suffices to find $\{B_n:n<\omega\}\subset\I_P$ such that for every $A \in \mathcal{A}$ there is $n\in\omega$ such that $A \subseteq^* B_n$. For each $n<\omega$, let
        \[
        \mathcal{A}_n = \{ A \in \mathcal{A} : \forall i<\omega\,(\lvert A\cap P_i^\mathrm{exp}\rvert \leq i^n)\}.
        \]
        Clearly, $\mathcal{A} = \bigcup_{n \in \omega} \mathcal{A}_n$. For each $A \in \mathcal{A}_n$, we define a function $f^n_A\colon\omega\to[P_i^\mathrm{exp}]^{\leq i^n}$ by
        \[
        f^n_A (i) = A \cap P_i^\mathrm{exp}.
        \]
        By \cref{lem:addtN}, for each $n<\omega$, we can find a function $S_n$ on $\omega$ such that 
		\begin{enumerate}
			\item $\lvert S_n(i)\rvert \leq i$,
			\item $S_n(i) \subset [P_i^\mathrm{exp}]^{\leq i^n}$, and
			\item for every $A \in \mathcal{A}_n$, $\forall^\infty i<\omega\,(f^n_A(i) \in S_n(i))$ holds.
		\end{enumerate}
        Then we set $B_n = \bigcup_{i<\omega} \bigcup  S_n(i)$.
        By (1) and (2), we have
        \[
        \lvert B_n \cap P_i^\mathrm{exp}\rvert = \left\lvert\bigcup S_n(i)\right\rvert \leq i \cdot i^n = i^{n+1}
        \]
        and thus $B_n \in \I_P$. By (3), for any $A\in\mathcal{A}_n$, $A\subset^* B_n$ holds.
	\end{proof}

    The above proof can be generalized to show the following.

    \begin{thm} \label{thm:addstaromega_grad_frag}
        For any gradually fragmented ideal $\I$ on $\omega$, $\addostar{\I} \geq \add_t{(\Null)}$.
    \end{thm}
    \begin{proof}
        Assume that $\kappa < \add_t(\Null)$ and let $\mathcal{A} \subset \I_P$ of size $\kappa$. It suffices to find $\{B_n:n<\omega\}\subset\I_P$ such that for every $A \in \mathcal{A}$ there is $n\in\omega$ such that $A \subseteq^* B_n$.
        
        Assume that a partition $\langle P_i\rangle_{i<\omega}$ of $\omega$ and a sequence $\langle\varphi_i\rangle_{i<\omega}$ of submeasures on $P_i$ witness that $\I$ is fragmented. Let $\varphi$ be a lower semi-continuous submeasure on $\omega$ defined by $\varphi(A) = \sup\{\varphi_i(A\cap P_i):i<\omega\}$. Note that $\I = \mathrm{Fin}(\varphi)$. Also, let $f\in\omega^\omega$ be a function witnessing that $\I$ is gradually fragmented. We may assume that $f$ is increasing and $f(0)=0$. For each $n<\omega$, let $g_n\in\omega^\omega$ be such that for all $i<\omega$,
        \begin{equation}\label{eqn:grad_frag}
            \forall j\geq g_n(i)\,\forall\overline{B}\in[\mathcal{P}(P_j)]^{\leq i}\,\left(\forall B\in\overline{B}\,(\varphi_j(B)\leq n)\Rightarrow\varphi_j\left(\bigcup\overline{B}\right)\leq f(n)\right).
        \end{equation}
        For each $i, n<\omega$, let $P^n_i = \bigcup_{j\in[g_n(i), g_n(i+1))]}P_j$. For each $n<\omega$, let
        \[
        \mathcal{A}_n = \{ A \in \mathcal{A} : \varphi(A)\leq n \}.
        \]
        As $\I = \mathrm{Fin}(\varphi)$, we have $\mathcal{A} = \bigcup_{n<\omega}\mathcal{A}_n$. For each $A \in \mathcal{A}_n$, we define $f^n_A\colon\omega\to\mathcal{P}(P^n_i)$ by
        \[
        f^n_A (i) = A \cap P^n_i.
        \]
        By \cref{lem:addtN}, for each $n<\omega$, we can find a function $S_n$ on $\omega$ such that
		\begin{enumerate}
			\item $\lvert S_n(i)\rvert \leq i$,
			\item $S_n(i) \subset \mathcal{P}(P^n_i)$,
            \item for every $B\in S_n(i)$, $\varphi(B) \leq n$, and 
			\item for every $A\in\mathcal{A}_n$, $\forall^\infty i<\omega\,(f^n_A(i)\in S_n(i))$ holds.
		\end{enumerate}
        Then we set $B_n = \bigcup_{i<\omega}\bigcup S_n(i)$.
        Using $(*)$, it follows from (1), (2) and (3) that
        \[
        \varphi(B_n\cap P^n_i) = \varphi(\bigcup S_n(i)) \leq f(n)
        \]
        and thus $B_n\in\Fin(\varphi) = \I$. By (4), for any $A\in\mathcal{A}_n$, $A\subset^* B_n$.
   \end{proof}

    Next we consider $\cof^*_\omega(\I)$. One interesting phenomena is that we have several arguments for $\cof^*_\omega(\I)$ that can be seen as dual to the previous arguments for $\add^*_\omega(\I)$, even though the definitions of $\add^*_\omega(\I)$ and $\cof^*_\omega(\I)$ are not really dual.

    \begin{prop}\leavevmode
        \begin{enumerate}
            \item $\cof^*_\omega(\mathscr{I}) = \cc$ in the case that $\mathscr{I}$ is $\mathcal{R}, \mathcal{S}, \conv, \mathcal{ED}, \mathcal{ED}_{fin}, \I_L$.
            \item $\cof^*_\omega(\finfin) = \dd$.
            \item $\cof^*_\omega(\nwd) = \cof(\mathcal{M})$.
        \end{enumerate}
    \end{prop}
    \begin{proof}
        Recall that $\cof^*(\I)\leq\cof^*_\omega(\I)$ always holds.
        
        (1) For those ideals $\I$, it is known that $\cof^*(\I) = \cc$; For $\I=\mathcal{R}, \mathcal{S}, \conv, \mathcal{ED}$, see \cite[Section 3]{Hru11}. For $\I = \edfin, \I_L$, this follows from \cite[Theorem 2.6]{HRRZ14}, since $ \edfin, \I_L$ are fragmented but not gradually fragmented (cf.\ \cite{BM14}).
            
        (2) It is known that $\cof^*(\finfin) = \dd$ (cf.\ \cite{Hru11}), so it suffices to show that $\cof^*_\omega(\finfin)\leq\dd$. For any $n<\omega$ and $f\in\omega^\omega$, let
        \[
        A_{n, f} = (n\times\omega)\cup\{\langle i, j\rangle\in\omega\times\omega: j\leq f(i)\}.
        \]
        Now let $\mathcal{F}\subset\omega^\omega$ be a dominating family of size $\dd$. Then it is easy to see that $\mathcal{A} := \{\{A_{n, f}: n<\omega\}: f\in\mathcal{F}\}$ witnesses $\cof^*_\omega(\finfin)\leq\dd$.
            
        (3) It is known that $\cof^*(\nwd) = \cof(\Meager)$ (due to Fremlin \cite{Fr91}, but also see \cite[Theorem 1.6(i)]{BHH04}), so it suffices to show that $\cof^*_\omega(\nwd)\leq\cof(\Meager)$. Note that the following argument is a dual of \Cref{prop:nwd_addm}. Let $\{A_\alpha:\a<\cof(\m)\}\subset\m$ be cofinal in $\m$. We may assume $A_\alpha=\bigcup_{n<\omega}[T_{\alpha,n}]$, where $T_{\alpha,n}\subset 2^{<\omega}$ is a nowhere dense tree. 
        For each $t\in 2^{<\omega}$, take $x^\alpha_t \in 2^\omega$ such that $t\subset x^\alpha_t\notin\bigcup_{n<\omega}[T_{\alpha,n}]$. Let $\{f_\beta:\beta<\dd\}$ be a dominating family of functions from $2^{<\omega}$ to $\omega$. Fix some bijection $e\colon\omega\to 2^{<\omega}$ and for each $k<\omega$, let
        \[
        D_{\alpha, \beta, k} = \{s\in 2^{<\omega}: \exists n\geq k\,(s\supseteq x^\alpha_{e(n)}\upharpoonright f_\beta(e(n))\}.
        \]
        Clearly $D_{\alpha, \beta, k}$ is open dense. It is routine to see that $\{\{2^{<\omega}\setminus D_{\alpha,\beta,k}:k<\omega\}:\alpha<\cofm\land\beta<\dd\}$ witnesses $\cof^*_\omega(\nwd)\leq\max\{\cof(\m),\dd\}=\cofm$.
    \end{proof}

    We show that $\cof^*_\omega(\I) <\dd$ is consistent for some ideal $\I$ by proving the dual of \cref{prop:addstaromega_IP} and \cref{thm:addstaromega_grad_frag}. In \cite{Kad00}, Kada introduced the following cardinal invariant that is closely related to the Laver property.

    \begin{dfn}
        \[
        \mathfrak{l} = \min\{\kappa:\forall g\in\omega^\omega\,\exists\mathcal{S}\subset\prod_{i<\omega}[g(i)]^{\leq i}\,(\lvert\mathcal{S}\rvert = \kappa\land\forall f\leq^* g\,\exists S\in\mathcal{S}\,(f\in^* S))\}.
        \]
    \end{dfn}
    
    Note that $\mathfrak{l}$ is a dual notion to the combinatorial characterization of $\add_t(\mathcal{N})$ by Pawlikowski. In \cite{Kad00}, it was mentioned that $\cof(\mathcal{N}) = \min\{\dd, \mathfrak{l}\}$ and that in the Laver model, $\omega_1 = \mathfrak{l} < \bb = \dd = \omega_2$ holds.
    
    \begin{prop}\label{prop:cofstaromega_IP}
		$\cof^*_\omega(\I_P) \leq \mathfrak{l}$.
	\end{prop}	
	\begin{proof}
        For each $i<\omega$, let
        \[
        Q_i = \prod_{n\leq i} [P_i^\mathrm{exp}]^{\leq i^n}.
        \]
        Let $\{S_\alpha:\alpha<\mathfrak{l}\}\subset\prod_{i<\omega}[Q_i]^{\leq i}$ be such that every $f\in\prod_{i<\omega}Q_i$ is captured by some $S_\alpha$ in the sense that for all but finitely many $i<\omega$, $f(i)\in S_\alpha(i)$ holds. For each $\alpha<\mathfrak{l}$ and each $n<\omega$, we define
        \[
        A_\alpha^n = \bigcup_{i\geq n} \bigcup\pi_n[S_\alpha(i)],
        \]
        where $\pi_n$ is the $n$-th projection map, i.e.\ $\pi_n(f) = f(n)$ for every function $f$ with $n\in\dom(f)$. For each $n<\omega$, $A_\alpha^n \in \I_P$ because for all $i\geq n$,
        \[
        \lvert A_\alpha^n\cap P^{\mathrm{exp}}_i\rvert = \lvert\pi_n[S_\alpha(i)]\rvert \leq i\cdot i^n = i^{n+1}
        \]
        holds. Let $A_\alpha = \{A^n_\alpha:n<\omega\}$ for each $\alpha<\mathfrak{l}$ and we claim that $\{A_\alpha : \alpha<\mathfrak{l}\}$ witnesses $\cof^*_\omega(\I_P)$; Let $\{B_n:n \in \omega\}\subset\I_P$. By adding more elements to the sequence if necessary, we may assume that $\varphi(B_n) \leq n$. Let $f$ be a function on $\omega$ defined by
        \[
        f(i) = \langle B_0 \cap P_i^\mathrm{exp}, B_1 \cap P_i^\mathrm{exp}, \ldots, B_i \cap P_i^\mathrm{exp} \rangle.
        \]
        Since $f(i)\in Q_i$ for all but finitely many $i<\omega$, $f$ must be captured by some $S_\alpha$. This implies $B_n \subset^* A^n_\alpha$ for all $n<\omega$.
	\end{proof}

    More generally, the following holds.

    \begin{thm}\label{thm:cofstaromega_grad_frag}
        $\cof^*_\omega(\I)\leq\mathfrak{l}$ for every gradually fragmented ideal $\I$.
    \end{thm}
    \begin{proof}
        As in \cref{thm:addstaromega_grad_frag}, let $\langle P_i, \varphi_i\rangle_{i<\omega}$ witness that $\I$ is fragmented and  let $\varphi$ be a lower semi-continuous submeasure on $\omega$ defined by $\varphi(A) = \sup\{\varphi_i(A\cap P_i):i<\omega\}$. Also, let $f\in\omega^\omega$ be an increasing function witnessing that $\I$ is gradually fragmented and for each $n<\omega$, let $g_n\in\omega^\omega$ be such that for all $i<\omega$, (\ref{eqn:grad_frag}) holds. For each $i, n<\omega$, let $P^n_i = \bigcup_{j\in[g_n(i), g_n(i+1))]}P_j$. For each $i<\omega$, let
        \[
        Q_i = \prod_{n\leq i}\{A\subset P^n_i:\varphi(A)\leq n\}.
        \]
        Let $\{S_\alpha:\alpha<\mathfrak{l}\}\subset\prod_{i<\omega}[Q_i]^{\leq i}$ be such that every element of $\prod_{i<\omega}Q_i$ is captured by some $S_\alpha$. For each $\alpha<\mathfrak{l}$ and each $n<\omega$, we define
        \[
        A_\alpha^n = \bigcup_{i\geq n} \bigcup\pi_n[S_\alpha(i)],
        \]
        where $\pi_n$ is the $n$-th projection map. Using (\ref{eqn:grad_frag}), we have $A_\alpha^n \in \mathrm{Fin}(\varphi) = \I$. Let $A_\alpha = \{A^n_\alpha:n<\omega\}$ for each $\alpha<\mathfrak{l}$ and one can show that $\{A_\alpha : \alpha<\mathfrak{l}\}$ witnesses $\cof^*_\omega(\I)$ by the same argument for \cref{prop:cofstaromega_IP}.
    \end{proof}

    This result can be seen as a reformulation of the following ``$\omega$-version'' of \cite[Theorem 2.2 \& Proposition 2.3]{HRRZ14}:

    \begin{thm}\label{thm:grad_frag_in_Laver}
        Let $\I$ be a gradually fragmented ideal on $\omega$ and let $\mathbb{P}$ be a forcing with the Laver property. Then $[\I]^\omega\cap V$ witnesses $\cof^*_\omega(\I)$. Therefore, in the Laver model, for all gradually fragmented ideals $\I$ on $\omega$, $\cof^*_\omega(\I) < \bb = \dd$ holds.
    \end{thm}
    \begin{proof}
        We continue to use the notation introduced in the proof of \cref{thm:cofstaromega_grad_frag}. Let $\{\dot{X}_n: n<\omega\}$ be $\mathbb{P}$-names of a subset of $\omega$ and assume that $\forces_{\mathbb{P}}\varphi(\dot{X}_n)<\infty$. It suffices to find $\{A_n\colon n<\omega\}\in [\I]^\omega\cap V$ such that for all $n<\omega$, $\forces_{\mathbb{P}}\dot{X}_n\subset^* A_n$ holds.
        For each $n < \omega$, find a name $\dot{Y}_n$ of a function on $\omega$ such that it is forced that
        \[
        \dot{Y}_n(i) =
        \begin{cases}
            \dot{X}_n \cap P^n_i & \text{if }\varphi(\dot{X}_n) \leq i,\\
            \emptyset & \text{if }\varphi(\dot{X}_n) > i.
        \end{cases}
        \]
        For each $n<\omega$, it is forced that $\forall i\in\omega\,(\varphi(\dot{Y}_n(i)) \leq i)$ and that $\dot{X}_n\subset^*\bigcup_{i<\omega}\dot{Y}_n(i)$. Let $\dot{h}$ be a name of a function on $\omega$ such that it is forced that
        \[
        \dot{h}(i) = \langle\dot{Y}_0(i), \dot{Y}_1(i), \ldots, \dot{Y}_i(i)\rangle.
        \]
        Since $\forces_{\mathbb{P}} \dot{h}\in\prod_{i<\omega}Q_i$, the Laver property implies that there is a function $S \in V$ such that, for all $i<\omega$, $S(i)\in [Q_i]^{\leq i}$ and $\forces_{\mathbb{P}}\forall i<\omega\,(\dot{h}(i)\in S(i))$. For each $n<\omega$, we define
        \[
        A_n = \bigcup_{i\geq n}\bigcup\pi_n[S(i)]\in V.
        \]
        Using (\ref{eqn:grad_frag}), $A_n\in \mathrm{Fin}(\varphi) = \I$. Also, it is forced that $\dot{X}_n\subset^* \bigcup_{i<\omega} \dot{Y}_n(i) \subset^* A_n$. This completes the proof.
    \end{proof}

    The following question still remains open.

    \begin{que}\label{que_cof_cof_omega}
        Is it consistent that $\cof^*(\I) < \cof^*_\omega(\I)$ for some tall Borel ideal $\I$ on $\omega$?
    \end{que}
    
    \begin{rem}\label{rem:aleph_omega_Cohen}
        There is a non-Borel example for this question:
        First, it is not hard to see that $\cof^*_\omega(\J)$ has uncountable cofinality for any ideal $\J$ in ZFC.
        Add $\aleph_\omega$ many Cohen reals $c_\a \in \ooo$ and let $\I$ be the ideal generated by $\{c_\a:\a<\aleph_\omega\}$. Then, $\cof^*(\I)=\aleph_\omega$ can be easily seen. Hence $\cof^*(\I) < \cof^*_\omega(\I)$. 
    \end{rem}
    
    We note that in \cite{BS99}, the same question was raised for an ultrafilter as Question 1 in Chapter 8.

    \subsection{Computation of $\non^*_\omega(\I)$}
    
    We compute $\non^*_\omega(\I)$ for various Borel non-P ideals $\I$. The first and third authors show the following in \cite{CM25}.
    
    \begin{prop}[{Cie\'slak--Mart{\' i}nez-Celis \cite[Proposition 4.7]{CM25}}]\label{prop:nonMI_random_and_conv}
        \[
        \non^*_\omega(\RandomGraph)=\non^*_\omega(\conv) = \omega_1.
        \]
    \end{prop}
    
    We start with studying Fubini products of ideals.
    
    \begin{prop}\label{prop:finfin_d}
        $\non^*_\omega(\finfin) = \mathfrak{d}$.
    \end{prop}
    \begin{proof}
        To show that $\dd\leq\non^*_{\omega}(\finfin)$, let $\mathcal{A}\subset[\omega\times\omega]^\omega$ be of size $<\dd$. We claim that $\mathcal{A}$ cannot be a witness for $\non^*_\omega(\finfin)$. By shrinking each member of $\mathcal{A}$, we may assume that every $A\in\mathcal{A}$ satisfies either
        \begin{enumerate}
            \item $A\subset n\times\omega$ for some $n<\omega$, or
            \item $A$ is an infinite subset of (a graph of) some function $h_A\colon\omega\to\omega$.
        \end{enumerate}
        Since $\lvert\mathcal{A}\rvert<\dd$, there is $f\in\omega^\omega$ such that for all $A\in\mathcal{A}$, $f\not\leq^* h_A$ (whenever $h_A$ is well-defined). Now let
        \[
        X_n = (n\times\omega)\cup\{\langle i, j\rangle:j\leq f(i)\}.
        \]
        Then $\{X_n:n<\omega\}\in[\finfin]^\omega$ witnesses that $\mathcal{A}$ does not witness $\non^*_\omega(\finfin)$, as desired.
    
        To show that $\non^*_\omega(\finfin)\leq\dd$, let $\mathcal{D}$ be a dominating family of strictly increasing functions such that $\forall g\in\omega^\omega\,\exists f\in\mathcal{D}\,(g<f)$. For each $f\in\mathcal{D}$, we let
        \[
        C_f = \{\langle n, f(n)\rangle : n<\omega\}.
        \]
        Then it is straightforward to check that $\mathcal{C} = \{C_f: f\in\mathcal{D}\}$ witnesses $\non^*_\omega(\finfin)$. We write the details for general cases below.
    \end{proof}
    
    \cref{prop:finfin_d} can be generalized as follows 
    but we note that \cref{prop:nonstar_omega_of_Fubini_prod} and \cref{cor:nonstar_vs_nonstar_omega} were already shown for maximal ideals in \cite{BS99}.
    
    \begin{prop}[cf.\ Brendle--Shelah {\cite[Proposition 5.1(d)]{BS99}}]\label{prop:nonstar_omega_of_Fubini_prod}
        For any ideals $\I, \J$ on countable sets,
        \[
        \non^*_\omega(\I\otimes\J) = \max\{\dd, \non^*_\omega(\I), \non^*_\omega(\J)\}.
        \]
    \end{prop}
    \begin{proof}
        We may assume that both $\I$ and $\J$ are ideals on $\omega$.
        Since $\finfin\subset\I\otimes\J$, \cref{lem:reductions_and_invstaromega}(2) and \cref{prop:finfin_d} imply that $\dd = \non^*_\omega(\finfin)\leq\non^*_\omega(\I\otimes\J)$.
    
        Note that $\non^*_\omega(\I), \non^*_\omega(\J)\leq\non^*_\omega(\I\otimes\J)$ do not directly follow from \cref{lem:reductions_and_invstaromega}(2), because $\I$ and $\J$ are Kat\v{e}tov reducible to $\I\otimes\J$, but not necessarily Kat\v{e}tov--Blass reducible.
        However, the inequalities still follow from the easy fact that if $\mathcal{A}\subset[\omega\times\omega]^\omega$ is a witness for $\non^*_\omega(\I\otimes\J)$, then $\{\pi_0[A]:A\in\mathcal{A}\}$ and $\{\pi_1[A]:A\in\mathcal{A}\}$ are witnesses for $\non^*_\omega(\I)$ and $\non^*_\omega(\J)$ respectively. Here, $\pi_i\colon\omega\times\omega\to\omega$ is defined by $\pi_i(\langle n_0, n_1\rangle) = n_i$.
        
        Now we show that $\non^*_\omega(\I\otimes\J) \leq \max\{\dd, \non^*_\omega(\I), \non^*_\omega(\J)\}$.
        Let $\mathcal{D}\subset\omega^\omega$ be a dominating family of strictly increasing functions such that $\forall g\in\omega^\omega\,\exists f\in\mathcal{D}\,(g<f)$. Let $\mathcal{A}$ and $\mathcal{B}$ be witnesses $\non^*_{\omega}(\I)$ and $\non^*_{\omega}(\J)$ respectively. For each $f\in\mathcal{D}, A\in\mathcal{A}, B\in\mathcal{B}$, we write
        \[
        C_{f, A, B} = \{\langle A(n), B(f(n))\rangle: n<\omega\},
        \]
        where $A(k), B(k)$ denote the $k$-th element of the increasing enumeration of $A$, $B$ respectively. We claim that
        \[
        \mathcal{C} := \{C_{f, A, B}: f\in\mathcal{D}, A\in\mathcal{A}, B\in\mathcal{B}\}
        \]
        witnesses $\non^*_{\omega}(\I\otimes\J)$. To prove this, let $\{ X_n:n<\omega\}\in[\I\otimes\J]^{\omega}$ be arbitrary. Then there are $Y_n\in\I$ and $\phi_n\colon\omega\to\J$ such that
        \[
        X_n\subset(Y_n\times\omega)\cup\{\langle i, j\rangle:j\in\phi_n(i)\}.
        \]
        Let $A\in\mathcal{A}$ such that $\forall n<\omega\,(\lvert A\cap Y_n\rvert<\omega)$. 
        Let $B\in\mathcal{B}$ such that $\forall n, i<\omega\,(\lvert B\cap\phi_n(i)\rvert<\omega)$. For each $n<\omega$, take $h_n\in\omega^\omega$ such that 
        \[
        \forall i<\omega\,(B\cap\phi_n(i)\subset B(h_n(i))).
        \]
        Now let $f\in\mathcal{D}$ be such that 
        for all $n<\omega$, $h_n\leq^* f$. Then $\lvert C_{f, A, B}\cap X_n\rvert<\omega$ for all $n<\omega$, since 
        $\forall^\infty m<\omega~A(m)\notin Y_n$ holds and 
        $h_n\leq^* f$ implies $\lvert C_{f, A, B}\cap\{\langle i, j\rangle: j\in\phi_n(i)\}\rvert < \omega$.
    \end{proof}

    \begin{cor}[cf.\ Brendle--Shelah {\cite[Corollary 5.3]{BS99}}]\label{cor:nonstar_vs_nonstar_omega}
        Let $\kappa < \lambda$ be regular cardinals. Then it is consistent that there is a Borel ideal $\I$ on $\omega$ such that $\non^*(\I) = \kappa$ and $\non^*_\omega(\I) = \lambda$.
    \end{cor}
    \begin{proof}
        If $\kappa = \omega$, then $\I$ can be taken as $\finfin$ and assume $\dd=\lambda$. If $\kappa$ is uncountable, then $\I$ can be taken as $\Fin\otimes\edfin$ and assume $\covm=\kappa$ and $\dd=\lambda$. (By \cite[Proposition 3.6]{HMM10}, $\covm=\min\{\dd,\non^*(\edfin)\}$.) Notice that $\non^*(\I\otimes\J) = \non^*(\J)$.
    \end{proof}

    \begin{prop}\label{prop:nonstar_omega_of_nwd}
        $\non^*_\omega(\nwd) = \non(\m)$.
    \end{prop}
    \begin{proof}
        To show $\non^*_\omega(\nwd)\leq\non(\m)$, let $X\subseteq2^\omega$ be such that $X\notin\m$. For $x\in X$ let $Y_x\coloneqq\{x\on n:n\in\omega\}\in[2^{<\omega}]^\omega$. To see $\{Y_x:x\in X\}$ witnesses $\non^*_\omega(\nwd)\leq|X|$, let $\{T_n:n\in\omega\}$ be a countable family of nowhere dense trees in $2^{<\omega}$. Take $x\in X$ such that $x\notin\bigcup_{n\in\omega}[T_n]$. Then $Y_x$ is almost disjoint with all $T_n$. 
        
        To show that $\non(\m)\leq\non^*_\omega(\nwd)$, let $\kappa<\mathrm{non}(\mathcal{M})$ and let $\{A_{\alpha}:\alpha<\kappa\}\subset[2^{<\omega}]^\omega$. It suffices to find a countable family $\{T_n:n<\omega\}$ of nowhere dense trees such that for every $\alpha<\kappa$ there is $n<\omega$ such that $\rvert A_\alpha\cap T_n\rvert=\omega$.
        We may assume that for each $\alpha<\kappa$, $A_\alpha$ forms an antichain and there is $x_{\alpha}\in 2^{\omega}$ such that for every $i<\omega$ all but finitely many members of $A_\alpha$ extend $x_{\alpha}\upharpoonright i$. Let $\{\sigma_\alpha^n:n<\omega\}$ be an enumeration of $A_\alpha$ and set $\sigma_\alpha^{-1} = \emptyset$ for convenience.
        We define a sequence $\{a^n_\alpha:n<\omega\}$ such that
        \[
        a^n_\alpha=\min\{i<\lvert\sigma_{\alpha}^{n}\rvert:\sigma_{\alpha}^{n}(i)\neq x_{\alpha}(i)\}.
        \]
        Without loss of generality, we assume that for every $n<\omega, \lvert\sigma_{\alpha}^{n-1}\rvert < a_{\alpha}^n < \lvert\sigma_{\alpha}^n\rvert$ holds. Let $y_{\alpha}\in2^{\omega}$ be defined by
        \[
        y_\alpha(i) = \begin{cases}
        x_\alpha(i) & \text{ if } i\in [\lvert\sigma_{\alpha}^{n-1}\rvert, a_{\alpha}^n) \text{ for some }n<\omega,\\
        \sigma_{\alpha}^n(i) & \text{ if } i\in [a_{\alpha}^n, \lvert\sigma_{\alpha}^n\rvert) \text{ for some }n<\omega.\\
        \end{cases}
        \]
        As $\kappa<\mathrm{non}(\mathcal{M})$, the set $\{x_{\alpha},y_{\alpha}:\alpha<\kappa\}$ is meager. So there is a real $z\in2^{\omega}$ and an interval partition $\langle I_{n}\rangle_{n\in\omega}$ of $\omega$ such that 
        \[
        \{x_{\alpha},y_{\alpha}:\alpha<\kappa\}\subseteq\{x\in2^{\omega}:\forall^\infty n\in\omega\,(x\upharpoonright I_{n}\neq z\upharpoonright I_{n})\}
        \]
        holds (cf.\ \cite[Theorem 5.2]{Bla10}). For each $n<\omega$, let $T_{n}\subseteq2^{<\omega}$ be a tree such that
        \[
        [T_{n}]=\{x\in2^{\omega}:\forall m \geq n\,(x\upharpoonright I_{m}\neq z\upharpoonright I_{m})\}.
        \]
        It is easy to see that for every $\alpha<\kappa$ there is $n\in\omega$ such that $A_{\alpha}\cap T_{n}$ is infinite. 
    \end{proof}

    Now we consider $\non^*_\omega(\I)$ for $F_\sigma$ non-P-ideals. In \cite[Question 4.11]{CM25}, the first and third authors asked the values of $\non^*_\omega(\I)$ in the case that $\I = \Solecki, \ed, \edfin$. We give complete answers for $\Solecki$ and $\ed$ below. To characterize $\non^*_\omega(\Solecki)$, we introduce the following variant of the covering number of the null ideal.

    \begin{dfn}
        $\covomn\coloneqq\min\{\lvert\mathcal{A}\rvert:\mathcal{A}\subset\n \land \forall B\in[\mathbb{R}]^\omega\,\exists A\in\mathcal{A}\,(B\subseteq A)\}$.
    \end{dfn}

    Using the notation in \cite[Definition 2.1.3]{BJ95}, $\covomn=\cof([\mathbb{R}]^{\omega},\n)$.
    In the following, $\rr_\sigma$ is the ``$\omega$-version'' of the reaping number $\rr$ defined by
    \[
    \rr_\sigma = \min \{ |\mathcal{A}| : \mathcal{A} \subseteq \Baire \land \forall \overline{B}\in[\mathcal{P}(\omega)]^\omega ~ \exists A \in \mathcal{A} ~ \forall B\in\overline{B} ~ (A \subseteq^* B \lor A \cap B =^* \emptyset)\}.
    \]
    The question whether $\rr = \rr_\sigma$ was originally posted in \cite{VOJTAS1992125}, and as far as we know, is still open. The reader interested in more information about this cardinal invariant can consult \cite{Bla10}.

    \begin{prop}\leavevmode
        \begin{enumerate}
            \item $\covn\leq\covomn\leq\min\{\nonm, \rr_\sigma\}$.
            \item $\cf(\covomn)\geq\omega_1$.
        \end{enumerate}
    \end{prop}
    \begin{proof}
        (1) By definition, $\covn\leq\covomn$. To show $\covomn\leq\nonm$, let $X\subseteq\br$ be a nonmeager set. Partition $\br$ into a null set $A$ and a meager set $B$. Given $Y\in[\br]^\omega$, it suffices to show that there is $x\in X$ such that $Y\subseteq x+A\coloneqq\{x+a:a\in A\}$. Assume otherwise. Then for all $x\in X$ there is $y\in Y$ such that $y-x\notin A$, i.e., $y-x\in B$. Thus we have $-X\coloneqq\{-x:x\in\br\}\subseteq\bigcup_{y\in Y}(-y+B)$, which means that a nonmeager set is contained by a meager set, a contradiction.
        The proof of $\covn\leq\rr_\sigma$ is similar to the standard proof of $\covn\leq\rr$: Let $\kappa < \covomn$ and let $\{ A_\alpha : \alpha < \kappa \} \subseteq \Baire$, we will show that $\{ A_\alpha : \alpha < \kappa \}$ is not a witness for $\rr_\sigma$. For all $\alpha < \kappa$, let $N_\alpha = \{ B \in \Baire : A_\alpha \subseteq^* B \lor A_\alpha \cap B =^* \emptyset \}$. It is easy to show that each $N_\alpha$ is Lebesgue null in $\omega^\omega$, and since $\kappa < \covomn$, there must be a sequence $\langle B_n : n < \omega \rangle$ such that, for each $\alpha<\kappa$, there is $n<\omega$ such that $B_n \notin N_\alpha$. But then, for each $\alpha < \kappa$ there is $n < \omega$ such that $|A_\alpha \cap B_n| = |A_\alpha \cap (\omega \setminus B_n))| = \omega$, and therefore $\kappa < \rr_\sigma$.

        (2) Assume on the contrary that there are strictly increasing cardinals $\langle\lambda_i\rangle_{i<\omega}$ such that $\lambda\coloneqq\sup_{i<\omega}\lambda_i=\covomn$. Let $\{A_\a:\a<\lambda\}\subseteq\n$ witness $\lambda=\covomn$. For $i<\omega$, there is $B_i\in[\mathbb{R}]^\omega$ such that $B_i\nsubseteq N_\a$ for all $\a<\lambda_i$. Then $B\coloneq\bigcup_{i<\omega}B_i \in[\mathbb{R}]^\omega$ cannot be covered by any $A_\a$ for $\a<\lambda$, which is a contradiction.
    \end{proof}

    Since $\cf(\covn)=\omega$ is consistent by \cite{She00}, we particularly have:
    \begin{cor}
        $\covn<\covomn$ is consistent.
    \end{cor}

    \begin{prop}\label{prop:Solecki_non_omega}
	   $\non^*_\omega(\mathcal{S}) = \covomn$.
    \end{prop}

    To prove this, we use the following lemma:

    \begin{lem}[{\cite[Lemma 5.5]{HMM10}}] \label{lem:outer_half} 
        For any $X\subseteq 2^\omega$ with $\mu^*(X)<\frac{1}{2}$, there is $Y\in[\Omega]^\omega$ such that for any $x\in X$ and for all but finitely many $V\in Y$,  $x\notin V$.
    \end{lem}
    
    \begin{proof}[Proof of \cref{prop:Solecki_non_omega}]
    	($\leq$) Let $\mathcal{F}\subseteq\mathcal{N}$ witness $|\mathcal{F}|=\covomn$. For each $N\in\mathcal{F}$, let $Y_N\in[\Omega]^\omega$ be as in Lemma \ref{lem:outer_half} when $X\coloneqq N$. To show that $\{Y_N:N\in\mathcal{F}\}$ witnesses $|\mathcal{F}|\geq\non^*_\omega(\mathcal{S})=\min\{|\mathcal{A}|:\mathcal{A}\subseteq[\Omega]^\omega, (\forall X\in[2^\omega]^\omega)~(\exists Y\in\mathcal{A})~(\forall x\in X)~(\forall^\infty V\in Y)~(x\notin V) \}$, let $X\in[2^\omega]^\omega$ and take $N\in\mathcal{F}$ such that $X\subseteq N$. Then, for any $x\in X\subseteq N$ and for all but finitely many $V\in Y_N$,  we have $x\notin V$ by Lemma \ref{lem:outer_half} and hence $\non^*_\omega(\mathcal{S})\leq|\mathcal{F}|$.
        
    	($\geq$) For $F\in[2^\omega]^{<\omega}$ let $S_F=\{C\in\Omega:F\cap C\neq \emptyset\}$. Let $\kappa<\covomn$ and let $\{X_\a:\a<\kappa\}\subseteq[\Omega]^\omega$. Define $Y_\a$ as $\{x\in2^\omega:x\text{ belongs to infinitely many elements of }X_\a\}$. We have that $Y_\a$ is of measure greater $\frac{1}{2}$. Note also that if $x\in Y_\a$ then $|S_{\{x\}}\cap X_\a|=\omega$. Let $Z_\a = Y_\a + \mathbb{Q}$. Clearly $Z_\a$ is of full measure. Now, as $\kappa <\covomn$, there is $A\in[\mathbb{R}]^\omega$ such that $A\cap Z_\a\neq\emptyset$ for all $\a<\kappa$. Then the collection $\{S_{\{x+q\}} : x\in A, q \in\mathbb{Q}\}$ is such that for every $\a < \kappa$ there are $x\in A$ and $q\in \mathbb{Q}$ such that $X_\a \cap S_{\{x+q\}}$ is infinite.
    \end{proof}

    Next, we consider the eventual different ideal.

    \begin{prop}\label{prop:non_star_omega_ed}
    	$\non^*_\omega(\ed)=\covm$.
    \end{prop}

    To prove this, we will use the following lemma:

    \begin{lem}[{\cite[Lemma 2.4.2]{BJ95}, \cite{Mil82}, \cite{Bar87}, \cite[Theorem 5.1]{CM23}}]\label{lem:chara_covm}
        For any cardinal $\kappa$, the following are equivalent:
        \begin{itemize}
		      \item $\kappa<\covm$.
		      \item $\forall F\in[\oo]^\kappa~\forall G\in[\ooo]^\kappa~\exists g\in\oo~\forall f\in F~\forall X\in G~\exists^\infty n\in X~(f(n)=g(n))$. \label{item:covm_two}
            \item $\forall F\in[\oo]^\kappa~\forall h\in\omega^\omega: \text{increasing}~\exists S\in\prod_{n<\omega}[\omega]^{\leq h(n)}~\forall f\in F~\exists^\infty n<\omega~(f(n)\in S(n))$. \label{item:covm_three}
        \end{itemize}
    \end{lem}

    \begin{proof}[Proof of \cref{prop:non_star_omega_ed}]
        To show $\covm\leq\non^*_\omega(\ed)$, let $\mathcal{A}\subseteq[\omega\times\omega]^\omega$ be of size $<\covm$. It suffices to show that $\mathcal{A}$ is not a witness for $\non^*_\omega(\ed)$. For each $A\in\calA$,
        \begin{enumerate}
            \item $A$ contains an infinite subset of $n\times\omega$ or \label{item_ed_verti}
            \item $A$ contains a graph of a function $f_A\colon X_A\to\omega$ for some infinite set $X_A$. \label{item_ed_pf}
        \end{enumerate}
        By \Cref{lem:chara_covm}, there is some $g\in\oo$ such that whenever $A\in\mathcal{A}$ satisfies \eqref{item_ed_pf}, $f_A(n)=g(n)$ holds for infinitely many $n\in X_A$. For each $n<\omega$, let
        \[
        B_n=(n\times\omega)\cup\{\langle i,g(i)\rangle:i<\omega\}\in\ed.
        \]
        Whichever $A\in\mathcal{A}$ satisfies \eqref{item_ed_verti} or \eqref{item_ed_pf}, $\lvert A\cap B_n\rvert = \omega$ for some $n<\omega$. Thus, $\lvert\mathcal{A}\rvert<\non^*_\omega(\ed)$.
	
        To show $\non^*_\omega(\ed)\leq\covm$, let $F\subseteq\oo$ of size $<\non^*_\omega(\ed)$. For $f\in F$, let $A_f\subseteq\omega\times\omega$ be the graph of $f$. Since $\lvert F\rvert<\non^*_\omega(\ed)$, there is $\{B_i:i<\omega\}\subset\ed$ such that for each $f\in F$, $\lvert A_f\cap B_i\rvert=\omega$ for some $i<\omega$. We may assume that every vertical section of $B_i$ is finite. For each $i<\omega$, we define $S_i\in([\omega]^{<\omega})^\omega$ by $S_i(n) = (B_i)_n$ for $n<\omega$. Since $B_i\in\ed$, $S_i\in([\omega]^{\leq k_i})^\omega$ for some $k_i<\omega$. Note that for any $f\in F$, there is $i<\omega$ such that $\exists^\infty n<\omega~f(n)\in S_i(n)$. Now let $h\geq^* 1$ be arbitrary. We choose sufficiently slow-growing function $g\in\omega^\omega$ so that $\sum_{i\leq g(n)}k_i\leq h(n)$ for all $n<\omega$. Then we define $S\in\prod_{n<\omega}[\omega]^{\leq h(n)}$ by $S(n) = \bigcup_{i\leq g(n)}S_i(n)$. Since for all $f\in F$, $\exists^\infty n<\omega~f(n)\in S(n)$, we have $\lvert F\rvert<\covm$ by \Cref{lem:chara_covm}.
    \end{proof}

    We have nothing more than observation about the values of $\non^*_\omega(\I)$ for other $F_\sigma$ ideals $\I$ such as $\edfin, \I_L$ and $\I_P$:
    
    \begin{lem}\label{lem:EDfin_Linear}
         Let $f\in\oo$ go to the infinity and consider $\edfin$ on the set $\{(n,m):m\leq2^n\}$, which is $KB$-equivalent to the original $\edfin$.
         Then, for any $\{I_i:i<\omega\}\subseteq\edfin$ there is $I\in\IL$ such that $I_i\subseteq^* I$ for all $i<\omega$. 
    \end{lem}
    
    \begin{cor}
    	$\non^*_\omega(\edfin)\leq\non^*(\IL)$.
    \end{cor}
    \begin{proof}
        Let $\mathcal{A}\subseteq[\{(n,m):m\leq2^n\}]^\omega$ witness $|\mathcal{A}|=\non^*(\I_L)$. Given $\{I_i:i<\omega\}\subseteq\edfin$, take $I\in\IL$ as in \Cref{lem:EDfin_Linear}. Take $A\in\mathcal{A}$ almost disjoint from $I$. Then $A$ is almost disjoint from all $I_i$, so we have $\non^*_\omega(\edfin)\leq|\mathcal{A}|=\non^*(\IL)$.
    \end{proof}

    \begin{rem}
        A similar argument shows that $\non^*_\omega(\I_L)\leq\non^*(\I_P)$. For the summable ideal $\I_{1/n}$, it is known that $\I_P\leq_{\mathrm{KB}}\I_{1/n}$ and $\non^*_\omega(\I_{1/n}) = \non^*(\I_{1/n}) \leq \non(\n)$ (cf.\ \cite[Theorem 3.7(2)]{HH07}), so $\non^*_\omega(\I_P)\leq\non(\n)$ also holds.
    \end{rem}

    \begin{que}
        Is there an $F_\sigma$ ideal $\I$ such that $\omega_1\leq\non^*(\I)<\non^*_\omega(\I)$ is consistent? Are $\edfin, \I_L, \I_P$ such examples?
    \end{que}

    We have the following general result for $F_\sigma$ ideals, which is the ``$\omega$-version'' of \cite[Corollary 4.6(1)]{HMM10}.

    \begin{thm}
        For every $F_\sigma$ ideal $\I$ on $\omega$, $\non^*_\omega(\I)\leq\mathfrak{l}$. Therefore, it is consistent that for all $F_\sigma$ ideals $\I$ on $\omega$, $\non^*_\omega(\I) < \bb$ holds (e.g. it holds in the Laver model).
    \end{thm}
    \begin{proof}
        Let $\varphi$ be a lower semi-continuous submeasure such that $\mathrm{Fin}(\varphi) = \mathscr{I}$. Fix an interval partition $\langle P_n\rangle_{n<\omega}$ of $\omega$ such that $\varphi(P_n) > n^3$. For each $n<\omega$, we set $Q_n = \{A\subset P_n:\varphi(A)\leq n^2\}$. Let $\{S_\alpha:\alpha<\mathfrak{l}\}\subset\prod_{n<\omega}[Q_n]^{\leq n}$ such that for any $H\in\prod_{n<\omega}Q_n$, there is $\alpha<\mathfrak{l}$ such that $\forall^\infty n<\omega\,(H(n)\in S_\alpha(n))$. For each $\alpha < \mathfrak{l}$ and $n<\omega$, we can pick $a_{\alpha}^n\in P_n\setminus\bigcup S_\alpha(n)$ because $\varphi(P_n)> \varphi(\bigcup S_\alpha(n))$. Let $A_\alpha = \{a_{\alpha}^n: n<\omega\}$.

        We claim that $\mathcal{A} = \{A_\alpha:\alpha<\mathfrak{l}\}$ is a witness for $\non^*_\omega(\I)$. To show this, let $\{X_n : n<\omega\}\subset \I$ be arbitrary. For each $n<\omega$, let $H_n$ be a function with domain $\omega$ defined by
        \[
        H_n(k) =
        \begin{cases}
            P_k \cap X_n & \text{if }\varphi(X_n)\leq k,\\
        \emptyset & \text{if }\varphi(X_n)>k.
        \end{cases}
        \]
        Also, let $H$ be a function on $\omega$ defined by $H(k) = \bigcup_{n < k}H_n(k)$. Then $H\in\prod_{n<\omega}Q_n$. So, there is $\alpha < \mathfrak{l}$ such that $\forall^\infty n<\omega\,(H(n)\in S_\alpha(n))$. Note that for all $n<\omega$, $X_n\subset^*\bigcup_{k<\omega}H(k)$, since $\varphi(X_n)<\infty$. It follows that for all $n<\omega$, $\lvert A_\alpha \cap X_n \rvert < \omega$.
    \end{proof}

    This can be regarded as a reformulation of the ``$\omega$-version'' of \cite[Theorem 4.3 and Lemma 4.4]{HMM10}.
    Recall that a forcing poset $\p$ has the \emph{Laver property} if for any $H\in\oo$ and any $\p$-name $\dot{f}$ of a member of $\prod_{n<\omega}H(n)$,  
    \[
    \Vdash_{\p}\exists A\in\prod_{n<\omega}[H(n)]^{\leq2^n}\cap V~\forall n<\omega~(f(n)\in A(n)).
    \]

    \begin{thm}\label{thm:nonstaromega_is_small_in_Laver}
        Let $\mathscr{I}$ be an $F_\sigma$ ideal on $\omega$ and let $\mathbb{P}$ be a poset with the Laver property. Then
        \begin{equation}
            \label{eq:Laver_property}
            \forces_{\mathbb{P}} \forall \langle \dot{X_n} : n \in \omega \rangle \subseteq \mathscr{I}~\exists X \in [\omega]^\omega \cap V~\forall n \in \omega~(|\dot{X_n} \cap X | < \omega).
        \end{equation}
    \end{thm}
    \begin{proof}
        Let $\varphi$ be a lower semi-continuous submeasure such that $\mathrm{Fin}(\varphi) = \mathscr{I}$ and let $\langle \dot{X_n} : n \in \omega \rangle$ be a sequence of names such that $\forces{\forall n \in \omega (\varphi(\dot{X_n}) \text{ is finite})}$. Find an interval partition of $\omega = \bigcup_{n<\omega} P_n$ such that $\varphi(P_n) > n^2 \cdot 2^n$. For each $n \in \omega$ find $\dot{H_n}$ the name of a function such that it is forced that:
        $$
        \dot{H_n}(k) = \begin{cases}
        \emptyset & \text{if }\varphi(\dot{X_n}) > k,\\
        P_k \cap \dot{X_n} & \text{if }\varphi(\dot{X_n}) \leq k.
        \end{cases}
        $$
        Observe that $\forces{\forall n \in \omega (\varphi(\dot{H_n}(k)) \leq k)}$. Let $\dot{H}$ be a name for a function such that $\dot{H}(k) = \bigcup_{i < k} \dot{H_i}(k)$. We will see that, for all $k\in\omega$,
        \begin{enumerate}
            \item $\forces{\varphi(\dot{H}(k)) \leq k^2}$,
            \item $\forces{\dot{H}(k) \subseteq P_k}$,
            \item $\forces{\dot{X_k} \subseteq^* \bigcup_{i\in\omega}\dot{H}(i)}$.
        \end{enumerate}
        1 and 2 are immediate from the definition of $\dot{H}$. To see 3: Let $p \in \mathbb{P}$ and let $q \leq p$ and $\ell \in \omega$ be such that $k < \ell$ and $q\forces{\varphi(\dot{X}_k) \leq \ell}$. Then, for every $i \geq \ell$ we have that $q \forces{\dot{H_k}(i) = P_i \cap \dot{X_k}}$, and therefore, if $n = \max P_\ell$, we have that $q \forces{\dot{X_k} \setminus n \subseteq \bigcup_{i > \ell}\dot{H}(i)}$.
    
        We will now continue with the proof: Since $\dot{H}$ is the name of a bounded function, by Laver property, there must be a function $G$ such that, for all $k \in \omega$
        \begin{enumerate}[label = (\alph*)]
            \item $G(k)$ is a list of subsets of $P_k$ of measure $\leq k^2$, 
            \item $|G(k)| \leq 2^k$,
            \item $\forces{\dot{H}(k) \in G(k)}$.
        \end{enumerate}
        Therefore, for each $k \in \omega$, (1), (2), (a) and (b) imply that $\bigcup G(k) \subseteq P_k$ and $\varphi(\bigcup G(k)) \leq k^2 \cdot 2^k$, so there is $x_k \in P_k \setminus \bigcup G(k)$. If $X = \{ x_k : k \in \omega\}$ then $X \cap \bigcup_{k \in \omega} \bigcup G(k) = \emptyset$, and therefore, by (3) and (c), we have that for all $k \in \omega$, $\forces{X \cap \dot{X_k}\text{ is finite}}$.
    \end{proof}

    By carefully choosing functions $f,g\in\oo$ which would depend on the submeasure associated to $\I$, we can prove Theorem \ref{thm:nonstaromega_is_small_in_Laver} for $\langle f, g\rangle$-bounding forcings (see \cite[Definition 7.2.13]{BJ95} and \cite[Theorem 7.2.19]{BJ95}).  Therefore, if we consider the forcing $\mathbb{PT}_{f,g}$ (see \cite[Definition 7.3.3]{BJ95} and \cite[Theorem 7.3.9]{BJ95}), we get the following result:
    
    \begin{cor}\label{cor:Fsigma_PTfg}
        Let $\I$ be a tall $F_\sigma$ ideal. Then it is consistent that $\bb = \non^*_\omega(\I) < \non(\mathcal{M})$.
    \end{cor}
    \begin{proof}
        See \cite[Model 7.6.6]{BJ95}.
    \end{proof}

    Note that the choice of functions $f,g$ depends on $\I$, so the argument does not prove the consistency of ``$\bb = \non^*_\omega(\I) < \non(\mathcal{M})$ for all $F_\sigma$ ideals $\I$.''

    \subsection{Remarks on $\I$-Louveau and $\I^*$-Ramsey null ideals}\label{sec:Louveau_Ramsey}

    As a direct application of $\omega$-versions of $*$-numbers of ideals, we observe that some results in Brendle--Shelah \cite{BS99} for ultrafilters actually hold for arbitrary ideals. Brendle--Shelah \cite{BS99} studied the null ideals of Laver forcing associated to ultrafilters, which they call the Louveau ideals, and computed their cardinal invariants. In \cite{CM25}, the first and third authors also considered the null ideals of Laver forcing associated to arbitrary ideals. 

    \begin{dfn}
        For an ideal $\I$ on a countable set $X$, the \emph{$\I$-Louveau null ideal} $L_{\mathscr{I}}$ is defined as the $\sigma$-ideal on $X^\omega$ generated by sets of the form
        \[
        L_\phi = \{x\in X^\omega:\exists^{\infty}n<\omega\,(x(n)\in\phi(x\upharpoonright n))\},
        \]
        where $\phi\colon\omega^{<\omega}\to\mathscr{I}$.
    \end{dfn}

    Given an ideal $\I$ on a countable set $X$, we say that a tree $T\subset X^{<\omega}$ is an \emph{$\I$-Laver tree} if letting $\sigma$ be the stem of $T$, for every $\tau\in T$ with $\sigma\subset\tau$, $\mathrm{succ}_T(\tau)\in\I^+$. The \emph{$\I$-Laver forcing}, denoted by $\mathbb{L}_\I$, is the forcing poset of all $\I$-Laver trees ordered by inclusion. Note that $\mathbb{L}_{\Fin}$ is the standard Laver forcing. Then $\mathbb{L}_\I$ is forcing equivalent to $\mathbb{P}_{L_\I}$ (see \cite{Mil12}).
    One can also define a variant of $L_\I$ as follows.
    
    \begin{dfn}
        For an ideal $\I$ on a countable set $X$, $J_{\mathscr{I}}$ is defined as the $\sigma$-ideal on $X^\omega$ generated by sets of the form
        \[
        J_\phi = \{x\in X^\omega:\exists^{\infty}n<\omega\,(x(n)\in\phi(n))\},
        \]
        where $\phi\colon\omega\to\mathscr{I}$.
    \end{dfn}

    It turned out that the computation in \cite{BS99} works for arbitrary ideals and there is no difference in the cardinal invariants of $L_\I$ and $J_\I$.

    \begin{thm}[Brendle--Shelah {\cite[Theorem 2]{BS99}}]
        \label{thm:Louveau_null}
        For any ideal $\mathscr{I}$ on a countable set, the following hold:
        \begin{enumerate}
            \item $\add(L_\I) = \add(J_\I) =  \min\{\mathfrak{b}, \add^*_\omega(\mathscr{I})\}$.
            \item $\non(L_\I) = \non(J_\I) = \max\{\mathfrak{d}, \non^*_\omega(\mathscr{I})\}$.
            \item $\cov(L_\I) = \cov(J_\I) = \min\{\mathfrak{b}, \cov^*(\mathscr{I})\}$.
            \item $\cof(L_\I) = \cof(J_\I) = \max\{\mathfrak{d}, \cof^*_\omega(\mathscr{I})\}$.
        \end{enumerate}
    \end{thm}
    \begin{proof}
        \begin{enumerate}
        \item It was already shown that $\min\{\bb, \add^*_\omega(\I)\}\leq\add(L_\I)\leq\add^*_\omega(\I)$ in \cite[Proposition 3.6 and Theorem 4.2]{CM25}. The proof can be easily adapted to $J_\I$. So, it is enough to show that $\add(L_\I)\leq\bb$ and $\add(J_\I)\leq\bb$. This follows from the stronger inequality $(\cov(L_\I)\leq)\cov(J_\I)\leq\bb$, which will be shown in (3).
        \item Since $J_\I \subset L_\I$ holds, $\non(J_\I)\leq\non(L_\I)$. So, it suffices to show that (i) $\dd\leq\non(J_\I)$, (ii) $\non^*_\omega(\I)\leq\non(J_\I)$, and (iii) $\non(L_\I)\leq\max\{\non^*_\omega(\I), \dd\}$.
        \begin{enumerate}[label = (\roman*)]
            \item Let $A \notin J_\I$. Then for any $f\in\omega^\omega$, there is $g\in A\setminus J_f$, which implies $f \leq^* g$. Therefore, $A$ is a dominating family and thus $\lvert A\rvert\geq\dd$.
            \item Let $A\notin J_\I$. We claim that $\mathcal{A}:=\{\ran(f):f\in A\}$ witnesses $\non^*_\omega(\I)$: Given $\overline{X} = \{X_n: n<\omega\}\in[\I]^\omega$, any $f\in A\setminus J_\phi$, where $\phi\colon\omega\to\I$ is defined by $\phi(n) = \bigcup_{m\leq n}X_m$, satisfies $\lvert\ran(f)\cap X_n\rvert < \omega$.
            \item To show $\non(L_\I)\leq\max\{\dd, \non^*_\omega(\I)\}$, let $\mathcal{A}$ be a witness for $\non^*_\omega(\I)$ and let $\mathcal{D}$ be a dominating family of functions from $\omega^{<\omega}$ to $\omega$. For each $A\in\mathcal{A}$ and $f\in\mathcal{D}$, choose $x_{A, f}\in\omega^\omega$ such that for all $\sigma\in\omega^{<\omega}$, $x_{A, f}(\sigma)\in A\setminus f(\sigma)$. We claim that $\{x_{A, f}: A\in\mathcal{A}\land f\in\mathcal{D}\}\notin L_\I$. To show this, let $\phi\colon\omega^{<\omega}\to\I$ be arbitrary. Then let $A\in\mathcal{A}$ be such that for all $\sigma\in\omega^{<\omega}$, $\lvert A\cap\phi(\sigma)\rvert<\omega$. Then let $f\in\mathcal{D}$ be such that for all $\sigma\in\omega^{<\omega}$, $A\cap\phi(\sigma)\subset f(\sigma)$. Then $x_{A, f}\notin L_\phi$.
        \end{enumerate}

        \item Since $J_\I \subset L_\I$ holds, $\cov(J_\I)\geq\cov(L_\I)$. So, it suffices to show that (i) $\bb\geq\cov(J_\I)$, (ii) $\cov^*(\I)\geq\cov(J_\I)$, and (iii) $\cov(L_\I)\geq\min\{\bb, \cov^*(\I)\}$.
        \begin{enumerate}[label = (\roman*)]
            \item If $\mathcal{B}$ is an unbounded family, then $\omega^\omega = \bigcup_{f\in\mathcal{B}}J_f$.
            \item If $\mathcal{A}\subset\I$ is a witness for $\cov^*(\I)$, then $\omega^\omega = \bigcup_{A\in\mathcal{A}}J_A$.
            \item Let $\kappa < \min\{\cov^*(\I), \bb\}$ and let $\mathcal{A}\subset L_\I$ be of size $\kappa$. We want to show that $\omega^\omega\neq\bigcup\mathcal{A}$. For each $A\in\mathcal{A}$, let $\phi_A\colon\omega^{<\omega}\to\I$ such that $A\subset L_{\phi_A}$. Since $\kappa<\cov^*(\I)$, we can find $X\in[\omega]^\omega$ such that for all $A\in\mathcal{A}$ and $\sigma\in\omega^{<\omega}$, $\lvert \phi_A(\sigma)\cap X\rvert < \omega$. Then for each $A\in\mathcal{A}$, we define $f_A\in\omega^\omega$ by $f_A(\sigma) = \min\{m<\omega: \phi_A(\sigma)\cap X\subset m\}$. Since $\kappa < \bb$, there is a function $g\colon\omega^{<\omega}\to\omega$ dominating all $f_A$'s. Now take a real $x\in\omega^\omega$ such that $x(n)\in X\setminus g(x\upharpoonright n)$ for all $n<\omega$. Then $x\notin L_{\phi_A}$ for all $A\in\mathcal{A}$, so $x\notin\bigcup\mathcal{A}$.
        \end{enumerate}

        \item We need to show that (i) $\cof(L_\I)\leq\max\{\dd, \cof^*_\omega(\I)\}$, (ii) $\cof(L_\I)\geq\cof^*_\omega(\I)$, and (iii) $\cof(L_\I)\geq\dd$. The proof below can be easily adopted to the proof for $J_\I$.
        \begin{enumerate}[label = (\roman*)]
            \item Let $\mathcal{A}$ be a witness for $\cof^*_\omega(\I)$ and let $\mathcal{D}$ be a dominating family of functions from $\omega^{<\omega}$ to $\omega$. For each $\overline{A} := \{A_n:n<\omega\}\in\mathcal{A}$ and $f\in\mathcal{D}$, let $\phi_{\overline{A}, f}\colon\omega^{<\omega}\to\I$ be defined by $\phi_{\overline{A}, f}(\sigma) = \bigcup_{i\leq f(\sigma)} A_i\cup f(\sigma)$. To see that $\{L_{\phi_{\overline{A}, f}}: \overline{A}\in\mathcal{A}\land f\in\mathcal{D}\}$ is a base of $L_\I$, let $\psi\colon\fsq\to\I$ be arbitrary. Take $\overline{A} := \{A_n:n<\omega\}\in\mathcal{A}$ such that for each $\sigma\in\fsq$, there is $g(\sigma)<\omega$ with $\psi(\sigma)\subseteq^*A_{g(\sigma)}$. Let $h\colon\fsq\to\omega$ be such that $\psi(\sigma)\subseteq A_{g(\sigma)}\cup h(\sigma)$ for each $\sigma\in\fsq$. Take $f\in\mathcal{D}$ which dominates $g,h$. To see $L_{\psi}\subset L_{\phi_{\overline{A}, f}}$, let $x\in L_\psi$ be arbitrary. Then for infinitely many $n<\omega$, $x(n)\in\psi(x\on n)\subset A_{g(x\on n)}\cup h(x\on n)\subset \bigcup_{i\leq f(x\on n)} A_i\cup f(x\on n)=\phi_{\overline{A}, f}(x\on n)$, which implies $x\in L_{\phi_{\overline{A}, f}}$. 
            \item Let $\mathcal{A} \subset L_\I$ be a base of $L_\I$. For any $A\in\mathcal{A}$, let $\phi_A\colon\omega^{<\omega}\to\I$ be such that $A\subset L_{\phi_A}$. We claim that $\overline{\mathcal{A}} = \{\mathrm{ran}(\phi_A):A\in\mathcal{A}\}$ is a witness for $\cof^*_\omega(\I)$. Let $\overline{B}:=\{B_n:n<\omega\}\in[\I]^\omega$ be arbitrary. Using some function $f\colon\omega\to\omega$ such that $\lvert f^{-1}[\{n\}]\rvert=\omega$ for all $n<\omega$, we define $\phi\colon\omega^{<\omega}\to\I$ by $\phi(\sigma) = B_{f(\lvert\sigma\rvert)}$. Take $A\in\mathcal{A}$ such that $L_{\phi}\subset A$. Suppose that there is $B\in\overline{B}$ such that for all $\sigma\in\omega^{<\omega}$, $B\not\subset^*\phi_A(\sigma)$. Then we can find $x\in\omega^\omega$ such that $x(n)\in B\setminus\phi_A(x\upharpoonright n)$ for all $n<\omega$, which implies $x\in L_\phi\setminus A$. This contradicts the choice of $A$.
            \item $\cof(L_\I)\geq\dd$ follows from $\cof(L_\I)\geq\non(L_\I)\geq\dd$.
        \end{enumerate}
        \end{enumerate}
    \end{proof}

    Brendle--Shelah \cite{BS99} also studied the null ideals of Mathias forcing associated to ultrafilters, which they call the Ramsey ideals, and computed their cardinal invariants. 

    \begin{dfn}
        Let $\I$ be an ideal on a countable set $X$. We define the \emph{$\I^*$-Mathias forcing} $\mathbb{R}_{\I^*}$ as follows: the conditions are pairs $\langle s, A\rangle\in[X]^{<\omega}\times\I^*$ such that $\max(s)<\min(A)$. The order is defined by
        \[
        \langle t, B\rangle\leq\langle s, A\rangle \iff t\supseteq s\land B\subset A\land t\setminus s\subset A.
        \]
        For $\langle s, A\rangle\in\mathbb{R}_{\I^*}$, we write
        \[
        [s, A] = \{Y\in[X]^\omega: s\subset Y\subset s\cup A\}.
        \]
        The \emph{$\I^*$-Ramsey null ideal} $R_{\I^*}$ consists of $Y\subset[X]^\omega$ such that
        \[
        \forall\langle s, A\rangle\in\mathbb{R}_\I\,\exists\langle t, B\rangle\leq\langle s, A\rangle\,(Y\cap[t, B] = \emptyset).
        \]
    \end{dfn}

    By definition, one can see that $\mathbb{R}_{\I^*}$ is forcing equivalent to $\mathbb{P}_{R_{\I^*}}$. Brendle--Shelah \cite{BS99} showed the theorem below only for ultrafilters, but their proof actually works for arbitrary filters. We leave the details to the careful readers.

    \begin{thm}[Brendle--Shelah {\cite[Theorem 1]{BS99}}]\label{thm:ramsey_null}
        For any tall ideal $\I$ on a countable set, the following holds:
        \[
        \add(R_{\I^*}) = \add^*(\I), \non(R_{\I^*}) = \non^*(\I), \cov(R_{\I^*}) = \cov^*(\I), \cof(R_{\I^*}) = \cof^*_\omega(\I).
        \]
    \end{thm}

    In contrast to \cref{thm:Louveau_null} and \cref{thm:ramsey_null}, we will see that the computation of cardinal invariants associated to $\I$-Miller null ideals are more complicated. Notably, uniformity and covering numbers of $\I$-Miller null ideals are not determined by the $*$-numbers of $\I$. Moreover, it is consistent that $M_\I$ and $K_\I$ have different uniformity and covering numbers for some ideal $\I$, e.g.\ $\finfin$ and $\ed$.

    \section{Computation of cardinal invariants of $M_\I$}\label{sec:ZFC}

    In this section, we discuss $\mathsf{ZFC}$-provable results on cardinal invariants of the $\sigma$-ideals $M_\I$ and $K_\I$.

    \subsection{Additivity and cofinality}\label{sec:add_cof}

    It turns out that the additivity and cofinality of $M_\I$ and $K_\I$ are the same as $L_\I$ and $J_\I$, which are characterized using the $\omega$-version of the $*$-additivity and cofinality of the original ideal $\I$, as in \Cref{thm:Louveau_null}. Namely, we obtain the following equalities for any ideal $\I$:
    \begin{align*}
        \add(M_\I) = \add(K_\I)&= \min\{\bb, \add^*_\omega(\I)\},\\ 
             \cof(M_\I) = \cof(K_\I) &= \max\{\dd, \cof^*_\omega(\I)\}.
    \end{align*}
    We start with the following bounds of $\add(M_\I)$ and $\cof(M_\I)$, some of which were mentioned by the first and third authors in \cite{CM25}:
    
    \begin{lem}\label{lem:add_cof_ineq}
        Let $\I$ be an ideal on a countable set.
        \begin{enumerate}
            \item \textnormal{(Cieślak--Martínez-Celis \cite[Proposition 3.6, Theorem 4.2]{CM25})}\\
            $\min\{\bb, \add^*_\omega(\I)\}\leq\add(M_\I)\leq\add^*_\omega(\I)$. \label{item:add} 
            \item $\cof^*_\omega(\I)\leq\cof(M_\I)\leq\max\{\dd, \cof^*_\omega(\I)\}$.\label{item:cof} 
        \end{enumerate}
        The same inequalities hold for $K_\I$ as well.
    \end{lem}
    \begin{proof}
        We may assume that $\I$ is an ideal on $\omega$. We will give proofs only for $M_\I$ because the essentially same argument also proves the lemma for $K_\I$.
        
        \eqref{item:add} The inequality is stated in \cite[Proposition 3.6 \& Theorem 4.2]{CM25} without a proof. To see $\min\{\bb, \add^*_\omega(\I)\}\leq\add(M_\I)$, let $\Phi$ be a family of functions from $\omega^{<\omega}$ to $\I$ such that $\lvert\Phi\rvert<\min\{\add^*_\omega(\I),\bb\}$ and we show that $\{M_\phi\colon\phi\in\Phi\}$ is not a witness for $\add(M_\I)$. Since $\{\phi(s):\phi\in\Phi\land s\in\omega^{<\omega}\}$ has size $<\add^*_\omega(\I)$, there is $\{B_n:n<\omega\}\subset\I$ such that for all $\phi\in\Phi$ and $s\in\omega^{<\omega}$, there is $n<\omega$ such that $\phi(s)\subset^* B_n$. We may assume that $\langle B_n:n<\omega\rangle$ is $\subset$-increasing and $n\subset B_n$ for each $n<\omega$. Then for each $\phi\in\Phi$, we can define $f_\phi\colon\omega^{<\omega}\to\omega$ by letting $f_\phi(s)$ be the least $n<\omega$ such that $\phi(s)\subset B_n$. Since $\lvert\Phi\rvert<\bb$, there is $f\colon\omega^{<\omega}\to\omega$ such that for all $x\in\omega^\omega$, 
        \[
        \forall^\infty n<\omega~f_\phi(x\upharpoonright n)\leq f(x\upharpoonright n).
        \]
        Define $\psi\colon\omega^{<\omega}\to\I$ by $\psi(s)\coloneqq B_{f(s)}$. Then $M_\phi\subset M_{\psi}$ for all $\phi\in\Phi$ because if $x\in M_\phi$, then for all but finitely many $n<\omega$,
        \[
        x(n)\in\phi(x\upharpoonright n)\subset B_{f_\phi(x\upharpoonright n)}\subset B_{f(x\upharpoonright n)} = \psi(x\upharpoonright n)
        \]
        holds and hence $x\in M_{\psi}$.
        
        To see $\add(M_\I)\leq\add_\omega^*(\I)$, let $\mathcal{A}\subset\I$ of size $<\add(M_\I)$ and we show that $\mathcal{A}$ is not a witness for $\add_\omega^*(\I)$. For $A\in\mathcal{A}$, let $\phi_A\colon\omega^{<\omega}\to\I$ denote the constant function of the value $A$. Since $\lvert\mathcal{A}\rvert<\add(M_\I)$, there is $\psi\colon\omega^{<\omega}\to\I$ such that $M_{\phi_A}\subset M_\psi$ for all $A\in\mathcal{A}$. We claim that for any $A\in\mathcal{A}$, $A\subset^*\psi(s)$ for all $s\in\omega^{<\omega}$. Suppose not. Then for some $A\in\mathcal{A}$, one can find $x\in\omega^\omega$ such that $x(n)\in A\setminus \psi(x\upharpoonright n)$ for $n<\omega$. Then $x$ witnesses $M_{\phi_A}\not\subset M_\psi$, which is a contradiction.
    
        \eqref{item:cof} To show $\cof^*_\omega(\I)\leq\cof(M_{\I})$, let $\{M_{\phi_\alpha}:\alpha<\kappa\}$ be a basis of $M_\I$. We claim that $\{\mathrm{ran}(\phi_\alpha):\alpha<\kappa\}$ is a witness for $\cof^*_\omega(\I)$; Suppose otherwise. Then there is $\{B_i:i<\omega\}\subset\I$ such that
        \[
        \forall\alpha<\kappa\,\exists i_\alpha < \omega\,\forall s\in\omega^{<\omega}~\lvert B_{i_\alpha}\setminus\phi_\alpha(s)\rvert = \omega.
        \]
        Then for each $\alpha < \kappa$, one can find $x\in\omega^\omega$ such that for all $n<\omega$, $x(n)\in B_{i_\alpha}\setminus\phi_\alpha(x\upharpoonright n)$ and thus $x\in M_{B_{i_\alpha}}\setminus M_{\phi_\alpha}$. Therefore, no $M_{\phi_\alpha}$ covers $\bigcup_{i<\omega}M_{B_i}$, which contradicts the choice of $M_{\phi_\alpha}$'s. 

        Now we show $\cof(M_\I)\leq\max\{\dd, \cof^*_\omega(\I)\}$. Let $\{\overline{A}_\alpha:\alpha<\cof^*_\omega(\I)\}$ be a witness for $\cof^*_\omega(\I)$. For each $\alpha<\cof^*_\omega(\I)$, we write $\overline{A}_\alpha = \{A_\alpha^i:i<\omega\}$.
        Also, let $\{f_\beta:\beta<\dd\}$ be a dominating family of functions from $\omega^{<\omega}$ to $\omega$. We may assume that each $f_\beta$ is increasing along branches, i.e.\ $s\subset t$ implies that $f_\beta(s)\leq f_\beta(t)$. For every $\alpha < \cof^*_\omega(\I)$ and $\beta<\dd$, we define $\phi_{\alpha, \beta}\colon\omega^{<\omega}\to\I$ by
        \[
        \phi_{\alpha, \beta}(s) = \bigcup_{i=0}^{f_{\beta}(s)}A_\alpha^i \cup f_\beta(s)
        \]
        for all $s\in\omega^{<\omega}$. We claim the collection $\{M_{\phi_{\alpha, \beta}}:\alpha<\cof^*_\omega(\I)\land\beta<\dd\}$ is a basis of $M_\I$; Take any set $M_\psi$, where $\psi\colon\omega^{<\omega}\to\I$. Let $\alpha<\cof^*_\omega(\I)$ be such that for each $s\in\omega^{<\omega}$ there is $i<\omega$ with $\psi(s)\subset^* A_{\alpha}^i$. Then let $g, h\colon\omega^{<\omega}\to\omega$ be functions such that for $s\in\omega^{<\omega}$,
        \[
        \psi(s)\subset A_\alpha^{g(s)} \cup h(s).
        \]
        There is a $\beta<\dd$ such that $f_\beta$ dominates both $g$ and $h$. Then we have $M_\psi\subseteq M_{\phi_{\alpha, \beta}}$, because if $x\in M_\psi$ then for all but finitely many $n$,
        \[
        x(n)\in\psi(x\upharpoonright n)\subset A_\alpha^{g(x\upharpoonright n)}\cup h(x\upharpoonright n)\subset \bigcup_{i=0}^{f_\beta(x\upharpoonright n)}A_\alpha^i \cup f_\beta(x\upharpoonright n) = \phi_{\alpha, \beta}(x\upharpoonright n)
        \]
        holds and thus $x\in M_{\phi_{\alpha, \beta}}$. This completes the proof for $M_\I$, but also this argument works for $K_\I$.
    \end{proof}
    
    The following lemma is used to prove $\add(K_\I)\leq\bb$ and $\cof(K_\I)\geq\dd$:
    
    \begin{lem}\label{lem:K_phi_psi}
        Let $\phi,\psi\colon \omega\to\I$ and assume that $\phi$ does not take empty values. Then $K_\phi\subseteq K_\psi$ implies $\forall^\infty n<\omega~\phi(n)\subset\psi(n)$.
    \end{lem}
    \begin{proof}
        Assume that there is $D\in[\omega]^\omega$ such that $\phi(n)\not\subset \psi(n)$ for every $n\in D$. Take some $x\in\omega^\omega$ such that $x(n)\in \phi(n)\setminus \psi(n)$ for $n\in D$ and $x(n)\in\phi(n)$ for $n\notin D$. Then $x\notin K_\phi\setminus K_\psi$.
    \end{proof}
    
    \begin{lem}\label{lem:addK_I_leq_b}
        $\add(K_\I)\leq\bb$ and $\cof(K_\I)\geq\dd$.
    \end{lem}
    \begin{proof}
        To see $\add(K_\I)\leq\bb$, let $F\subset\omega^\omega$ of size $<\add(K_\I)$ and it suffices to show that $F$ is not an unbounded family. We may assume $f(n)>0$ for all $f\in F$ and $n<\omega$. Since $\omega\subseteq\Fin\subseteq\I$, $F\subseteq\I^\omega$ and hence there is $\psi\in\I^\omega$ such that $K_f\subseteq K_\psi$ for all $f\in F$. By \Cref{lem:K_phi_psi}, $\forall^\infty n<\omega~f(n)\subset\psi(n)$ holds for all $f\in F$. Then a function $g\in\omega^\omega$ defined by $g(n)\coloneqq\min(\omega\setminus\psi(n))$ dominates $F$, which completes the proof. One can prove $\cof(K_\I)\geq\dd$ in a dual manner.
    \end{proof}

    In the case of $M_\I$, the idea is the same, but we need to be more careful.

    \begin{dfn}
        For $\sigma,\tau\in\fsq$ and $\phi\colon\fsq\to\I$, we write $\sigma\subseteq_\phi\tau$ if $\sigma\subseteq\tau$ and for all $n\in|\tau|\setminus|\sigma|$, $\tau(n)\in\phi(\tau\on n)$.
    \end{dfn}
    
    \begin{lem}\label{lem:M_phi_psi}
        Let $\phi,\psi\colon\fsq\to\I$ and assume that $\phi$ does not take empty values. Then $M_\phi\subseteq M_\psi$ implies $\forall\sigma\in\fsq \exists\tau\supseteq_\phi\sigma\forall\rho\supseteq_\phi\tau~\phi(\rho)\subseteq\psi(\rho)$.
    \end{lem}
    \begin{proof}
        Assume otherwise. By inductive construction, we can find $x\in M_\phi$ such that $x(n)\in \phi(x\on n)\setminus \psi(x\on n)$ for infinitely many $n<\omega$ and hence $x\in M_\phi\setminus M_\psi$.
    \end{proof}

    \begin{dfn}
	   For $x\in\oo$ and $g\colon \fsq\to\omega$, we say that $x$ is \emph{adaptively dominated} by $g$, denoted by $x\sbd g$, if $x(n)\leq g(x\on n)$ for all but finitely many $n<\omega$.
    \end{dfn}

    \begin{lem}[see e.g.\ {\cite[Table 1, Page 463]{Bla10}}]\label{lem:seq_unbounded}
	   Every unbounded family is $\sbd$-unbounded. Namely, for any ($\leq^*$-)unbounded family $B\subseteq\oo$ and $g\colon \fsq\to\omega$, there is $x\in B$ such that $\lnot(x\sbd g)$. In addition, the smallest sizes of a $\sbd$-unbounded family and a $\sbd$-dominating family are $\bb$ and $\dd$, respectively.
    \end{lem}
    \begin{proof}
        Let $B\subset\omega^\omega$ be an unbounded family and $g\colon\omega^{<\omega}\to\omega$. Fix an enumeration $\{\sigma_i:i<\omega\}$ of $\omega^{<\omega}$ and for each $i<\omega$, let $T^i\subseteq\omega^{<\omega}$ be the tree given by
        \[
        \tau\in T^i \iff (\tau\subset\sigma_i)\lor(\sigma_i\subset\tau\land\forall n \in \lvert\tau\rvert \setminus \lvert\sigma_i\rvert~\tau(n)\leq g(\tau\upharpoonright n))
        \]
    	and let $f^i\in\omega^\omega$ be given by $f^i(n) = \max\{t(n):t\in T^i\cap\omega^{n+1}\}$. Also, we define $f\colon\omega\to\omega$ by $f(n) = \max\{f^i(n):i\leq n\}$. Since $B\subseteq\oo$ is an unbounded family, there is $x\in B$ such that $x(n)>f(n)$ for infinitely many $n<\omega$. Assume toward a contradiction that there exists $m<\omega$ such that for all $n\geq m$, $x(n)\leq g(x\on n)$. Let $i<\omega$ be such that $x\on m=\sigma_i$. Then $x$ is a branch of $T^i$, so $x(n)\leq f(n)$ for all $n\geq \max\{m,i\}$, which is a contradiction. The latter statement is straightforward.
    \end{proof}
    
    \begin{lem}\label{lem:addM_I_leq_b}
        $\add(M_\I)\leq\bb$ and $\cof(M_\I)\geq\dd$.
    \end{lem}
    \begin{proof}
        To see $\add(M_\I)\leq\bb$, let $F\subseteq\oo$ of size $<\add(M_\I)$. It suffices to find $g\colon\omega^{<\omega}\to\omega$ such that $f\sbd g$ for all $f\in F$ by \Cref{lem:seq_unbounded}. For $f\in F$, we define $\phi_f\colon\omega^{<\omega}\to\I$ by $\phi_f(\sigma) =  f(\lvert\sigma\rvert)+1$. Since $\lvert F\rvert<\add(M_\I)$, there is $\psi\colon\omega^{<\omega}\to\I$ such that $M_{\phi_f}\subset M_\psi$ for all $f\in F$. Fix an enumeration $\{\tau_i:i<\omega\}$ of $\omega^{<\omega}$.
        For each $x\in\omega^{\leq\omega}$ and $i<\omega$, we define $x^i\in\omega^{\leq\omega}$ by $\dom(x^i)=\dom(x), x^i(n) = \tau_i(n)$ for $n<\lvert\tau_i\rvert$ and $x^i(n) = x(n)$ for $n\geq\lvert\tau_i\rvert$. We also define $\psi^\prime\colon\omega^{<\omega}\to\I$ and $g\colon\omega^{<\omega}\to\omega$ by
        \begin{align*}
            \psi^\prime(\sigma) &=\bigcup\{\psi(\sigma^i):i\leq|\sigma|\},\\
            g(\sigma) &= \min(\omega\setminus\psi^\prime(\sigma)).
        \end{align*}
        For each $f\in F$, by \Cref{lem:M_phi_psi} (putting $\sigma\coloneqq\emptyset$), there is $i_f<\omega$ such that
        \begin{equation}\label{eq:rho}
            \forall\rho\supseteq_{\phi_f}\tau_{i_f}~(\phi_f(\rho)\subseteq\psi(\rho)). 
        \end{equation}
        Fix $f\in F$ for now and write $i\coloneqq i_f$ and $\phi\coloneqq\phi_f$. For $n\geq\rvert\tau_{i}\lvert$,  since $f^i(n)=f(n)\in f(n)+1 = \phi(f^i\on n)$, we have $f^i\on n\supseteq_{\phi}f^i\on\lvert\tau_i\rvert =\tau_{i}$. Then by \eqref{eq:rho}, for $n\geq\max\{\lvert\tau_{i}\rvert,i\}$,
        \[
        f(n)+1=\phi(f^i \on n)\subseteq\psi(f^i\on n)\subseteq\psi^\prime(f\on n),
        \]
        which implies $f(n)\leq g(f\on n)$. Since $f\in F$ was arbitrary, we have $f\sbd g$ for all $f\in F$. One can prove $\cof(M_\I)\geq\dd$ in a dual manner.
    \end{proof}

    Together with \Cref{lem:add_cof_ineq}, we obtain:
    
    \begin{thm}\label{thm:exact_values_of_add}
        For any ideal $\I$ on a countable set, the following holds:
        \begin{align*}
            \add(K_\I) = \add(M_\I) &= \min\{\bb, \add^*_\omega(\I)\}\\ 
            \cof(K_\I) = \cof(M_\I) &= \max\{\dd, \cof^*_\omega(\I)\}.
        \end{align*}
    \end{thm}

    Most of the ideals $\I$ in Table \ref{table_values} satisfy $\add^*_\omega(\I)\leq\bb$ and $\cof^*_\omega(\I)\geq\dd$  and hence $\add(K_\I)=\add(M_\I)=\add^*_\omega(\I)$ and $\cof(K_\I)=\cof(M_\I)=\cof^*_\omega(\I)$ for such $\I$. However, this is not always the case for any ideal $\I$:
    
    \begin{cor}\label{cor:grad_add_cof_diff}
        \begin{enumerate}
            \item It is consistent that for any gradually fragmented ideal $\I$, $\add(K_\I)=\add(M_\I)=\bb<\add^*_\omega(\I)$ holds.
            \item It is consistent that for any gradually fragmented ideal $\I$, $\cof^*_\omega(\I) < \dd=\cof(K_\I)=\cof(M_\I)$ holds.
        \end{enumerate}
    \end{cor}
    \begin{proof}
        \begin{enumerate}
            \item Work in the model constructed in \cite{Paw85} where $\bb < \add_t(\mathcal{N})$ holds. Let $\I$ be a gradually fragmented ideal. Then, $\add_t(\mathcal{N})\leq\add^*_\omega(\I)$ by \Cref{thm:addstaromega_grad_frag}.
            \item Work in the Laver model and let $\I$ be a gradually fragmented ideal. Then, $\cof^*_\omega(\I) < \bb=\dd$ by \Cref{thm:grad_frag_in_Laver}.
        \end{enumerate}
    \end{proof}

    \subsection{Uniformity and covering}\label{sec:non_cov}

    We first observe that there are natural bounds for uniformity and covering numbers of $K_\I$ and $M_\I$.

    \begin{lem}\label{lem:b_nonM}
        For any ideal $\I$ on a countable set,
        \begin{align*}
            \bb\leq\non(K_\I) &\leq \non(M_\I)\leq\non(\mathcal{M})\\
            \covm\leq\cov(M_\I) & \leq\cov(K_\I)\leq\dd.
        \end{align*}
    \end{lem}
    \begin{proof}
        We may assume that $\I$ is an ideal on $\omega$. Then we have
        \[
        K_{\sigma}\subset K_\I \subset M_\I \subset \mathcal{M},
        \]
        where $K_\sigma$ is the ideal of $\sigma$-compact sets of the Baire space. This implies the claim, since $\non(K_\sigma) = \bb$ and $\cov(K_\sigma) = \dd$.
    \end{proof}

    \begin{prop}
        $K_\I = M_\I$ if and only if $\I=\Fin$.
    \end{prop}
    \begin{proof}
        $K_{\Fin}=M_{\Fin}=K_\sigma$ is easy.
        When $\I\neq\Fin$, fix an infinite set $A=\{a_0<a_1<\cdots\}\in\I$ and inductively construct an $\I$-branching tree $T\subseteq\fsq$ as follows: for every $k<\omega$ and $\sigma\in\omega^{2k}$, $\suc_{T}(\sigma)\coloneqq A$ and $\suc_{T}(\sigma^\frown\langle a_i\rangle)\coloneqq\{i\}$. Clearly $T\in M_\I$, so we show $T\notin K_\I$. Let $\phi\colon\omega\to\I$ be arbitrary. For each $k<\omega$, take a natural number $i_k\notin\phi(2k+1)$ and define $x\in\oo$ by $x(2k)\coloneqq a_{i_k}$ and $x(2k+1)\coloneqq i_k$. Then, $x\in T$ but for any $k<\omega$, $x(2k+1)\notin \phi(2k+1)$. Since $\phi$ was arbitrary, we have $T\notin K_\I$.
    \end{proof}

    \begin{lem}[{\cite[Proposition 4.16]{CM25}, \cite[Lemma 4.3]{CGMRS24}}]\label{lem:Katetov_non}
        Let $\I,\J$ be ideals on $\omega$ such that $\I\leq_{\mathrm{K}}\J$. Then $\non(M_\I)\leq\non(M_\J)$ and $\cov(M_\I)\geq\cov(M_\J)$ hold. The same also holds for $K_\I$ and $K_\J$.
    \end{lem}
    \begin{proof}
        We only prove $\non(M_\I)\leq\non(M_\J)$.
    	Let $f\in\oo$ witness $\I\leq_{\mathrm{K}} \J$.
        Assume $F\subset\oo$ of size $<\non(M_\I)$ is given.
    	For $y\in \oo$, define $y^\prime\in\oo$ by $y^\prime(n)=f(y(n))$ for $n<\omega$.
        Since $F^\prime\coloneqq\{y^\prime:y\in F\}$ has size $<\non(M_\I)$,
        there is $\phi\colon\fsq\to\I$ such that $F^\prime\subset M_\phi$.
        Define $\psi\colon\fsq\to\J$ by $\psi(\sigma)\coloneqq f^{-1}(\phi(\sigma))$ for $\sigma\in\fsq$. Then we have $F\subset M_\psi$.
    \end{proof}
    
    \begin{thm}\label{thm:nonmi_leq_max}
        For every ideal $\I$ on a countable set,
        \begin{align*}
            \non(M_\I) &\leq \max\{\bb,\non^*_\omega(\I)\},\\
            \cov(M_\I) &\geq \min\{\dd,\cov^*(\I)\}.
        \end{align*}
    \end{thm}
    \begin{proof}
        We may assume that $\I$ is an ideal on $\omega$. We only prove the former statement, since the latter is easier. Let $B\subseteq\oo$ and $\cal{A}\subseteq\ooo$ witness $\bb=\lvert B\rvert$ and $\non^*_\omega(\I)=\lvert\mathcal{A}\rvert$. 
        For $A\in\ooo$ and $f\in\oo$, we define $h(A,f)\in\oo$ by
        \[
        h(A, f)(n) = \min(A\setminus f(n)).
        \]
        We claim that $F\coloneqq\set{h(A,f)}{A\in\cal{A}, f\in B}$ is not in $\mzbi$. Given $\phi\colon \fsq\to\I$, 
        there is $A\in\cal{A}$ such that for all $\sigma\in\fsq$, $\lvert A\cap \phi(\sigma)\rvert<\omega$. Let $g\colon \fsq\to\omega$ be such that $A\cap \phi(\sigma)\subseteq g(\sigma)$ for all $\sigma\in\omega^{<\omega}$. Notice that $\{h(A,f):f\in B\}$ is also an unbounded family since $h(A,f)(n)\geq f(n)$ for $f\in B$ and $n<\omega$. 
        By \Cref{lem:seq_unbounded}, there is $f\in B$ such that $h(n)>g(h\on n)$ for infinitely many $n<\omega$, where $h\coloneqq h(A, f)$. Then $h(n)\in A\setminus \varphi(h\on n)$ for such $n<\omega$, so $h\in F\setminus M_\phi$.
    \end{proof}

    \begin{rem}\label{rem:finfin_b_d}
        The converse inequalities of \Cref{thm:nonmi_leq_max} cannot be proved in general, witnessed by $\I=\finfin$: $\non(M_{\finfin})\leq\nonm$ by \Cref{lem:b_nonM}, but $\max\{\bb,\non^*_\omega(\finfin)\}=\dd$ by \Cref{prop:finfin_d}. 
    \end{rem}

    By \Cref{prop:nonMI_random_and_conv}, we have:

    \begin{cor}
        If $\I$ is either $\RandomGraph$ or $\conv$, $\non(M_\I) = \non(K_\I) = \bb$ and $\cov(M_\I) = \cov(K_\I) = \dd$.
    \end{cor}
    
    By \Cref{cor:Fsigma_PTfg}, we have:
    
    \begin{cor}\label{cor:Fsigma_MI}
        Let $\I$ be an $F_\sigma$ ideal. Then it is consistent that $\bb = \non(M_\I) < \non(\mathcal{M})$. 
    \end{cor}

    \subsubsection{Relations with constant evasion and prediction numbers}\label{sec:rel_w_const}

    We relate $\non(M_\I)$ (resp.\ $\cov(M_\I)$) for $\I = \ed, \edfin, \finfin$ to the constant evasion (resp.\ prediction) numbers, which stems from Kamo's work (\cite{Kam00,Kam01}). As one of the consequences, we obtain the affirmative answer for \cite[Question 4.19]{CM25}, where the first and third authors asked whether $\cov(M_{\finfin})<\dd$ is consistent. We start with some definitions, following the notation introduced in the recent work by Cardona and Repick{\'y} \cite{CR25}.

    \begin{dfn}\label{dfn:constant_evasion_numbers}\leavevmode
        \begin{itemize}
            \item A \emph{predictor} is a function $\pi\colon\fsq\to\omega$. Let $\Pred$ denote the set of all predictors.
            \item Let $f\in\oo$ and $\pi\in\Pred$. For $k\geq 2$, we say \emph{$\pi$ $k$-constantly predicts $f$}, denoted by $f\pck\pi$, if:
            \[
            \forall^\infty i<\omega~ \exists j\in\left[i,i+k\right)~ f(j)=\pi(f\on j).
            \]
            We say \emph{$\pi$ constantly predicts $f$}, denoted by $f\pc\pi$, if $\pi$ $k$-constantly predicts $f$ for some $k\geq 2$.
            \item Define the following cardinal invariants: 
            \begin{align*}
				\eeck & \coloneqq\min\{|F|:F\subseteq\oo,\forall \pi\in \Pred~\exists f\in F~ \lnot(f\pck\pi)\},\\
                \eec & \coloneqq\min\{|F|:F\subseteq\oo,\forall \pi\in \Pred~\exists f\in F~ \lnot(f\pc\pi)\},\\
				\vvck & \coloneqq\min\{|\Pi|:\Pi\subseteq \Pred,\forall f\in\oo~\exists\pi\in \Pi~ f \pck \pi\},\\
                \vvc & \coloneqq\min\{|\Pi|:\Pi\subseteq \Pred,\forall f\in\oo~\exists\pi\in \Pi~ f \pc \pi\}.
		      \end{align*}
            \item For $b\in(\omega\setminus2)^\omega$, we define constant evasion/prediction on the restricted space $\prod b\coloneqq\prod_{n<\omega} b(n)$ as follows: $\Pred_b$ denotes the set of all functions $\pi$ with $\dom(\pi)=\bigcup_{n<\omega}\prod_{i<n} b(i)$ and $\pi(\sigma)\in b(|\sigma|)$ for each $\sigma\in\dom(\pi)$. 
            \begin{align*}
				\eec_b(k) & \coloneqq\min\{|F|:F\subseteq\textstyle\prod b,\forall \pi\in \Pred_b~\exists f\in F~ \lnot(f\pck\pi)\},\\
                \eec_b & \coloneqq\min\{|F|:F\subseteq\textstyle\prod b,\forall \pi\in \Pred_b~\exists f\in F~ \lnot(f\pc\pi)\},\\
				\vvc_b(k) & \coloneqq\min\{|\Pi|:\Pi\subseteq \Pred_b,\forall f\in\textstyle\prod b~\exists\pi\in \Pi~ f \pck \pi\},\\
                \vvc_b & \coloneqq\min\{|\Pi|:\Pi\subseteq \Pred_b,\forall f\in\textstyle\prod b~\exists\pi\in \Pi~ f \pc \pi\}.
		      \end{align*}
            When $b$ is the constant function with value $n<\omega$, namely, when $b=\omega\times\{n\}$ for some $n<\omega$, we use $n$ as subscripts instead of the function $b$ itself.
        \end{itemize}
    \end{dfn}

    One easily sees:
    \begin{fac}\leavevmode
        \begin{enumerate}
            \item If $2\leq k\leq l$, then
            \[\eeck\leq\eec(l)\leq\eec\text{ and }\vvc\leq\vvc(l)\leq\vvck.
            \]The same inequalities hold for constant evasions/predictions on $\prod b$ for any $b\in(\omega\setminus2)^\omega$.
            \item If $b\leq^*b^\prime$, then
            \begin{align*}
                \eec\leq \eec_{b^\prime}\leq \eec_b\leq\eec_2,\\
                \vvc_2\leq\vvc_b\leq\vvc_{b^\prime}\leq\vvc.
            \end{align*}
            Also, for $k\geq 2$,
            \begin{align*}
                \eec(k)\leq \eec_{b^\prime}(k)\leq \eec_b(k)\leq\eec_2(k),\\
                \vvc_2(k)\leq\vvc_b(k)\leq\vvc_{b^\prime}(k)\leq\vvc(k).
            \end{align*}
        \end{enumerate}
    \end{fac}
  
    \begin{prop}\label{prop:eect_nonMed}
    	$\eect\leq\non(M_{\ed})$ and $\vvct\geq\cov(M_{\ed})$.
    \end{prop}
    \begin{proof}
    	To show $\eect\leq\non(M_{\ed})$, let $F\subseteq(\omega\times\omega)^\omega$ of size $<\eect$. For $(f,g)\in F$, $f*g\in\oo$ is given by $f*g(2n)\coloneqq f(n)$ and $f*g(2n+1)\coloneqq g(n)$. Put $F^*\coloneqq\{f*g:(f,g)\in F\}\subseteq\oo$. Since $F^*$ has size $<\eect$, there is a predictor $\pi\colon\fsq\to\omega$ such that $f*g\pck\pi$ for all $(f,g)\in F$. In particular,
    	\begin{equation*}
    		\forall (f,g)\in F~\forall^\infty n<\omega\text{ either }
    		\begin{cases}
    			f(n)=\pi((f*g)\on 2n), & \text{or} \\
    			g(n)=\pi((f*g)\on 2n+1).         &
    		\end{cases}
    	\end{equation*}
    	Define $\sigma\colon(\omega\times\omega)^{<\omega}\to\omega$ by $\sigma(u,v)\coloneqq\pi(u*v)$. For $(u,v)\in (\omega\times\omega)^{<\omega}$, define $s(u,v)\colon\omega\to\omega$ by $s(u,v)(a)\coloneqq\pi((u*v)^\frown a)$. Again for $(u,v)\in (\omega\times\omega)^{<\omega}$, let $\phi(u,v)\coloneqq\{(a,b)\in\omega\times\omega:a\leq\sigma(u,v)\text{ or }b=s(u,v)(a)\}$. Note $\phi(u,v)\in\ed$. Then it is routine to check $(f,g)\in M_\phi$ for all $(f,g)\in F$. The latter inequality is proved in the dual manner.
    \end{proof}

    By the essentially same proof as above, we have:

    \begin{prop}\label{prop:edfin_const}
        For any increasing $b\in\omega^\omega$, $\ee^{\mathrm{const}}_b(2)\leq\non(M_{\edfin})$ and $\vv^{\mathrm{const}}_b(2)\geq\cov(M_{\edfin})$.
    \end{prop}

    Moreover, we can even get characterizations of $\non(M_{\finfin})$ and $\cov(M_{\finfin})$ using the following variants of constant evasion and prediction:
    
    \begin{dfn}
        Let $f\in\omega^\omega$ and $\pi\in\Pred$. For each $k\geq 2$, we say \emph{$\pi$ $k$-constantly bounding predicts $f$}, denoted by $f\sqsubset^{\mathrm{pc}, k}_{\leq}\pi$, if
        \[
        \forall^\infty i<\omega\,\exists j\in[i, i+k)\,(f(j)\leq\pi(f\upharpoonright j)).
        \]
        Define:
        \begin{align*}
			\ee^\mathrm{const}_{\leq}(k) &\coloneqq\min\{|F|:F\subseteq\oo,\forall \pi\in \Pred~\exists f\in F~ \lnot(f\sqsubset^{\mathrm{pc}, k}_{\leq}\pi)\},\\
			\vv^\mathrm{const}_{\leq}(k) &\coloneqq\min\{|\Pi|:\Pi\subseteq \Pred,\forall f\in\oo~\exists\pi\in \Pi~ f \sqsubset^{\mathrm{pc}, k}_{\leq} \pi\}.\\
		\end{align*}
    \end{dfn}

    \begin{prop}\label{prop:non_M_finfin}
        $\non(M_{\finfin}) = \ee^\mathrm{const}_{\leq}(2)$ and $\cov(M_{\finfin}) = \vv^{\mathrm{const}}_{\leq}(2)$.
    \end{prop}
    
    To see this, it is convenient to introduce a seemingly weaker version of constant bounding prediction:
    
    \begin{dfn}
        Let $f\in\oo$ and $\pi\in\Pred$. For each $k\geq 2$, we say $\pi$ \emph{weakly $k$-constantly bounding predicts $f$}, denoted by $f\sqsubset^{\mathrm{wpc}, k}_{\leq}\pi$, if
        \[
        \forall^\infty m<\omega\,\exists j\in[km, km+k)\,(f(j)\leq\pi(f\upharpoonright j)).
        \]
        Define:
        \begin{align*}
			\ee^\mathrm{wconst}_{\leq}(k) &\coloneqq\min\{|F|:F\subseteq\oo,\forall \pi\in \Pred~\exists f\in F~ \lnot(f\sqsubset^{\mathrm{wpc}, k}_{\leq}\pi)\},\\
			\vv^\mathrm{wconst}_{\leq}(k) &\coloneqq\min\{|\Pi|:\Pi\subseteq \Pred,\forall f\in\oo~\exists\pi\in \Pi~ f \sqsubset^{\mathrm{wpc}, k}_{\leq} \pi\}.\\
		\end{align*}
    \end{dfn}

    \begin{lem}\label{lem:weakly_constant_prediction}
        $\ee^\mathrm{wconst}_{\leq}(k)=\ee^\mathrm{const}_{\leq}(k)$ and $\vv^\mathrm{wconst}_{\leq}(k)=\vv^\mathrm{const}_{\leq}(k)$.
    \end{lem}
    \begin{proof}
        $\ee^\mathrm{wconst}_{\leq}(k)\geq\ee^\mathrm{const}_{\leq}(k)$ is clear. To prove the converse inequality, let $F\subseteq\oo$ of size $<\ee^\mathrm{wconst}_{\leq}(k)$. For $f\in\oo$ and $l<k$, define the $l$-shift $f^l\in\oo$ of $f$ by:
        \begin{equation*}
    		f^l(n)=
    		\begin{cases}
    			0 & \text{if }n<l \\
    			f(n-l)         &\text{if }n\geq l
    		\end{cases}
    	\end{equation*}
        For $t\in\fsq$, define $t^l$ analogously (so $\dom(t^l)=l+|t|$). Let $F^\prime\coloneqq\{f^l:f\in F, l<k\}$. Since $F^\prime$ has size $<\ee^\mathrm{wconst}_{\leq}(k)$, there is $\pi^\prime\colon\fsq\to\omega$ such that $f^l\sqsubset^{\mathrm{wpc}, k}_{\leq}\pi^\prime$ for all $f\in F$ and $l<k$. Define $\pi\colon\fsq\to\omega$ by $\pi(t)=\max\{\pi^\prime(t^l):l<k\}$. 
        Let $f\in F$ be arbitrary.
        Since $f^l\sqsubset^{\mathrm{wpc}, k}_{\leq}\pi^\prime$ for all $l<k$, there is $m_0>0$ such that:
        \[\forall l<k~\forall m\geq m_0\,\exists j\in[km, km+k)\,(f^l(j)\leq\pi^\prime(f^l\upharpoonright j)).\]
        To see $f\sqsubset^{\mathrm{pc}, k}_{\leq}\pi$, let $i\geq km_0+k$ be arbitrary. Let $m\geq m_0$ and $l<k$ be such that $i=km-l$. Take $j\in [km, km+k)$ be such that $f^l(j)\leq\pi^\prime(f^l\on j)$.
        Thus $f(j-l)=f^l(j)\leq\pi^\prime(f^l\on j)=\pi^\prime((f\on (j-l))^l)\leq\pi(f\on (j-l))$.
        Since $j-l\in [km-l, km+k-l)=[i,i+k)$ and $i\geq km_0+k$ was arbitrary, we have $f\sqsubset^{\mathrm{pc}, k}_{\leq}\pi$.
    \end{proof}

    \begin{proof}[Proof of \Cref{prop:non_M_finfin}]
        By \Cref{lem:weakly_constant_prediction}, we show $\non(M_{\finfin}) = \ee^\mathrm{wconst}_{\leq}(2)$ instead.
    
        ($\geq$) Let $F\subseteq(\omega\times\omega)^\omega$ be of size $<\ee^\mathrm{wconst}_{\leq}(2)$. For $(f,g)\in (\omega\times\omega)^\omega$, $f*g\in\omega^\omega$ is given by $f*g(2n)= f(n)$ and $f*g(2n+1)= g(n)$. 
        For $(u,v)\in(\omega\times\omega)^{<\omega}$, define $u*v\in\omega^{<\omega}$ analogously.
        Put $F^*=\{f*g:(f,g)\in F\}\subseteq\omega^\omega$. Since $F^*$ has size $<\ee^\mathrm{wconst}_{\leq}(2)$, there is a predictor $\pi\colon\omega^{<\omega}\to\omega$ such that $f*g\sqsubset^{\mathrm{wpc}, 2}_{\leq}\pi$ for all $(f,g)\in F$. Thus, we have
    	\begin{equation*}
    		\forall (f,g)\in F~\forall^\infty n<\omega\text{ either }
    		\begin{cases}
    			f(n)\leq\pi((f*g)\upharpoonright 2n), & \text{or} \\
    			g(n)\leq\pi((f*g)\upharpoonright 2n+1).         &
    		\end{cases}
    	\end{equation*}
    	Define $\sigma\colon(\omega\times\omega)^{<\omega}\to\omega$ by $\sigma(u,v)=\pi(u*v)$. For $(u,v)\in (\omega\times\omega)^{<\omega}$, define $s(u,v)\colon\omega\to\omega$ by $s(u,v)(a)=\pi((u*v)^\frown a)$. Again for $(u,v)\in (\omega\times\omega)^{<\omega}$, let $\phi(u,v)=\{(a,b)\in\omega\times\omega:a\leq\sigma(u,v)\text{ or }b \leq s(u,v)(a)\}$. Note $\phi(u,v)\in\finfin$. Then it is routine to check $F\subset M_\phi$.
    
        ($\leq$) Let $F\subset\omega^\omega$ be of size $<\non(M_{\finfin})$. For each $f\in F$, define $f_\mathrm{even}, f_\mathrm{odd}\in\omega^\omega$ by $f_\mathrm{even}(n) = f(2n)$ and $f_\mathrm{odd}(n) = f(2n+1)$. Put $F_* = \{\langle f_\mathrm{even}, f_\mathrm{odd}\rangle:f\in F\}$. Since $F_*$ has size $<\non(M_{\finfin})$, there is $\phi\colon\omega^{<\omega}\to\finfin$ such that $F_*\subset M_{\phi}$. We may assume that there are $\sigma\colon(\omega\times\omega)^{<\omega}\to\omega$ and $s\colon(\omega\times\omega)^{<\omega}\to\omega^\omega$ such that $\phi(u, v) = \{(a, b): a\leq\sigma(u, v)\lor v\leq s(u, v)(a)\}$. Then for every $f\in F$,
        \[
        \forall^\infty n<\omega\text{ either }
    		\begin{cases}
    			f_\mathrm{even}(n)\leq\sigma(f_\mathrm{even}\upharpoonright n, f_\mathrm{odd}\upharpoonright n), & \text{or} \\
    			f_\mathrm{odd}(n)\leq s(f_\mathrm{even}\upharpoonright n, f_\mathrm{odd}\upharpoonright n)(f_\mathrm{even}(n)).         &
    		\end{cases}
        \]
        Now define $\pi\colon\omega^{<\omega}\to\omega$ by
        \begin{align*}
            \pi(t) &= \sigma(t_\mathrm{even}, t_\mathrm{odd})\\
            \pi(t^{\frown}\langle n\rangle) &= s(t_\mathrm{even}, t_\mathrm{odd})(n)
        \end{align*}
        for all $t\in\omega^{<\omega}$ of even length and all $n<\omega$, where $t_\mathrm{even},t_\mathrm{even}$ are defined analogously. Then for all $f\in F$, we have $f\sqsubset^{\mathrm{wpc}, 2}_{\leq}\pi$. $\vv^{\mathrm{const}}_b(2)\geq\cov(M_{\edfin})$ is proved in the dual manner.
    \end{proof}

    The above proof can be easily generalized to higher dimensions. Let $\Fin^{\otimes 2} = \finfin$. For each $2\leq k <\omega$, we define $\Fin^{\otimes k+1} = \Fin\otimes\Fin^{\otimes k}$ by induction.

    \begin{prop}
        For each $2\leq k<\omega$, $\non(M_{\mathrm{Fin}^{\otimes k}}) = \ee^\mathrm{const}_{\leq}(k)$ and $\cov(M_{\mathrm{Fin}^{\otimes k}}) = \vv^{\mathrm{const}}_{\leq}(k)$.
    \end{prop}
    
    Brendle \cite[Theorem 3.6(b)]{Bre03} showed the consistency of $\eec>\bb$ and $\vvc<\dd$ (with large continuum), but essentially proved more:
    
    \begin{thm}[Brendle \cite{Bre03}]\label{thm:const_evasion_b}\leavevmode
        \begin{enumerate}
            \item Given $\kappa<\lambda=\lambda^{<\lambda}$ regular uncountable, there is a poset forcing $\eect=\lambda=\continuum$ and $\bb=\kappa$.
            \item Given $\kappa$ regular uncountable and $\lambda=\lambda^{<\omega}>\kappa$, there is a poset forcing $\vvct=\kappa$ and $\dd=\lambda=\continuum$.
        \end{enumerate}
    \end{thm}
    
    Thus, we particularly have:
    
    \begin{cor}
    	Both $\bb<\non(M_{\ed})$ and $\cov(M_{\ed})<\dd$ are  consistent.
    \end{cor}

    Together with the following result, we show that the uniformity and covering numbers of $K_\I$ and $M_\I$ are consistently different for $\I = \mathrm{Fin}^{\otimes k}, \ed$.

    \begin{prop}\label{prop:K_finfin_b}
        For each $2\leq k<\omega$, $\non(K_{\mathrm{Fin}^{\otimes k}}) = \bb$ and $\cov(M_{\mathrm{Fin}^{\otimes k}}) = \dd$.
    \end{prop}
    \begin{proof}
        We only prove the former equation for $k = 2$, because our proof can be easily generalized to higher dimensions and the latter equation can be shown in a dual manner.
        
        It suffices to show $\non(K_{\finfin})\leq\bb$. Let $B\subseteq\oo$ be an unbounded family and assume all $f\in B$ are strictly increasing. We shall show $F\coloneqq\{f\times g:f,g\in B\}\subseteq\left(\omega\times\omega\right)^\omega$ is not in $K_{\finfin}$. Let $\phi\colon\omega\to\finfin$ be arbitrary. For each $n<\omega$, there are $k_n<\omega$ and $h_n\in\oo$ such that $\phi(n)\subseteq \rb{k_n\times\omega }\cup \{(i,j):i<\omega, j<h_n(i)\}$. Since $B$ is unbounded, there are $f\in B$ and $D\in\ooo$ such that $k_n<f(n)$ for $n\in D$. For $n<\omega$, put $h(n)\coloneqq h_n(f(n))$ and $h^\prime(n)\coloneqq h(d_n)$, where $d_n$ denotes the $n$-th element of $D$. Since $B$ is unbounded, there are $g\in B$ and $E\in\ooo$ such that $h^\prime(n)<g(n)$ for $n\in E$. Then, for any $n\in E$, $g(d_n)\geq g(n)>h^\prime(n)=h(d_n)=h_{d_n}(f(d_n))$. Thus, for all $m\in\{d_n:n\in E\}$, we have $(f(m),g(m))\notin \phi(m)$, so $F$ is not in $K_{\finfin}$.
    \end{proof}

    By \Cref{lem:b_nonM} and \Cref{lem:Katetov_non}, we have:
    
    \begin{cor}
        For every ideal $\I$ on a countable set such that $\I \leq_{\mathrm{K}} \mathrm{Fin}^{\otimes k}$ for some $2\leq k<\omega$, $\non(K_{\I})=\bb$ and $\cov(K_{\I})=\dd$. In particular, these equalities hold for $\I = \RandomGraph,\conv,\ed$.
    \end{cor}
    
    By \Cref{thm:const_evasion_b}, we get the following consistency.

    \begin{cor}
        For every ideal $\I$ on a countable set such that $\I \leq_{\mathrm{K}} \mathrm{Fin}^{\otimes k}$ for some $2\leq k<\omega$, $\non(K_\I)<\non(M_\I)$ and $\cov(K_\I)>\cov(M_\I)$ are both consistent.
    \end{cor}

    We do not know other examples of ideals $\I$ such that $\non(K_\I)$ and $\non(M_\I)$ are consistently different. In particular, the consistency of $\non(K_{\edfin})<\non(M_{\edfin})$ is not known (Instead, we will see the consistency of $\non(K_{\edfin}) > \bb$ in \Cref{sec:posetEDfin}).

    \subsubsection{Nowhere dense ideal and Solecki ideal}

    By reformulating Sabok--Zapletal's results on forcing properties of $\mathbb{M}_\nwd$ and $\mathbb{M}_\Solecki$ in \cite{SZ11}, we compute the uniformity and covering numbers of $M_\nwd$ and $M_\Solecki$. Recall \Cref{lem:quasi}.
    
    \begin{dfn}\label{dfn:concat_func}
         Let $\mathsf{concat}\colon (\cantortree)^\omega \rightarrow \cantor$ be the concatenating function defined by $\mathsf{concat}(x) = x(0) ^\frown x(1) ^\frown x(2) ^\frown \cdots\in2^\omega$. We also define $\mathsf{concat}\colon (2^{<\omega})^{<\omega}\to 2^{<\omega}$ in the analogous way and confuse notations.
    \end{dfn}

    \begin{lem}[Sabok--Zapletal, \cite{SZ11}]\label{lem:bbmnwd_adds_Cohen}
         Let $\forces_{\mathbb{M}_{\nwd}}\dot{r}\coloneqq\bigcup\{\stem(T):T\in\dot{G}\}\in\cano$ be (the canonical $\mathbb{M}_{\nwd}$-name of) an $\mathbb{M}_{\nwd}$-generic real. Then
         \[
         \forces_{\mathbb{M}_{\nwd}} \mathsf{concat}(\dot{r})\in 2^\omega \text{ is Cohen generic.}
         \]
    \end{lem}
    \begin{proof}
         Let $N\subset 2^{\omega}$ be closed nowhere dense and let $S\subset 2^{<\omega}$ be a tree such that $[S] = N$.
         Let $T \in \mathbb{M}_{\nwd}$. We can now refine it to a tree $T' \in \mathbb{M}_\nwd$ such that, for all $\sigma \in T'$, $\mathrm{succ}_{T'}(\sigma)$ is either a singleton or a $\nwd$-positive set disjoint from $S\cap \{\tau \in 2^{<\omega} \colon \mathsf{concat}(\sigma) \subset \tau\}$.
         
         We claim that if $x \in [T']$ then $\mathsf{concat}(x) \notin N$; Let $\sigma$ be the stem of $T'$. Then $\mathsf{concat}(\sigma) ^\frown \langle x(\lvert\sigma\rvert)\rangle \notin S$. This implies that, for all $n > \lvert\sigma\rvert$, $\mathsf{concat}(x \restriction n) \notin S$, and therefore $\mathsf{concat}(x) \notin [S] = N$. 
    \end{proof}
    
    By \cref{lem:b_nonM} and \cref{lem:quasi}, we have:
    
    \begin{cor}\label{cor:Mnwd_nonm}
        $\non(\mathcal{M}) = \non(M_\nwd)$ and $\cov(\mathcal{M}) =\cov(M_\nwd)$.
    \end{cor}

    In \cite{SZ11}, it was shown that $\mathbb{M}_\Solecki$ collapses Lebesgue outer measure. \cref{cor:nonMSolecki} is a reformulation of this result. First, we recall that $\Solecki_\varepsilon$ is generated by the sets of the form 
    \[
    I_y = \{ C \in \Omega_{\varepsilon} \colon y \in C \},
    \]
    where $\Omega_\varepsilon$ is the collection of clopen subsets of $2^\omega$ of measure $\varepsilon$. Note that for every $A \subseteq \Omega_{\varepsilon}$,
    \[
    A\in \Solecki_\varepsilon^+ \iff \forall y \in 2^{\omega}~\exists C \in A~(y \notin C).
    \]
    The standard Solecki ideal $\Solecki$ is $\Solecki_{1/2}$. Note that $\cantor$ with the standard Haar measure $\lambda$ is isomorphic to $(2^\omega)^{k}$ with the product measure $\lambda_k$, so we may consider the clopen sets on $(2^\omega)^{k}$ instead so the respective Solecki ideal's are going to be Kat\v{e}tov equivalent. For each $k \in \omega$, we define $h_k\colon\Omega\to\Omega_\varepsilon$ by 
    \[
    h_k(C) = (2^\omega)^{k} \setminus \prod_{i<k} (2^\omega \setminus C)
    \]
    for any $C\subset 2^\omega$.

    \begin{lem}
        $h_k$ is a Kat\v{e}tov reduction witnessing that $\Solecki \geq_K \Solecki_{1 - 2^{-k}}$.
    \end{lem}
    \begin{proof}
        Let $ y \coloneqq \langle y_i \rangle_{i < k} \in (2^\omega)^k$ and let $C\in\Omega_\varepsilon$ be such that $y \in h_k(C)$. Then there must be an $i < k$ such that $y_i \notin \cantor \setminus C$, so $y_i \in C$. In other words, $h_k^{-1}[I_{y}]\subset\bigcup_{i<k}I_{y_i}\in\Solecki$.
    \end{proof}

     From now on we consider that $h_k$ is a function between the clopen sets of $\cantor$ of the respective measures (instead of the clopens on ${\cantor}^k$). Given $x \in \Omega^\omega$, define
     \[
     B_x = \{ y \in \cantor : \exists ^\infty n < \omega~( y \notin h_n (x(n))) \}.
     \]
     Note that $B_x$ has measure zero, since each $h_n(x(n))$ has measure $1 - 2^{-n}$.

    \begin{lem}\label{lem:Soloecki_B_x}
        Let $y \in \cantor$ and let $T \in \mathbb{M}_\Solecki$, then there is a tree $T' \in \mathbb{M}_\Solecki$ such that $T' \leq T$, and for all $x \in [T']$, $y \in B_x$.
    \end{lem}
    \begin{proof}
        The argument is similar to the proof of \Cref{lem:bbmnwd_adds_Cohen}. We can refine $T$ to a tree $T'\in\mathbb{M}_\Solecki$ such that, for all $\sigma \in T'$, $\mathrm{succ}_{T'}(\sigma)$ is either a singleton or a $\Solecki$-positive set disjoint from $h_{\lvert\sigma\rvert}^{-1}[I_y]$. To show that $T'$ has the desired property, let $x \in [T']$. For any $n<\omega$ such that $\mathrm{succ}_{T'}(x\upharpoonright n)\in\Solecki^+$, we have $x(n)\notin h_{\lvert\sigma\rvert}^{-1}[I_y]$ and thus $y\notin h_n(x(n))$. There are infinitely many such $n<\omega$, so $y\in B_x$.
    \end{proof}
    
    Combined with a standard fusion argument, we can strengthen this lemma as follows.

    \begin{lem}\label{lem:Solecki_omega_cover}
        Let $\{ y_n : n \in \omega \} \subseteq \cantor$ and let $T \in \mathbb{M}_\Solecki$, then there is a tree $T' \in \mathbb{M}_\Solecki$, $T' \leq T$ such that, for all $x \in [T']$ and all $n \in \omega$, $y_n \in B_x$.
    \end{lem}
    \begin{proof}
        For a tree $T$, we write
        \[
        \mathrm{lv}_n (T) = \{ t \in T : t \text{ passes through at most }n \text{ splitting nodes} \}.
        \]
        We also write $T' \leq_n T$ if $T' \leq T$ and $\mathrm{lv}_n (T) = \mathrm{lv}_n(T')$. 
    
        \begin{claim}
            Let $T \in \mathbb{M}_\Solecki$ and $n \in \omega$, then there is a tree $T' \in \mathbb{M}_\Solecki$ such that $T' \leq_n T$, and for all $x \in [T']$, $y_n \in B_x$.
        \end{claim}
        \begin{proof}[Proof of the claim.]
            For each finial node $s$ of $\mathrm{lv}_n(T)$, we apply \Cref{lem:Soloecki_B_x} to $T_s$ to obtain a tree $T'_s\in\mathbb{M}_\Solecki$ such that $T'_s \leq T_s$ and for all $x \in [T_s']$, $y_n \in B_x$. The tree $T' = \bigcup \{ T'_s : s \text{ is a final node of lv}_n(T) \}$ is the tree we are looking for.
        \end{proof}
        
        We apply repeatedly the claim to get a sequence of trees $T = T_0 \geq_0 T_1 \geq_1 T_2 \geq_2 \cdots$ such that, for all $n<\omega$, if $x \in [T_n]$ then $y_n \in B_x$. Then the tree $T' \coloneqq \bigcap_{n \in \omega}T_n$ satisfies the requirements.
    \end{proof}

    \begin{cor}\label{cor:nonMSolecki}
        $\covo{\mathcal{N}} \leq \non(M_\Solecki)$ and $\cov(M_\Solecki) \leq \non(\mathcal{N})$.
    \end{cor}
    \begin{proof}
        If $X \notin M_\Solecki$, then there is a tree $T \in \mathbb{M}_\Solecki$ such that every $T' \in \mathbb{M}_\Solecki$ with $T' \leq T$ has the property that $[T'] \cap X \neq \emptyset$. Then, \Cref{lem:Solecki_omega_cover} implies that $\{ B_x : x \in X \}$ witnesses $\covomn\leq|X|$. 
        The other inequality follows dually: Let $\kappa < \cov(M_\Solecki)$ and $\{y_\alpha : \alpha \in \kappa\} \in \cantor$. Define $Y_\alpha \coloneqq \{ x \in \Omega^\omega : y_\alpha \notin B_x \}$. By the previous lemma $Y_\alpha \in M_\Solecki$ and, since $\kappa < \cov(M_\Solecki)$, there must be $x \notin \bigcup_{\alpha < \kappa} Y_\alpha$. It follows that $y_\alpha \in B_x$ for each $\alpha < \kappa$.
    \end{proof}

    By \Cref{prop:Solecki_non_omega} and \Cref{thm:nonmi_leq_max}, we get the exact values of $\non(M_\mathcal{S})$ and $\cov(M_\mathcal{S})$:
    
    \begin{thm}\label{thm:M_S_exact}
        $\non(M_\mathcal{S}) = \max\{\mathfrak{b}, \cov_\omega(\mathcal{N})\}$ and $\cov(M_\mathcal{S}) = \min\{\mathfrak{d}, \non(\mathcal{N})\}$.
    \end{thm}

    \begin{rem} \label{rem:Solecki_optimal}
        This implies that the inequalities of \Cref{thm:nonmi_leq_max} are sharp: Indeed, when $\I=\Solecki$, $\non(M_\mathcal{S}) =\max\{\mathfrak{b}, \non^*_\omega(\Solecki)\}$ by \Cref{prop:Solecki_non_omega}.
    \end{rem}

    \begin{rem}
         $\bb<\covn\leq\covomn$ holds in the random model, and it is not hard to see that $\cov_\omega(\mathcal{N})<\bb$ holds in the Laver model (see e.g. \cite[Lemma 7.2.3]{BJ95}). Thus $\bb$ and $\covomn$ have independent values. As in the case of Laver model, one can see that in the $\mathbb{PT}_{f,g}$-model, $\bb=\cov_\omega(\mathcal{N})<\nonm$ holds (see e.g. \cite[Lemma 7.2.15]{BJ95}). Therefore, $\non(M_\Solecki) = \max\{\mathfrak{b}, \cov_\omega(\mathcal{N})\}$ is consistently different from $\bb$ and $\nonm$.
    \end{rem}

    \Cref{thm:M_S_exact} can be improved for $K_\Solecki$ and we will prove this using \Cref{lem:quasi}. Let us introduce the forcing notion $\mathbb{K}_{\I}$ that corresponds to the $\sigma$-ideal $K_\I$:
    
    \begin{dfn}
        Let $\mathbb{K}_{\I}$ denote the collection of all $T\subseteq\baire$ such that for every $\sigma\in T$ there is $N>|\sigma|$ such that the spectrum
        \[
        \Spec_T(\tau,N)\coloneqq\{\tau(N):\tau\in T, \sigma\subseteq\tau, |\tau|=N+1\}
        \]
        is $\I$-positive.
    \end{dfn}

    The following dichotomy holds:

    \begin{lem}
        For any analytic set $A\subseteq\omega^{\omega}$ either $A\in K_{\mathscr{I}}$ or there is $T\in\mathbb{K}_{\I}$ such that $[T]\subseteq A$. Therefore, $\mathbb{K}_\I$ is forcing equivalent to $\mathbb{P}_{K_\I}$.
    \end{lem}
    \begin{proof}
        We will show this for closed sets only. The standard unfolding trick gives the proof for analytic sets. Let $D\subseteq\omega^{\omega}$ be a closed set and let $T\subseteq\omega^{<\omega}$ be such a tree that $D=[T]$. We define a derivative of $T$ as follows:
        \begin{center}
            $T'=\{\sigma\in T:\exists N>|\sigma|$ such that $\{\tau(N):\tau\in T, \sigma\subseteq\tau, |\tau|=N+1\}\in\I^{+}\}$
        \end{center}
        Define $T_{0}=T$, $T_{\alpha+1}=T_{\alpha}'$ and for limit $\alpha<\omega_{1}$ let $T_{\alpha}=\bigcap_{\gamma<\alpha}T_{\gamma}$. Let $\alpha^{*}<\omega_{1}$ be the smallest ordinal such that $T_{\alpha^{*}+1}=T_{\alpha^{*}}$. Consider two cases:\\
        First, $T_{\alpha^{*}}\neq\emptyset$. Then by the definition of derivative, it is easy to construct required tree.\\
        Second, $T_{\alpha^{*}}=\emptyset$. Then $[T]=\bigcup_{\alpha<\alpha^{*}}([T_{\alpha}]\setminus[T_{\alpha+1}])$. By the definition of the derivative we have that $[T_{\alpha}]\setminus[T_{\alpha+1}]\in K_{\mathscr{I}}$. It follows that $[T]$ is covered by countable union of $K_{\mathscr{I}}$-sets. 
    \end{proof}

    \begin{dfn}
        For $y\in2^\omega$, let $N_y\coloneqq\{x\in\Omega^\omega:y\notin B_x\}$. 
        Let $N_\Solecki$ be the $\sigma$-ideal on $\Omega^\omega$ generated by $N_y$'s.
    \end{dfn}
    
    \begin{lem}
        $\covomn\leq\non(N_\Solecki)$ and $\cov(N_\Solecki)\leq\nonn$.
    \end{lem}
    \begin{proof}
        To see $\covomn\leq\non(N_\Solecki)$, let $X\subseteq\Omega^\omega$ with $X\notin N_\Solecki$. We claim that $\{B_x:x\in X\}\subseteq\n$ is a witness for $\covomn$. To show this, let $Y\in[2^\omega]^\omega$. Since $N\coloneqq\bigcup_{y\in Y}N_y\in N_\Solecki$ and $X\notin N_\Solecki$, there is $x\in X\setminus\bigcup_{y\in Y}N_y$. Then we have $Y\subseteq B_x$.
        
        The other inequality $\cov(N_\Solecki)\leq\nonn$ can be proved in the dual manner.
    \end{proof}

    \begin{lem}
        Let $\forces_{\KS}\dot{r}\coloneqq\bigcup\{\stem(T):T\in\dot{G}\}\in\Omega^\omega$ be (the canonical $\KS$-name of) a $\KS$-generic real. Then
        \[
        \forces_{\KS}\dot{r}\text{ is $N_\Solecki$-quasi-generic.}
        \]
    \end{lem}
    \begin{proof} 
        The proof is essentially the same as \Cref{lem:Soloecki_B_x}.
        Given $y\in2^\omega$ and $T\in\KS$, refine $T$ to $T'$ such that for all $s \in T'$, there is $N\geq\lvert\sigma\rvert$ such that $\Spec_{T'}(\sigma, N)$ is $\Solecki$-positive and disjoint from $h_{\lvert\sigma\rvert}^{-1}[I_y]$. Then $x \in [T']$ implies $y \in B_x$. This means $\forces_{\KS}\dot{r}\in\Omega^\omega$ is $N_\Solecki$-quasi-generic.
    \end{proof}

    By \Cref{lem:quasi}, we have:
    \begin{thm}\label{thm:K_S_exact}
        $\non(K_\mathcal{S}) = \max\{\mathfrak{b}, \cov_\omega(\mathcal{N})\}$ and $\cov(K_\mathcal{S}) = \min\{\mathfrak{d}, \non(\mathcal{N})\}$.
    \end{thm}

    \subsubsection{Asymptotic density zero ideal}

    Pawlikowski \cite{Paw00} studied $K_\Z$ and proved:
    
    \begin{lem}[{\cite[Lemma 2.3]{Paw00}}]
    	\label{lem:paw00_max}
    	There is a (continuous) function $f\colon \oo\times2^\omega\to\oo$ such that for any $F\in K_\Z$, there is $g\in\oo$ such that for any $x\in\oo$ unbounded from $g$, $\set{y\in2^\omega}{f(x,y)\in F}\in\mathcal{E}$. 
    \end{lem}
        
    \begin{cor}\label{cor:paw00_ineq}
    	$\non(K_\Z)\leq\max\{\bb,\none\}$ and $\cov(K_\Z)\geq\min\{\dd,\cove\}$.
    \end{cor}
    \begin{proof}
        We will only prove the former inequality, as the other one follows dually. Let $B\subseteq\oo$ and $E\subseteq2^\omega$ witness $\bb$ and $\none$, respectively. We shall show $f[B\times E]\notin K_\Z$ where $f$ is as above. If not, let $g\in\oo$ as in Lemma \ref{lem:paw00_max}. Since $B$ is unbounded, some $x\in B$ satisfies $\mathcal{E}\ni \set{y\in2^\omega}{f(x,y)\in f[B\times E]}\supseteq E\notin \mathcal{E}$, a contradiction. 
    \end{proof}

    We prove the $M_\Z$-version of \Cref{lem:paw00_max}, in the following form:
    
    \begin{lem}\label{lem:MZ_max}
		There is a (continuous) function $f\colon \oo\times2^\omega\to\oo$ such that for any $F\in M_\Z$, there is $g\colon\fsq\to\omega$ such that for any $x\in\oo$, $x$ is either adaptively dominated by $g$ or  $\set{y\in2^\omega}{f(x,y)\in F}\in\mathcal{E}$. 
	\end{lem}

    In the following proof, note that for any $A\in\Z$, there is $k_0<\omega$ such that for any $k\geq k_0$,
    \[
    \frac{|A\cap \left[k,k+2^k\right)|}{2^k}\leq\frac{1}{2}.
    \]
    Also, for a set $u$ of partial functions from $\omega$ to $2$, define
    \[
    [u]\coloneq\bigcup\{[\sigma]:\sigma\in u\}=\bigcup\{x\in2^\omega:\exists\sigma\in u\,(\sigma\subseteq x)\}.
    \]
    
	\begin{proof}
		Fix a function $e\colon\bigcup\{2^K:K\in[\omega]^{<\omega}\}\to\omega$ such that for each $K\in[\omega]^{<\omega}$, $e$ on $2^K$ is bijective to $\left[|K|, |K|+2^{|K|}\right)$. For $x\in\oo$, let $\langle I_{x,i}:i\in\omega\rangle$ denote the interval partition of $\omega$ of each length $x(i)+1$. Define
        \[
        f(x,y)\coloneq\langle e(y\on I_{x,i}):i\in\omega \rangle.
        \]
        Note that $f(x,y)=e(y\on I_{x,i})\geq |I_{x,i}|=x(i)+1$. To see this $f$ works, let $F\in M_\Z$ be arbitrary and take $\phi\colon\fsq\to\Z$ such that $F\subseteq M_\phi$. Since $\Z$ is a P-ideal, there is $S\colon\omega\to\Z$ such that $\phi(t)\subseteq^* S(i)$ for any $i\in\omega$ and $t\in\omega^i$, and take $h\colon\fsq\to\omega$ such that $\phi(t)\subseteq S(i) \cup h(t)$ for such $i$ and $t$.
		Choose $z\in\oo$ such that for any $i\in\omega$ and $k\geq z(i)$, 
		\[\frac{|S(i)\cap\left[k,k+ 2^{k}\right)|}{2^{k}}\leq\frac{1}{2}.\]
		Let $m\colon\fsq\to\omega$ be such that for any $x\in\oo$ and $y\in2^\omega$,
		\[m(x\on i)=\max\{h(f(x,y)\on i):y\in2^\omega\},\]
		which is possible since only $x\on i$ and finitely many $y$ are relevant here.
		
		Define $g\colon\fsq\to\omega$ by $g(t)\coloneqq\max\{z(i),m(t)\}$ for $i\in\omega$ and $t\in\omega^i$. Assume $x\in\oo$ is not adaptively dominated by $g$, that is, there is $D\in\ooo$ such that for any $i\in D$, $x(i)>g(x\on i)$. For $i\in\omega$ put
        \[
        S^*_{x,i}\coloneqq\{\sigma\in2^{I_{x,i}}:e(\sigma)\in S(i)\}.
        \]
        Since $e$ on $2^{I_{x,i}}$ is bijective to $\left[x(i)+1, x(i)+1+2^{x(i)+1}\right)$, 
        \begin{equation}\label{eq:leq_half}
            \text{ if }i\in D,~\frac {|S^*_{x,i}|}{2^{|I_{x,i}|}}\leq\frac{|S(i)\cap\left[x(i)+1,x(i)+1+ 2^{x(i)+1}\right)|}{2^{x(i)+1}}\leq\frac{1}{2},
        \end{equation}
		since $x(i)+1>x(i)> g(x\on i)\geq z(i)$.
		Thus, it suffices to show that:
		\begin{equation}\label{eq:Y_D}
			\{y\in 2^\omega:f(x,y)\in M_\phi\}\subseteq\bigcup_{j\in\omega}\bigcap_{j\leq i\in D}[S_{x,i}^*],
		\end{equation} 
        since the latter set is $F_\sigma$ null in $2^\omega$ by \eqref{eq:leq_half}.
		To prove \eqref{eq:Y_D}, let $y\in 2^\omega$ be such that $f(x,y)\in M_\phi$. Take $j\in\omega$ such that for $i\geq j$, 
		\[
        f(x,y)(i) \in \phi(f(x,y)\on i)\subseteq S(i)\cup h(f(x,y)\on i).
        \]
		By the choice of $m$, for such $i$, either $f(x,y)(i) \in  S(i)$ or $f(x,y)(i) \in  m(x\on i)$. If $j\leq i\in D$, then
        \[
        f(x,y)(i)\geq x(i)+1>g(x\on i)\geq m(x\on i),
        \]
        so the former case $f(x,y)(i) \in  S(i)$ holds, which means $f(x,y)(i)=e(y\on I_{x,i})\in S(i)$. Thus, $y\on I_{x,i}\in S^*_{x,i}$ and hence $y\in[S_{x,i}^*]$.
		Therefore, we obtain \eqref{eq:Y_D}.
	\end{proof}

    \begin{cor}\label{cor:nonMZ_leq_max}
    	$\non(M_\Z)\leq\max\{\bb,\none\}$ and $\cov(M_\Z)\geq\min\{\dd,\cove\}$.
    \end{cor}
    \begin{proof}
        We will only prove the first inequality. Let $B\subseteq\oo$ and $E\subseteq2^\omega$ witness $\bb$ and $\none$, respectively. We shall show $F\coloneqq f(B\times E)\notin M_\Z$ where $f$ is as above. If not, let $g\colon\fsq\to\omega$ as in Lemma \ref{lem:MZ_max}. By Lemma \ref{lem:seq_unbounded}, some $x\in B$ is not adaptively dominated by $g$. However, it follows $\mathcal{E}\ni \set{y\in2^\omega}{f(x,y)\in  F=f(B\times E)}\supseteq E\notin \mathcal{E}$, a contradiction. 
    \end{proof}
    
    In the following lemma, we consider $\Z$ as an ideal on $\sq$, so for any $A\subseteq\sq$,
    \[
    A\in\Z\iff\lim_{k\to\infty}\frac{|A\cap 2^k|}{2^k}=0.
    \]
    Recall the concatenating function $\mathsf{concat}$ in \Cref{dfn:concat_func}.
    
    \begin{lem}
    \label{lem:bbmz_adds_PEg}
        Let $\forces_{\bbmz}\dot{r}\coloneqq\bigcup\{\stem(T):T\in\dot{G}\}\in\cano$ be (the canonical $\bbmz$-name of) an $\bbmz$-generic real. Then $\forces_{\bbmz} \mathsf{concat}(\dot{r})\in 2^{\omega}$ is $\mathcal{E}$-quasi-generic.
    \end{lem}
    \begin{proof}
        We prove that for any tree $T\subseteq\sq$ such that $[T]\subseteq2^\omega$ is null, $\forces \mathsf{concat}(\dot{r})\notin[T]$. To see this, let $p\in\bbmz$ be arbitrary and we may assume that every node of $p$ is either non-splitting or $\Z^+$-splitting. Let $\sigma$ be the stem of $p$ and $A\coloneqq\suc_p(\sigma)\subseteq\sq$. Since $A\in\Z^+$, there is $\varepsilon\in(0,1)\cap\mathbb{Q}$ such that for infinitely many $k<\omega$: 
        \begin{equation}\label{eq:Ageq_varepsilon}
            \frac{|A\cap 2^k|}{2^k}\geq\varepsilon.
        \end{equation}
        Put $s\coloneqq \mathsf{concat}(\sigma)\in\sq$. Assume on the contrary that for infinitely many $k<\omega$:
        \begin{equation}\label{eq:Tgeq_varepsilon}
            \frac{|T\cap[s]\cap2^{|s|+k}|}{2^k}\geq\varepsilon.
        \end{equation}
        Since the left-hand side is non-increasing with respect to $k$, \eqref{eq:Tgeq_varepsilon} holds for all $k<\omega$. However, it implies $\frac{\mu([T])}{\mu([s])}\geq\varepsilon>0$ where $\mu$ denotes the Lebesgue measure on $2^\omega$, which contradicts that $[T]$ is null. Hence, there exists $k_0<\omega$ such that for all $k\geq k_0$, 
        \begin{equation}\label{eq:Tlessthan_varepsilon}
            \frac{|T\cap[s]\cap2^{|s|+k}|}{2^k}<\varepsilon.
        \end{equation}
        Take $k\geq k_0$ satisfying \eqref{eq:Ageq_varepsilon} and put $\Sigma\coloneqq\{\sigma^\frown a:a\in A\cap 2^k\}$. Note that $\mathsf{concat}$ is one-to-one on $\Sigma$ and $\mathsf{concat}[\Sigma]\subseteq [s]\cap 2^{|s|+k}$. By \eqref{eq:Ageq_varepsilon} and \eqref{eq:Tlessthan_varepsilon},
        \begin{equation*}
            |T\cap[s]\cap 2^{|s|+k}|<\varepsilon\cdot2^k\leq|A\cap 2^k|=|\mathsf{concat}[\Sigma]|,
        \end{equation*}
        so we have $\mathsf{concat}[\Sigma]\nsubseteq T\cap[s]\cap 2^{|s|+k}$.
        Take $a\in A\cap 2^k$ such that
        \[
        \mathsf{concat}(\sigma^\frown a)\notin T\cap[s]\cap 2^{|s|+k}
        \]
        and let $p^\prime\coloneqq p\cap[\sigma^\frown a]$. Then, we have $p^\prime\forces \mathsf{concat}(\dot{r})\notin [T]$, since $p^\prime\forces \mathsf{concat}(\dot{r})\on(|s|+k)=\mathsf{concat}(\sigma^\frown a)\notin T$.
    \end{proof}
    
    \begin{thm} \label{thm:nonMZ_exact_value}
        $\non(M_{\Z})\geq\none$ and $\cov(M_{\Z})\leq\cove$. Therefore,
        \[
        \non(M_\Z)=\max\{\bb,\none\}\text{ and }\cov(M_\Z)=\min\{\dd,\cove\}.
        \]
    \end{thm}
    \begin{proof}
        It follows from \Cref{cor:nonMZ_leq_max}, \Cref{lem:bbmz_adds_PEg} and \Cref{lem:quasi}.
    \end{proof}
    
    As a corollary of \Cref{thm:nonMZ_exact_value}, we obtain new bounds of $\non^*(\Z)$ and $\cov^*(\Z)$. We quickly recall previous results about these cardinal invariants.
    
    \begin{thm}[Hern\'andez-Hern\'andez--Hru\v s\'ak \cite{HH07}]\leavevmode
        \begin{align*}
            \min\{\dd, \cov(\n)\}\leq & \non^*(\Z)\leq\max\{\dd, \non(\n)\},\\
            \min\{\bb, \cov(\n)\}\leq & \cov^*(\Z)\leq\max\{\bb, \non(\n)\}.
        \end{align*}
    \end{thm}
    
    \begin{thm}[Raghavan--Shelah \cite{RS17}]
        \label{fac:nonZ_covZ_bounds}
        $\bb\leq\non^*(\Z)$ and $\cov^*(\Z)\leq\dd$.
    \end{thm}
    
    The lower bounds for $\non^*(\Z)$ and the upper bounds for $\cov^*(\Z)$ above were improved by Raghavan:
    
    \begin{thm}[Raghavan \cite{Rag20}]\label{fac:nonZ_covZ_bounds_2}
        \[
        \min\{\dd,\uu\}\leq\non^*(\Z)\text{ and }\cov^*(\Z)\leq\max\{\bb,\mathfrak{s}(\mathfrak{pr})\},
        \]
        where $\uu$ denotes the ultrafilter number and $\mathfrak{s}(\mathfrak{pr})$ is a variant of the splitting number $\mathfrak{s}$ introduced in \cite{Rag20}.\footnote{By definition, $\mathfrak{s}\leq\mathfrak{s}(\mathfrak{pr})$ but it is unknown to be distinguishable from $\mathfrak{s}$.}
    \end{thm}
    
    \begin{thm}\label{thm:new_bounds}
        $\none\leq\non^*(\Z)$ and $\cov^*(\Z)\leq\cove$. Also, $\cov^*(\Z)\leq\non^*(\Z)$.
    \end{thm}
    
    \begin{proof}
        Since $\Z$ is a P-ideal, $\non^*_\omega(\Z)=\non^*(\Z)$. By \Cref{thm:nonMZ_exact_value}, \Cref{thm:nonmi_leq_max} and \Cref{fac:nonZ_covZ_bounds}, we have
        \[
        \none\leq\non(M_{\Z})\leq\max\{\bb,\non^*_\omega(\Z)\}=\max\{\bb,\non^*(\Z)\}=\non^*(\Z).
        \]
        Similarly,
        \[
        \cov(\mathcal{E})\geq\cov(M_\Z)\geq\min\{\dd, \cov^*(\Z)\}=\cov^*(\Z).
        \]
        Since $\cove\leq\rr\leq\uu$, using \Cref{fac:nonZ_covZ_bounds} as well, we have
        \[
        \cov^*(\Z) \leq \min\{\dd, \cove\} \leq \min\{\dd, \uu\} \leq \non^*(\Z).
        \]
    \end{proof}

    We finish this section by remarking that the cardinal invariants $\min\{\dd,\uu\}$ and $\none$ are independent, so are $\max\{\bb,\mathfrak{s}(\mathfrak{pr})\}$ and $\cove$: 
    \begin{itemize}
        \item $\min\{\dd,\uu\}>\none$ holds in the Cohen model. 
        \item $\min\{\dd,\uu\}\leq \uu<\mathfrak{s}\leq\none$ holds in the Blass-Shelah model(\cite{BS87}). (Or: $\min\{\dd,\uu\}\leq\dd<\none$ holds in the model constructed in \cite{She92} (see also \cite[Section 3.4, Theorem 1]{Bre95}, it is easy to see $\ee_{ubd}\leq\ee_{fin}=\ee_2\leq\none$).)
        \item $\max\{\bb,\mathfrak{s}(\mathfrak{pr})\}\geq\bb>\cove$ holds in the Laver model.
        \item $\max\{\bb,\mathfrak{s}(\mathfrak{pr})\}\leq\dd<\covn\leq\cove$ holds in the random model.
    \end{itemize}

    \section{Consistency results}\label{sec:Consistency}
    
    In this section, we study consistency results for cardinal invariants associated with the $\sigma$-ideals $M_\I$ and $K_\I$. We conclude by constructing a model of  Cicho\'n's maximum in \Cref{mainthm:ext_CM}.
    
    \subsection{Forcing theory}
    
    We list the necessary facts on relational systems, and some preservation theorems that will be used to prove our consistency results.
    
    \begin{dfn}\leavevmode
        A triple $\R=\langle X, Y, \sqsubset\rangle$ is called a \emph{relational system} if $X$ and $Y$ are non-empty sets and $\sqsubset$ is a relation from $X$ to $Y$. Elements of $X$ are called \emph{challenges}, and elements of $Y$ are \emph{responses}. We say that \emph{$x$ is met by $y$} if $x\sqsubset y$.
    	\begin{itemize}
    		\item A subset $F\subseteq X$ is \emph{$\R$-unbounded} if no response meets all challenges in $F$, i.e., $\neg\exists y\in Y~\forall x\in F~(x\sqsubset y)$.
    		\item A subset $F\subseteq Y$ is \emph{$\R$-dominating} if every challenge is met by some response in $F$, i.e., $\forall x\in X~\exists y\in F~(x\sqsubset y)$.
    		\item We say that $\R$ is \emph{non-trivial} if $X$ is $\R$-unbounded and $Y$ is $\R$-dominating.
            \item Two cardinal invariants associated to $\R$ are introduced as follows:
    		\begin{align*}
    			\bb(\R) &= \min\{|F|\colon F\subseteq X \text{ is }\R\text{-unbounded}\},\\
    			\dd(\R) &= \min\{|F|\colon F\subseteq Y \text{ is }\R\text{-dominating}\}.
    		\end{align*}	
    	\end{itemize}	
    \end{dfn}
    
    Hereafter, we assume that a relational system is always non-trivial.
    		
    \begin{dfn}
        Let $\R=\langle X,Y,\sqsubset\rangle$ be a relational system. The \emph{dual of $\R$} is the relational system $\R^\bot\coloneq\langle Y,X,\sqsubset^\bot\rangle$, where $y\sqsubset^\bot x$ if and only if $\lnot(x\sqsubset y)$.
    \end{dfn}
    		
    \begin{dfn}\label{dfn:Tukey_connection}
    	Let $\R=\langle X,Y,\sqsubset \rangle, \R^{\prime}=\langle X^{\prime},Y^{\prime},\sqsubset^{\prime}~\rangle$ be two relational systems. We say that $(\Phi_-,\Phi_+)\colon\R\rightarrow\R^\prime$ is a \emph{Tukey connection from $\R$ into $\R^{\prime}$} if $\Phi_-\colon X\rightarrow X^{\prime}$ and $\Phi_+\colon Y^{\prime}\rightarrow Y$ are functions such that
    	\begin{equation*}
    		\forall x\in X~\forall y^{\prime}\in Y^{\prime}~\Phi_-(x)\sqsubset^{\prime} y^{\prime}\Rightarrow x \sqsubset \Phi_{+} (y^{\prime}).
    	\end{equation*}
    	We write $\R\preceq_T\R^{\prime}$ if there is a Tukey connection from $\R$ into $\R^{\prime}$ and call $\preceq_T$ the Tukey order. Tukey equivalence is defined by $\R\cong_T\R^{\prime}$ if and only if $\R\preceq_T\R^{\prime}\land \R^{\prime}\preceq_T\R$.
    \end{dfn}
    		
    \begin{fac}\label{Tukey order and b and d}
        Let $\R, \R^\prime$ be relational systems.
    	\begin{enumerate}
    		\item $\R\preceq_T\R^{\prime}$ implies $(\R^{\prime})^\bot\preceq_T\R^\bot$.
    		\item $\R\preceq_T\R^{\prime}$ implies $\mathfrak{b}(\R^{\prime})\leq\mathfrak{b}(\R)$ and $\mathfrak{d}(\R)\leq\mathfrak{d}(\R^{\prime})$.
    		\item $\bb(\R^\bot)=\dd(\R)$ and $\dd(\R^\bot)=\bb(\R)$.
    	\end{enumerate}
    \end{fac}
    
    We will list several basic lemmas from \cite{CM22}, where the role of Tukey connections in forcing Cicho\'n's maximum is clarified.
    
    \begin{dfn} \label{dfn_RS}
    	For an ideal $I$ on a set $X$, we define two relational systems:
        \begin{align*}
            \bar{I} & = \langle I, I, \subseteq\rangle\\
                    C_I &=\langle X,I,\in\rangle.
        \end{align*}
    	We write $\R\lq I$ to mean $\R\lq\bar{I}$ (we do the same for $\succeq_T$ and $\cong_T$). 
        Note that we have $\bb(\bar{I})=\addi,\dd(\bar{I})=\cofi$ and $\bb(C_I)=\noni,~\dd(C_I)=\covi$.
    \end{dfn}
    
    In this section, $\theta$ will always be a regular uncountable cardinal.
    
    \begin{fac}\leavevmode
    	Let $A$ be a set of size $\geq\theta$. Also, let $\R$ be a relational system.
    	\begin{enumerate}
    		\item If $\R\lq C_{[A]^{<\theta}}$, then $\theta\leq\bb(\R)$ and $\dd(\R)\leq|A|$.
    		\item If $C_{[A]^{<\theta}}\lq \R$, then $\bb(\R)\leq\theta$ and $|A|\leq\dd(\R)$.
    	\end{enumerate}
    \end{fac}
    		
    In \Cref{sec:CM}, 
    to obtain \Cref{thm:CM} from \Cref{thm:CM_left} using the \textit{submodel method}, 
    ``$\R\cong_T C_{[A]^{<\theta}}$'' does not work, but ``$\R\cong_T[A]^{<\theta}$'' does (see \cite{GKMS22}).
    The following fact gives a sufficient condition which implies $C_{[A]^{<\theta}}\cong_T [A]^{<\theta}$.
    
    \begin{fac}[{\cite[Lemma 1.15]{CM22}}] 
        \label{fac_suff_eq_CI_and_I}
    	If $A$ is a set with $|A|^{<\theta}=|A|$, then $C_{[A]^{<\theta}}\cong_T [A]^{<\theta}$.
    \end{fac}
    
    \begin{fac}[{\cite[Lemma 2.11]{CM22}}]\label{fac_cap_V}
    	Let $A$ be a set of size $\geq\theta$. Every ccc poset forces $[A]^{<\theta}\cong_T[A]^{<\theta}\cap V$ and $C_{[A]^{<\theta}}\cong_T C_{[A]^{<\theta}}\cap V$. Moreover, it forces $\mathfrak{x}([A]^{<\theta})=\mathfrak{x}^V([A]^{<\theta})$ where $\mathfrak{x}$ represents ``$\add$'', ``$\cov$'', ``$\non$'' or ``$\cof$''.
    \end{fac}
    
    The following fact will be used to keep several cardinal invariants of Cicho\'n's diagram small through the forcing iteration we shall perform in \Cref{sec:CM}.
    
    \begin{fac}[{\cite[Corollary 2.16, also Example 2.15]{Yam25}}]
    	\label{fac:smallness_for_addn_and_covn_and_nonm}
    	Let $\p$ be a finite support iteration of ccc posets of length $\gamma\geq\theta$.		
    	\begin{enumerate}
    		\item Assume that each iterand is either:
    		\begin{itemize}
    			\item of size $<\theta$,
    			\item a subalgebra of random forcing, or
    			\item $\sigma$-centered.
    		\end{itemize}
    		Then, $\p$ forces $C_{[\gamma]^{<\theta}}\lq\n$,
    		in particular, $\addn\leq\theta$.
    		\item  Assume that each iterand is either:
    		\begin{itemize}
    			\item of size $<\theta$, or
    			\item $\sigma$-centered.
    		\end{itemize}
    		Then, $\p$ forces $C_{[\gamma]^{<\theta}}\lq C_\n^\bot$,
    		in particular, $\covn\leq\theta$.
    		\item Assume that each iterand is:
    		\begin{itemize}
    			\item of size $<\theta$.
    		\end{itemize}
    		Then, $\p$ forces $C_{[\gamma]^{<\theta}}\lq C_\m$, in particular, $\nonm\leq\theta$.
    	\end{enumerate}
    \end{fac}
    
    We introduce the relational systems for $\non(M_\I)$, $\non(K_\I)$ and their duals:
    
    \begin{dfn}
        For an ideal $\I$ on a countable set $X$, we define the following relational systems:
        \begin{itemize}
            \item $\mathbf{M}_{\I}\coloneqq\langle X^\omega,\I^{(X^{<\omega})}, \triangleleft^* \rangle$, where $x\triangleleft^*\phi:\Leftrightarrow\forall^\infty n<\omega~x(n)\in\phi(x\on n)$.
            \item $\mathbf{K}_{\I}\coloneqq\langle X^\omega,\I^{\omega}, \in^* \rangle$, where $x\in^*\phi:\Leftrightarrow\forall^\infty n<\omega~x(n)\in\phi(n)$.
        \end{itemize}
        Note that $\bb(\mathbf{M}_{\I})=\non(M_\I)$, $\dd(\mathbf{M}_{\I})=\cov(M_\I)$, $\bb(\mathbf{K}_{\I})=\non(K_\I)$, $\dd(\mathbf{K}_{\I})=\cov(K_\I)$.
    \end{dfn}
    
    We will make use of the preservation theory presented in \cite{CM19}, which is a generalization of the classical preservation theory in \cite{JS90} and \cite{Bre91}.
    
    \begin{dfn}[{\cite[Definition 4.1]{CM19}}]
    	A relational system $\R=\langle X,Y,\sqsubset\rangle$ is called a \emph{Polish relational system} if the following hold:
    	\begin{enumerate}
    		\item $X$ is a perfect Polish space.
    		\item  $Y$ is analytic subset of a Polish space $Z$.
    		\item $\sqsubset =\bigcup_{n<\omega}\sqsubset_n$ for some ($\subseteq$-)increasing sequence  $\langle \sqsubset_n:n<\omega\rangle$ of closed subsets of $X\times Z$ such that for any $n<\omega$ and any $y\in Y$, $\{x\in X\colon x\sqsubset_n y\}$ is closed nowhere dense subset of $X$.
    	\end{enumerate}
    	When dealing with a Polish relational system, we interpret it depending on the model we are working in.
    \end{dfn}
    
    \begin{dfn}[{\cite{JS90}}]
        Let $\R=\langle X,Y,\sqsubset\rangle$ be a Polish relational system. A poset $\p$ is \emph{$\theta$-$\R$-good} if for any $\p$-name $\dot{y}$ for a member of $Y$, there is a non-empty set $Y_0\subseteq Y$ of size $<\theta$ such that for any $x\in X$, if $x$ is not met by any $y\in Y_0$, then $\p$ forces $x$ is not met by $\dot{y}$. If $\theta=\aleph_1$, we say ``$\R$-good'' instead of ``$\aleph_1$-$\R$-good''.
    \end{dfn}
    
    \begin{fac}
    	[{\cite[Lemma 4]{Mej13}, \cite[Theorem 6.4.7]{BJ95}}]
        Let $\R=\langle X,Y,\sqsubset\rangle$ be a Polish relational system. Every poset of size $<\theta$ is $\theta$-$\R$-good. In particular, Cohen forcing is $\R$-good.
    \end{fac}
    
    \begin{fac}[{\cite{JS90}, \cite[Corollary 4.13]{CM19}, \cite[Corollary 4.10]{BCM25}}]\label{fac:goodness_iteration}
        Let $\R=\langle X,Y,\sqsubset\rangle$ be a Polish relational system. Then any finite support iteration of ccc $\theta$-$\R$-good posets is $\theta$-$\R$-good.
    \end{fac}
    
    The following explains how $\R$-goodness is useful to control the cardinal invariants associated to $\R$.
    
    \begin{fac}[{\cite[Lemma 2.15]{CM22}, \cite[Theorem 4.11]{BCM25}}]\label{fac:FM21}
        Let $\R=\langle X,Y,\sqsubset\rangle$ be a Polish relational system and $\theta$ be an uncountable regular cardinal. Let $\p$ be the finite support iteration of non-trivial ccc $\theta$-$\R$-good posets of length $\gamma\geq\theta$. Then $\p$ forces $C_{[\gamma]^{<\theta}}\lq\R$.
    \end{fac}
    		
    We need goodness results for two specific Polish relational systems. One is associated with the constant evasion/prediction numbers discussed in \Cref{sec:rel_w_const}, and the other is associated with the usual bounding/dominating numbers. Recall the notions introduced in \Cref{dfn:constant_evasion_numbers}.
    
    \begin{dfn}
        Let $k\geq 2$.
        \begin{align*}
            \mathbf{CPR}_2(k) & \coloneqq\langle2^\omega,\Pred_2,\pck\rangle.\\
            \mathbf{CPR}_2 & \coloneqq\langle2^\omega,\Pred_2,\pc\rangle.
        \end{align*}
        These are Polish relational systems. Also, note that $\bb(\mathbf{CPR}_2(k))=\eec_2(k)$, $\dd(\mathbf{CPR}_2(k))=\vvc_2(k)$, $\bb(\mathbf{CPR}_2)=\eec_2$ and  $\bb(\mathbf{CPR}_2)=\vvc_2$.
    \end{dfn}
    
    Brendle and Shelah proved the following lemma when $\mu=\omega$, but their proof can be easily generalized.
    
    \begin{lem}[{\cite{BS03}, see also \cite[Lemma 5.29]{CR25}}] 
    \label{lem:twok_linked_CPR_good}
        Let $2\leq k<\omega$ and $\mu$ be an infinite cardinal. Then every $\mu$-$2^k$-linked\footnote{For $n\geq2$, a subset $Q$ of a poset $\p$ is \textit{$n$-linked} if any $n$-many conditions of $Q$ have a common extension. $\p$ is \textit{$\mu$-$n$-linked} if it is a union of $\mu$-many $n$-linked components.} poset is $\mu^+$-$\mathbf{CPR}_2(k)$-good.
    \end{lem}
    
    This lemma implies $\mathbf{CPR}_2$-goodness as follows:
    
    \begin{cor}\label{cor:infty_linked_CPR_good}
        Assume that $\mu$ is an infinite cardinal and a poset $\p$ is $\mu$-$l$-linked for any $l<\omega$. Then $\p$ is $\mu^+$-$\mathbf{CPR}_2$-good.
    \end{cor}
    \begin{proof}
        Let $\dot{\pi}$ be a $\p$-name of a member of $\Pred_2$. For each $k\geq2$, let $\Pi(k)\subset\Pred_2$ of size $\leq\mu$ induced by \Cref{lem:twok_linked_CPR_good} since $\p$ is $\mu$-$2^k$-linked. To see $\Pi\coloneqq\bigcup_{k\geq2}\Pi(k)$ witnesses that $\p$ is $\mu^+$-$\mathbf{CPR}_2$-good, take $x\in 2^\omega$ such that $\lnot(x\pc \pi)$ for any $\pi\in\Pi$. Assume towards contradiction that $p\Vdash x\pc \dot{\pi}$ for some $p\in\p$. Then there are $q\leq p$ and $k\geq 2$ such that $q\Vdash x\pck\dot{\pi}$, which contradicts that $\Pi(k)$ witnesses $\p$ is $\mu^+$-$\mathbf{CPR}_2(k)$-good.
    \end{proof}
    
    We now introduce the relational system for $\bb$ and $\dd$.
    
    \begin{dfn}
        Define the relational system $\mathbf{D}\coloneqq\langle \oo,\oo,\leq^*\rangle$. Note that this is a Polish relational system and $\bb(\mathbf{D})=\bb, \dd(\mathbf{D})=\dd$.
    \end{dfn}
    
    To treat $\mathbf{D}$-goodness, we will make use of the notion of \textit{Fr-limits}, introduced by Mej{\'\i}a \cite{Mej19}.
    
    \begin{dfn}
    	Let $\p$ be a poset.
    	\begin{enumerate}
    		\item For a countable sequence $\bar{p}=\langle p_m:m<\omega\rangle\in\p^\omega$, we define $\dot{W}(\bar{p})$ as the $\p$-name of an index set of the sequence $\bar{p}$ as follows:
            \[
            \Vdash_\p\dot{W}(\bar{p})\coloneq\{m<\omega:p_m\in\dot{G}\}
            \]
            where $\dot{G}$ denotes the canonical $\p$-name of a generic filter.
    		\item $Q\subseteq \p$ is \emph{Fr-linked} if there exists a function $\lim\colon Q^\omega\to\p$ such that for any countable sequence 
    		$\bar{q}\in Q^\omega$, 
    		\begin{equation}
    			\textstyle{\lim\bar{q}} \Vdash |\dot{W}(\bar{q})|=\omega.
    		\end{equation}
            Additionally, if $\ran(\lim)\subseteq Q$, we say $Q$ is closed-Fr-linked.
    		\item For an infinite cardinal $\mu$, $\p$ is \emph{$\mu$-(closed-)Fr-linked} if it is a union of $\mu$-many (closed-)Fr-linked components. When $\mu=\aleph_0$, we use $\sigma$ instead as usual. Define $<\mu$-(closed-)Fr-linkedness in the same way (for uncountable $\mu$). 
    	\end{enumerate}
    	We often say ``$\p$ has (closed-)Fr-limits'' instead of ``$\p$ is $\sigma$-(closed-)Fr-linked''.
    \end{dfn}
    
    \begin{lem}
    	\label{lem:poset_is_its_size_Frlinked}
    	Every poset $\p$ is closed-$|\p|$-Fr-linked. In particular, Cohen forcing $\mathbb{C}$ has closed-Fr-limits.
    \end{lem}
    
    \begin{fac}[\cite{Mej19}]
        \label{fac:Fr_good}
        Any $\mu$-Fr-linked poset is $\mu^+$-$\mathbf{D}$-good for any infinite cardinal $\mu$. In particular, any $<\mu$-Fr-linked poset is $\mu$-$\mathbf{D}$-good for any uncountable cardinal $\mu$.
    \end{fac}
    
    We will now present a stronger property using (non-principal) ultrafilters, introduced in \cite{GMS16}, which will be used to get our desired model of Cicho\'n's maximum in \Cref{sec:CM}. 
    
    \begin{dfn}\label{dfn_UF_linked} 
    	Let $D$ be an ultrafilter and $\p$ be a poset.
    	\begin{enumerate}
    		\item A subset $Q\subseteq \p$ is \emph{$D$-lim-linked} if there exist a $\p$-name $\dot{D}^\prime$ of an ultrafilter extending $D$ and a function $\lim^D\colon Q^\omega\to\p$ such that for any countable sequence $\bar{q}\in Q^\omega$, 
    		\begin{equation}
    			\textstyle{\lim^D\bar{q}} \Vdash \dot{W}(\bar{q})\in \dot{D}^\prime.
    		\end{equation}
            Additionally, if $\ran(\lim^D)\subseteq Q$, we say that $Q$ is \emph{closed-$D$-lim-linked}.
            \item A subset $Q\subseteq \p$ is \emph{(closed-)UF-lim-linked} if it is (closed-)$D$-lim-linked for any ultrafilter $D$.
            \item For an infinite cardinal $\mu$, $\p$ is \emph{$\mu$-(closed-)UF-lim-linked} if it is a union of $\mu$-many (closed-)UF-lim-linked components. When $\mu=\aleph_0$, we use $\sigma$ instead as usual. Define $<\mu$-(closed-)UF-lim-linkedness in the same way (for uncountable $\mu$).  
    	\end{enumerate}
        We often say ``$\p$ has (closed-)UF-limits'' instead of ``$\p$ is $\sigma$-(closed-)UF-lim-linked''.
    \end{dfn}
    
    The following forcing-free characterization of $D$-lim-linkedness given by the fourth author will be useful.
    
    \begin{lem}[{\cite[Lem 3.28]{Yam25}}]\label{lem:chara_UF_linked}
    	Let $D$ be an ultrafilter, $\p$ a poset, $Q\subseteq\p$, $\lim^D\colon Q^\omega\to\p$. Then the following are equivalent:
    	\begin{enumerate}
    		\item $\lim^D$ witnesses $Q$ is $D$-lim-linked.\label{item_lim_1}
    		\item $\lim^D$ satisfies $(\star)_k$ for all $k<\omega$, where \label{item_lim_2}
    		\begin{align*}
    			(\star)_k:& \text{``Given }\bar{q}^j=\langle q_m^j:m<\omega\rangle\in Q^\omega\text{ for }j<k\text{ and }r\leq\textstyle{\lim^D}\bar{q}^j\text{ for all }j<k,\\
    			&\text{then }\{m<\omega:r \text{ and all }q_m^j \text{ for } j<k \text{ have a common extension}\}\in D\text{''}.
    		\end{align*}
    	\end{enumerate}
    \end{lem}
    
    This characterization is also useful to prove Fr-linkedness:
    
    \begin{lem}
    	\label{lem:Fr_linked_star_one}
    	Let $D$ be an ultrafilter, $\p$ a poset, $Q\subseteq\p$ and $\lim^D\colon Q^\omega\to\p$.
    	If $\lim^D$ satisfies $(\star)_1$ in Lemma \ref{lem:chara_UF_linked}, then  $\lim^D$ witnesses that $Q$ is Fr-linked.
    \end{lem}
    \begin{proof}
    	By $(\star)_1$ and since $D$ is non-principal,
    	\[\text{if }\bar{q}=\langle q_m:m<\omega\rangle\in Q^\omega\text{ and }r\leq\textstyle{\lim^D}\bar{q},\text{ then we have }\exists^\infty m<\omega~\exists s\leq r,q_m.\]
    	Thus $\textstyle{\lim^D\bar{q}} \Vdash \exists^\infty m<\omega~q_m \in \dot{G}$.
    \end{proof}
    
    We will use closed-Fr-linkedness to control the values of $\non(M_{\finfin})$ and $\cov(M_{\finfin})$, not by goodness properties but by directly using closed-Fr-limits (\Cref{lem:closed_Fr_limits_keep_nonMfinfin_small}). To this end, we formulate a finite support iteration of ${<}\kappa$-closed-Fr-linked forcings. The formalization is based on \cite{Mej19}, \cite{CGHY24}.
    
    \begin{dfn}
    	\label{dfn:Gamma_iteration}
        Let $\kappa$ be an uncountable regular cardinal.
    	\begin{itemize}
    		\item 
    		A finite support iteration $\p_\gamma=\langle(\p_\xi,\qd_\xi):\xi<\gamma\rangle $ of ccc forcings is a \emph{${<}\kappa$-closed-Fr-iteration with witnesses $\langle\theta_\xi:\xi<\gamma\rangle$ and $\langle\dot{Q}_{\xi,\zeta}:\zeta<\theta_\xi,\xi<\gamma\rangle$} if for any $\xi<\gamma$, $\theta_\xi$ is a cardinal $<\kappa$ and
    		$\langle\dot{Q}_{\xi,\zeta}:\zeta<\theta_\xi\rangle$ are $\p_\xi$-names satisfying: 
    		\[
    		\Vdash_{\p_\xi}\dot{Q}_{\xi,\zeta}\subseteq\qd_\xi\text{ is closed-Fr-linked for }\zeta<\theta_\xi\text{ and }\bigcup_{\zeta<\theta_\xi}\dot{Q}_{\xi,\zeta}=\qd_\xi.
    		\]
    		\item A condition $p\in\p_\gamma$ is \emph{determined} if for each $\xi\in\dom(p)$, there is $\zeta_\xi<\theta_\xi$ such that  $\Vdash_{\p_\xi} p(\xi)\in \dot{Q}_{\xi,\zeta_\xi}$. Note that there are densely many determined conditions (proved by induction on $\gamma$).
    		
    	\end{itemize}
    \end{dfn}
    
    Closed-Fr-limits for conditions of the iteration $\p_\gamma$ are defined for ``refined'' sequences:
    
    \begin{dfn}
    	\label{dfn:uniform_delta_system}
        Let $\p_\gamma$ be a ${<}\kappa$-closed-Fr-iteration and $\delta$ be an ordinal.
    	We say that $\bar{p}=\langle p_m: m<\delta\rangle\in(\p_\gamma)^\delta$ is \emph{a uniform $\Delta$-system} if:
    	\begin{enumerate}
    		\item Each $p_m$ is determined, witnessed by $\langle\zeta^m_{\xi} :\xi\in\dom(p_m)\rangle$,
    		\item the family $\{\dom(p_m): m<\delta\}$ is a $\Delta$-system with root $\nabla$,
    		\item there is a sequence $\langle\zeta_\xi^{*} :\xi\in\nabla\rangle$ such that for $\xi\in\nabla$, $\zeta_\xi^{*}=\zeta^m_\xi$ for all $m<\omega$, i.e., all $p_m(\xi)$ are forced to be in a common closed-Fr-linked component for $\xi\in\nabla$,
    		\item all $\dom(p_ m)$ have $n^\prime$ elements, and $\dom(p_ m)=\{\a_{n, m}:n<n^\prime\}$ is the increasing enumeration,
    		\item there is $r^\prime\subseteq n^\prime$ such that $n\in r^\prime\Leftrightarrow\a_{n, m}\in\nabla$ for $n<n^\prime$,
    		\item for $n\in n^\prime\setminus r^\prime$, $\langle\a_{n, m}: m<\delta \rangle$ is (strictly) increasing.
    	\end{enumerate}
    \end{dfn}
    
    \begin{dfn} \label{dfn:of_lim_ite}
        Let $\p_\gamma$ be a ${<}\kappa$-closed-Fr-iteration and $\bar{p}=\langle p_m: m<\omega\rangle\in(\p_\gamma)^\omega$ be a uniform $\Delta$-system with root $\nabla$. We (inductively) define $p^\infty=\lim\bar{p}$ as follows:
    	\begin{enumerate}
    		\item $\dom(p^\infty)\coloneq\nabla$,
    		\item $p^\infty\on\xi\Vdash_{\p_\xi} p^\infty(\xi)\coloneq\lim\langle p_m(\xi):m\in\dot{W}(\bar{p}\on\xi)\rangle$ for $\xi\in\nabla$, where $\bar{p}\on\xi\coloneq\langle p_m\on\xi:m<\omega\rangle\in(\p_\xi)^\omega$. \label{eq_lim_second}
    	\end{enumerate}
    \end{dfn}
    
    To see that the second item is valid, $\dot{W}(\bar{p}\on\xi)$ has to be infinite, which is true: 
    
    \begin{lem}[{\cite{Mej19}, \cite[Lemma 3.6]{CGHY24}}]\label{lem:Fr-limit_principle}
        Let $\p_\gamma$ be a ${<}\kappa$-closed-Fr-iteration. Let $\bar{p}=\langle p_m: m<\omega\rangle\in(\p_\gamma)^\omega$ be a uniform $\Delta$-system and $p^\infty\coloneq\lim\bar{p}$.
    	Then
    	$p^\infty\Vdash_{\p_\gamma}\lvert\dot{W}(\bar{p})\rvert=\omega$.
    \end{lem}
    
    The following is specific for closedness:
    
    \begin{lem}\label{lem:closedness}
        Let $\p_\gamma$ be a ${<}\kappa$-closed-Fr-iteration. Let $\bar{p}=\langle p_m: m<\omega\rangle\in(\p_\gamma)^\omega$ be a uniform $\Delta$-system with parameters as in \Cref{dfn:uniform_delta_system}. Then $\lim\bar{p}$ is determined, witnessed by $\langle\zeta^{*}_{\xi} :\xi\in\nabla\rangle$.
    \end{lem}
    \begin{proof}
        Direct from the definitions of $\lim\bar{p}$ and closedness.
    \end{proof}

    \subsection{Separations into two values}\label{sec:Fr-limits}

    In this subsection, we study the separation of cardinal invariants into two values.

    \subsubsection{The poset $\posetED$ and Fr-limits.}
    
    We have seen the connection between $\non(M_{\I})$ (and $\cov(M_{\I})$) and the constant evasion (and prediction) number in \Cref{sec:non_cov}. In \Cref{prop:non_M_finfin}, we have shown that $\non(M_{\finfin})=\eec_{\leq}(2)$, and in \Cref{prop:eect_nonMed}, we proved that $\eect\leq\non(M_{\ed})$, so $\max\{\eect,\bb\}\leq\non(M_{\ed})$. We will show that equality cannot be proved. 
    
    We introduce the forcing notion $\posetED$ which generically adds an $\mathbf{M}_{\ed}$-dominating function $\phi_G\colon(\omega\times\omega)^{<\omega}\to\ed$.
    
    \begin{dfn}\label{dfn_posetED}
        The poset $\posetED$ is defined as follows: Its conditions are $p=(\sigma^p,s^p,n^p,F^p)$, where
        \begin{itemize}
        	\item $\sigma^p\colon(\omega\times\omega)^{<n^p}\to\omega$ is a finite partial function, 
        	\item $s^p\colon\{(u^\frown a,v):(u,v)\in(\omega\times\omega)^{<n^p},a\in\omega\}\to\omega$ is a finite partial function, 
        	\item $n^p\in\omega$  
        	\item $F^p\subseteq\oo$ is finite.   
        \end{itemize}
        The order $q\leq p$ is given by
        \begin{itemize}
        	\item $\sigma^q\supseteq \sigma^p$, $s^q\supseteq s^p$, $n^q\geq n^p$ and $F^q\supseteq F^p$,  
        	\item $\forall i\in\left[n^p,n^q\right)\forall x,y\in F^p$, either:  
        	\begin{itemize}
        		\item $(x\on i, y\on i)\in\dom(\sigma^q)$ and $x(i)\leq\sigma^q(x\on i, y\on i)$, or  
        		\item $(x\on (i+1), y\on i)\in\dom(s^q)$ and $y(i)= s^q(x\on (i+1), y\on i)$.
        	\end{itemize}
        \end{itemize}
        Let $G\subset\posetED$ be a generic filter. In $V[G]$,  $\phi_G\colon(\omega\times\omega)^{<\omega}\to\omega\times\omega$ is given by: 
    	\[
        \phi_G(u,v)\coloneqq\bigcup_{p\in G}\{(a,b)\in\omega\times\omega:a\leq\sigma^p(u,v)\text{ or }b= s^p(u^\frown a,v)\}.
        \]
    \end{dfn}
    
    One can easily show the following.
    
    \begin{lem}
        \label{lem:posetED_basics}\leavevmode
    	\begin{itemize}
            \item For $n<\omega$, $\{p:n_p\geq n\}$ is dense. Thus $\posetED\Vdash\dom(\phi_G)=(\omega\times\omega)^{<\omega}$.
            \item For $(u,v)\in(\omega\times\omega)^{<\omega}$,  $\posetED\Vdash\forall^\infty n<\omega~|(\phi_G(u,v))_n|=1$, so $\ran(\phi_G)\subseteq\ed$.
            \item For $f,g\in \oo$, $\{p:f,g\in F^p\}$ is dense. 
    		Thus $\posetED\Vdash\forall(f,g)\in(\omega\times\omega)^\omega\cap V,~(f,g)\in M_{\phi_G}$, i.e., $(\omega\times\omega)^\omega\cap V\subseteq M_{\phi_G}$.
    	\end{itemize}
    \end{lem}
    
    $\posetED$ is $\sigma$-centered and $\sigma$-Fr-linked.
    
    \begin{lem}\label{lem:posetED_centered_Frlinked}
    	Fix $n\in\omega$, $\sigma\colon(\omega\times\omega)^{<n}\to\omega$ finite partial function and $s\colon\{(u^\frown a,v):(u,v)\in(\omega\times\omega)^{<n},a\in\omega\}\to\omega$ finite partial function. Let $L<\omega$ and $\bar{x}^*=\{x^*_l:l<L\}\subseteq\omega^n$ such that all $x^*_l$ are pairwise different. 
    	Then
        \[
        Q\coloneqq\{p\in\posetED\colon\sigma^p=\sigma, s^p=s,n^p=n, \{x\on n\colon x\in F^p\}=\bar{x}^*\}
        \]
    	is centered and Fr-linked. In particular, $\posetED$ is $\sigma$-centered and $\sigma$-Fr-linked.
    \end{lem}
    \begin{proof}
        Centeredness is clear. To prove $Q$ is Fr-linked, take some ultrafilter $D$, we will show $(\star)_1$ as in \Cref{lem:Fr_linked_star_one}.
    	For $\bar{x}=\langle x^m\in\oo:m<\omega\rangle$, 
    	define a partial function $\bar{x}^\infty\colon\omega\to\omega$ as follows: for each $i<\omega$, if there (uniquely) exists $a_i<\omega$ such that $\{m<\omega:x^i(m)=a_i\}\in D$,
    	let $i\in\dom(\bar{x}^\infty)$ and $\bar{x}^\infty(i)\coloneq a_i$. Otherwise let $i\notin\dom(\bar{x}^\infty)$.
    	
    	For $\bar{q}=\langle q_m=(\sigma,s,n,F_m=\{x^m_l:l<L\}):m<\omega\rangle\in Q^\omega$,
    	we define $\lim^D\bar{q}\coloneqq q^\infty\coloneqq(\sigma^\infty,s^\infty,n^\infty,F^\infty)$ as follows: 
    	\begin{itemize}
    		\item $\bar{x}_l\coloneq\langle x^m_l:m<\omega\rangle$, $A\coloneq\{l<k:{\dom(\bar{x}_l^\infty})=\omega\}, B\coloneq L\setminus A$.
    		\item $F^\infty\coloneq\{{\bar{x}_l^\infty}:l\in A\}$.
    		\item For $l\in B$, $n_l\coloneqq \min(\omega\setminus \dom(\bar{x}_l^\infty))$
    		(hence $n_l\geq n$). For $l\in A$, put $n_l\coloneqq\omega$ for convenience.
    		\item $n^\infty\coloneqq n \cup \max\{n_l+1:l\in B\}$. 
    		\item For each $l\in B$, let $t_l\in\omega^{n_l}$ be such that \[
            X_0^l\coloneqq\{m<\omega:x_l^m\on n_l=t_l\}\in D.
            \]
            Put $X_0\coloneqq\bigcap_{l\in B}X_0^l\in D$.
    		\item For $i<\omega$,
            put
            \[
            W_i\coloneqq\{\bar{x}_l^\infty\on i:l\in A\}\cup\{t_l\on i:l\in B, i\leq n_l\}.
            \]
            Note that $W_i$ is a finite subset of $\omega^i$.
    		\item Take a finite partial function $\sigma^\infty\colon(\omega\times\omega)^{<n^p}\to\omega$ such that:
    		\begin{enumerate}
    			\item $\sigma^\infty\supseteq\sigma$.
    			\item For any $i\in\left[n,n^\infty\right)$ and any $u\in W_{i+1}$ and $v\in W_i$, $(u\on i,v)\in\dom(\sigma^\infty)$ and $u(i)\leq\sigma^\infty(u\on i,v)$.
    		\end{enumerate}
    		\item $s^\infty\coloneqq s$.
    	\end{itemize}
    	Clearly $q^\infty$ is a valid condition,
    	so it is enough to show that $\lim^D$ satisfies $(\star)_1$.
    	Assume that  $r\leq q^\infty=(\sigma^\infty,s^\infty,n^\infty,F^\infty)$.
    	Let $r\coloneq (\sigma^r,s^r,n^r,F^r)$.
    	For $l\in A$, put
        \[
        X_1^l\coloneqq\{m<\omega:x^m_l\on (n^r+1)=x^\infty_l\on(n^r+1)\}\in D.
        \]
    	Let $X_1\coloneqq \bigcap\{X_1^l:l\in A\}\in D$.
    	For $l\in B$ and $l^\prime<L$, let
        \[
        X_2^{l,l^\prime}\coloneqq \{m<\omega:(x^m_l\on(n_l+1),x^m_{l^\prime}\on n_l)\notin\dom(s^r)\}\in D.
        \]
    	Put $X_2\coloneqq\bigcap\{X_2^l:l\in B,l^\prime<L\}\in D$.
    	For $l\in B$ and $i>n_l$, let
        \[
        X_3^{l,i}\coloneqq \{m<\omega:x^m_l\on i\notin\operatorname{proj}_1(\dom(\sigma^r))\cup\operatorname{proj}_2(\dom(\sigma^r))\}\in D,
        \]
        where $\operatorname{proj}_1(\dom(\sigma^r))\coloneqq\dom(\dom(\sigma^r))$ and $\operatorname{proj}_2(\dom(\sigma^r))\coloneqq\ran(\dom(\sigma^r))$.
    	Put $X_3\coloneqq\bigcap\{X_3^{l,i}:l\in B, i\in\left(n_l,n^r\right)\}\in D$. Put $X\coloneq X_0\cap X_1\cap X_2\cap X_3\in D$.
        
    	It is enough to show that for all $m\in X$, $q_m$ and $r$ are compatible.
    	Fix such $m$. Note that $\sigma^r\supseteq\sigma^\infty\supseteq\sigma$ and $s^r\supseteq s^\infty\supseteq s$.  
    	Take a condition $q^\prime\coloneq (\sigma^\prime,s^\prime,n^\prime,F^\prime)$ (as a common extension of $q_m$ and $r$) such that: 
    	\begin{itemize}
    		\item $F^\prime\coloneq F^r\cup F_m$.
    		\item $n^\prime\coloneqq n^r$.
    		\item $\sigma^\prime\supseteq\sigma^r$ and for all $(u,v)\in(\omega\times\omega)^{<\omega}\setminus \dom (\sigma^r)$ of the form $(x^{m}_{l}\on i, x^{m}_{l^\prime}\on i)$ for some $l,l^\prime<L$ and $i<n^r$, let $(u,v)\in\dom(\sigma^\prime)$ and:
            \begin{equation}
                \sigma^\prime(u,v)\coloneqq \max\{x^{m}_{l^*}(i):l^*<L\}. 
            \end{equation}
            Note that there are only finitely many such $(u,v)$.
    		\item $s^\prime\supseteq s^r$ and for all $(u^\prime,v)\in\fsq\times\fsq\setminus \dom (s^r)$ of the form $(x^{m}_{l}\on (n_l+1), x^{m}_{l^\prime}\on n_l)$ for some $l\in B$ and $l^\prime<L$, 
            let $(u^\prime,v)\in\dom(s^\prime)$ and:
            \begin{equation}
            \label{equation:s_prime_ED}
                s^\prime(u^\prime,v)\coloneqq x^{m}_{l^\prime}(n_l),
            \end{equation}
            which is possible since all $x^*_l$ are pairwise distinct. Note that there are only finitely many such $(u^\prime,v)$.
    	\end{itemize}
    	By the choice of $X_3$, for any $l\in B$, $i\in\left(n_l,n^r\right)$ and $l^\prime<L$, we have 
    	$(x^m_l\on i, x^m_{l^\prime}\on i), (x^m_{l^\prime}\on i, x^m_l\on i)\notin \dom (\sigma^r)$ and hence
    	\begin{equation}
    		\label{eq:sigma_prime}
    		x^m_l(i)\leq\sigma^\prime(x^m_l\on i, x^m_{l^\prime}\on i)\text{ and }x^m_{l^\prime}(i)\leq\sigma^\prime(x^m_{l^\prime}\on i, x^m_l\on i).
    	\end{equation}
    	By the choice of $X_2$, for any $l\in B$ and $l^\prime<L$ we have $(x^m_l\on(n_l+1),x^m_{l^\prime}\on n_l)\notin\dom(s^r)$ and hence 
    	\begin{equation}
    		\label{eq:s_prime}
    		x^m_{l^\prime}(n_l)= s^\prime(x^m_l\on(n_l+1),x^m_{l^\prime}\on n_l).
    	\end{equation}
    	
    	$q^\prime\leq r$ trivially holds.
    	To see $q^\prime\leq q_m$, let $l,l^\prime<L $ and $ i\in\left[n, n^\prime\right)$. 
    	We show that:
        \begin{equation}
            \label{eq:posetED_WANT}
            \text{either }x^m_l(i)\leq\sigma^\prime(x^m_l\on i, x^m_{l^\prime}\on i)\text{ or }
            x^m_{l^\prime}(i)= s^\prime(x^m_l\on(i+1),x^m_{l^\prime}\on i).
        \end{equation}
        
    	First assume $l,l^\prime\in A$. If $ i\in\left[n, n^\infty\right)$, by the choice of $X_1$, $x^m_l\on (i+1)=\bar{x}^\infty_l\on (i+1)\in W_{i+1}$ and $x^m _{l^\prime}\on  i=\bar{x}^\infty_{l^\prime}\on  i\in W_i$. Thus $x^m_l(i)\leq\sigma^\infty(x^m_l\on i, x^m_{l^\prime}\on i)=\sigma^\prime(x^m_l\on i, x^m_{l^\prime}\on i)$. If $ i\in\left[n^\infty, n^r\right)$, \eqref{eq:posetED_WANT} follows from $r\leq q^\infty$ and by the choice of $X_1$.
    
        Thus we may assume either $l\in B$ or $l^\prime\in B$. If $i\in\left(n_l,n^r\right)$ or $i\in\left(n_{l^\prime},n^r\right)$, we are done by \eqref{eq:sigma_prime}. 
    	If $i\notin\left(n_l,n^r\right)$ and $i\notin\left(n_{l^\prime},n^r\right)$, we particularly have $i\in \left[n, n^\infty\right)$.
    	Put $u\coloneqq x^m_l \on (i+1)$ and $v\coloneqq x^m_{l^\prime} \on i$. There are three cases.
        \begin{enumerate}[label = (\roman*)]
            \item
            If $u\in W_{i+1}$ and $v\in W_i$, then $(u\on i,v)\in\dom(\sigma^\infty)$ and $u(i)\leq\sigma^\infty(u\on i, v)=\sigma^\prime(u\on i, v)$ by the choice of $\sigma^\infty$, so \eqref{eq:posetED_WANT} is true.
    
            \item
            If $v\notin W_i$, then $l^\prime\in B$ holds, since otherwise $v=x^m _{l^\prime}\on  i=\bar{x}^\infty_{l^\prime}\on  i\in W_i$ by the choice of $X_1$. Since $i\notin\left(n_{l^\prime},n^r\right)$, we have $i\leq n_{l^\prime}$. By the choice of $X_0$, $v=x_{l^\prime}^m\on i=t_{l^\prime}\on i\in W_i$, a contradiction.
    
            \item
            If $u\notin W_{i+1}$, for the same reason as the previous case, $l\in B$ and $i+1>n_l$ since otherwise we would obtain $u\in W_{i+1}$. Since $i\notin\left(n_l,n^r\right)$, we have $i=n_l$. In this case $x^m_{l^\prime}(n_l)= s^\prime(x^m_l\on(n_l+1),x^m_{l^\prime}\on n_l)$ by \eqref{eq:s_prime} and hence \eqref{eq:posetED_WANT} is true.
        \end{enumerate}
        Therefore, \eqref{eq:posetED_WANT} is satisfied in any case and hence $q^\prime\leq q_m$.
    \end{proof}
    
    Brendle introduced a forcing notion $\p^\omega$ which adds a generic predictor on the $2$-constant prediction:
    
    \begin{fac}[{\cite{Bre03}}]\label{fac:p_omega}
        For $2\leq k<\omega$, there is a $\sigma$-linked poset $\p^\omega$ such that:
        \[
        \Vdash_{\p^\omega}\exists\dot{\pi}\in\Pred~\forall x\in\oo\cap V~ (x\pck\dot{\pi}).
        \]
    \end{fac}
    
    Now we prove the consistency of $\max\{\bb,\eect\}<\non(M_{\ed})$ (and the dual).
    
    \begin{thm}
    \label{thm:Con_max_b_eect_nonMED}\leavevmode
        \begin{enumerate}
            \item Given $\kappa<\lambda=\lambda^{<\kappa}$ regular uncountable, there is a ccc poset forcing $\bb=\eec(2)=\eec_2=\kappa$ and $\non(M_{\ed})=\cc=\lambda$.
            In particular, $\max\{\bb,\eect\}<\non(M_{\ed})$ is consistent.
            \item Given $\kappa$ regular uncountable and $\lambda=\lambda^{\omega}>\kappa$, there is a ccc poset forcing $\cov(M_{\ed})=\kappa$ and $\dd=\vvc_2=\lambda=\continuum$. In particular, $\cov(M_{\ed})<\min\{\dd,\vvct\}$ is consistent.
        \end{enumerate}
    \end{thm}
    \begin{proof}
        \begin{enumerate}
            \item Using a bookkeeping argument, craft $\p$ a finite support iteration of length $\lambda$ whose iterands are either:
            \begin{enumerate}
                \item $\posetED$,
                \item a subforcing of Hechler forcing of size $<\kappa$, or
                \item a subforcing of $\p^\omega$ of size $<\kappa$.
            \end{enumerate}
            such that $\posetED$ appears cofinally, and all possible witnesses for $\bb$ and $\eec(2)$ of size $< \kappa$ are destroyed.
            In the $V^\p$-extension, clearly $\lambda\leq\non(M_{\ed})\leq\cc$ and $\bb,\eec(2)\geq\kappa$. By \Cref{lem:poset_is_its_size_Frlinked} and \Cref{lem:posetED_centered_Frlinked}, every iterand is $<\kappa$-Fr-linked and hence $\kappa$-$\mathbf{D}$-good by \Cref{fac:Fr_good}.
            Thus $\p$ is also $\kappa$-$\mathbf{D}$-good by \Cref{fac:goodness_iteration}, so we have $\bb\leq\kappa$ by \Cref{fac:FM21}. Similarly, by \Cref{lem:posetED_centered_Frlinked}, every iterand is $<\kappa$-$l$-linked for any $l<\omega$, $\p$ is $\kappa$-$\mathbf{CPR}_2$-good by \Cref{cor:infty_linked_CPR_good} and hence $\eec(2)\leq\eec_2\leq\kappa$.
            \item Let $\p$ be a finite support iteration of length $\lambda+\kappa$ such that the first $\lambda$-many iterands are Cohen forcings and each of the rest is $\posetED$. It is easy to see that $\p$ forces $\kappa\leq\covm\leq\cov(M_{\ed})\leq\kappa$ and $\cc\leq\lambda$. 
            Since every iterand is $\mathbf{D}$-good and $\mathbf{CPR}_2$-good, we have $\Vdash_{\p}\dd,\vvc_2\geq\lambda$ by \Cref{fac:FM21}.
        \end{enumerate}
    \end{proof}

    \subsubsection{The poset $\posetEDfin$ and closed-Fr-limits.}\label{sec:posetEDfin}

    Let $\langle I_i=\left[M_i,M_{i+1}\right)\rangle_{i<\omega}$ be the interval partition of $\omega$ with $|I_i|=i$ and we consider $\edfin$ on this interval partition. We introduce the forcing notion $\posetEDfin$ which generically adds a $\mathbf{K}_{\edfin}$-dominating function $\phi_G\colon\omega\to\edfin$.
	
    \begin{dfn} 	
		A forcing notion $\posetEDfin$ is defined as follows:
		\begin{itemize}
			\item Conditions are $p=(i,k,s,\varphi)=(i_p,k_p,s_p,\varphi_p)$ where $i,k<\omega$, $s\in([M_i]^{\leq k})^{<\omega}$ and $\varphi\colon \omega\to [\omega]^{\leq k}$ such that $\varphi\on |s|=s$.
			\item The order $(i^\prime,k^\prime,s^\prime,\varphi^\prime)\leq(i,k,s,\varphi)$ is defined by: $i^\prime\geq i$, $k^\prime\geq k$, $\varphi^\prime(n)\supseteq\varphi(n)$ for all $n<\omega$, $|s^\prime|\geq|s|$ and for $n<|s|$,
			\begin{gather}
				\label{gath_end_ext_EDfin} s^\prime(n)\cap M_i=s(n),\text{ and}\\ 
				\label{gath_density_small_EDfin} \text{if }i\leq j< i^\prime,\text{ then } |s^\prime(n)\cap I_j|\leq n.
			\end{gather}
			\item For a generic filter $G$, define $\phi_G\colon\omega\to\mathcal{P}(\omega)$ by:
			\[\phi_G(n)\coloneq\bigcup \{s_p(n):p\in G\text{ and }n<|s|\}.\]
		\end{itemize}
	\end{dfn}
	
	One easily sees:
    
	\begin{lem}
		\label{lem:pEDfin_basics}
		The following sets are open dense:
		\begin{enumerate}
			\item\label{item_lem_pEDfin_basics_x} For $x\in\oo$, $\{p:x(n)\in\varphi_p(n)\text{ for all }n\geq |s_p|\}$, 
			\item for $i<\omega$, $\{p:i_p\geq i\}$, 
			\item\label{item_lem_pEDfin_basics_s} for $n<\omega$, $\{p:|s_p|\geq n\}$. 
		\end{enumerate}
	\end{lem}
    
	By Lemma \ref{lem:pEDfin_basics}\eqref{item_lem_pEDfin_basics_s}, we may always assume $|s_p|\geq k_p$.
	By \eqref{gath_density_small_EDfin} and by Lemma \ref{lem:pEDfin_basics} \eqref{item_lem_pEDfin_basics_x}, we have the following:
	
    \begin{lem}
		Let $G$ be a generic filter. The following are true in $V[G]$:
		\begin{enumerate}
			\item $\forall n<\omega~\forall^\infty i<\omega~|\phi_G(n)\cap I_i|\leq n$. In particular, $\phi_G(n)\in\edfin$.
			\item $\forall x\in\oo\cap V~\forall^\infty n<\omega~x(n)\in \phi_G(n)$. In other words, $\oo\cap V\subset K_{\phi_G}$.
		\end{enumerate}
	\end{lem}
	
    $\posetEDfin$ is $\sigma$-centered and $\sigma$-Fr-linked:
	
    \begin{lem}
		\label{lem:pEDfin_has_Fr-limits}
        For $i,k<\omega$ and $s\in([M_i]^{\leq k})^{<\omega}$ with $|s|\geq k$,
		\[
        Q=\{p\in\posetEDfin:i_p=i, k_p=k, s_p=s\}
        \]
        is centered and closed-Fr-linked. In particular, $\posetEDfin$ is $\sigma$-centered and $\sigma$-closed-Fr-linked.
	\end{lem}
	\begin{proof}
        Centeredness is clear.
		To show $(\star)_1$ in Lemma \ref{lem:Fr_linked_star_one}, take some non-principal ultrafilter $D$ on $\omega$ and let and $\bar{q}=\langle(i,k,s,\varphi_m)\rangle_{m<\omega}\in Q^\omega$. 
		Define $\varphi_\infty\colon\omega\to[\omega]^{\leq k}$ by:
		\begin{equation}
			\label{eq_def_of_varphi_infty_EDfin}
			a\in\varphi_\infty(n):\Leftrightarrow\{m<\omega:a\in\varphi_m(n)\}\in D,
		\end{equation}
		for $n,a<\omega$. Put $\lim^D\bar{q}\coloneqq q_\infty\coloneqq(i,k,s,\varphi_\infty)$ (so $\ran(\lim^D)\subseteq Q$). To see $(\star)_1$, 
		assume $q^\prime=(i^\prime,k^\prime,s^\prime,\varphi^\prime)\leq q_\infty$.
		Put $M\coloneq M_i$ and $M^\prime\coloneq M_{i^\prime}$.
		By \eqref{eq_def_of_varphi_infty_EDfin}, there exists $X\in D$ such that:
		\begin{equation}
			\label{eq_X_EDfin}
			\text{For all }m\in X\text{ and }n\in|s^\prime|\setminus|s|, ~\varphi_m(n)\cap M^\prime=\varphi_\infty(n),
		\end{equation}
        since $\varphi_\infty(n)\subseteq M^\prime$ for such $n$ by $q^\prime\leq q_\infty$.
		Fix $m\in X$ and we will define a common extension $r$ of $q^\prime$ and $q_m$. 
        Define $\psi\colon\omega\to [\omega]^{\leq k^\prime+k}$ by $\psi(n)\coloneqq\varphi^\prime(n)\cup\varphi_m(n)$ and $t\coloneq\psi\on|s^\prime|$. Note that:
		\begin{equation}
			\label{eq_t_n_equal_s_prime_n_EDfin}
			\text{for }n<|s|,~t(n)=\varphi^\prime(n)\cup\varphi_m(n)=s^\prime(n)\cup s(n)=s^\prime(n).
		\end{equation}
        Put $r\coloneq(i^{\prime\prime},k^\prime+k,t,\psi)$, where $i^{\prime\prime}\geq i^\prime$ is so large that $r$ is a valid condition. 
		To see $r\leq q^\prime$, it is enough to check that for $n<|t|$:
		\begin{enumerate}
			\item\label{item_r_leq_q_prime_first_EDfin} $t(n)\cap M^\prime=s^\prime(n)$, and
			\item\label{item_r_leq_q_prime_second_EDfin}  $\text{if }i^\prime\leq j< i^{\prime\prime},\text{ then } |t(n)\cap I_j|\leq n$.
		\end{enumerate}
		We show \eqref{item_r_leq_q_prime_first_EDfin}.
		If $n<|s|$, by \eqref{eq_t_n_equal_s_prime_n_EDfin} we have
        \[
        t(n)\cap M^\prime=s^\prime(n)\cap M^\prime=s^\prime(n).
        \]
        If $n\in|t|\setminus|s|$, by \eqref{eq_X_EDfin} we have
        \[
        t(n)\cap M^\prime=(\varphi^\prime(n)\cap M^\prime)\cup(\varphi_m(n)\cap M^\prime)=\varphi^\prime(n)\cup\varphi_\infty(n)=\varphi^\prime(n)=s^\prime(n)
        \]
        by $\varphi^\prime(n)\supseteq\varphi_\infty(n)$, which follows from $q^\prime\leq q_\infty$.
		
		We show \eqref{item_r_leq_q_prime_second_EDfin}. 
		If $n<|s|$, then
        \[
        t(n)\cap I_j=s^\prime(n)\cap I_j=\emptyset
        \]
        by \eqref{eq_t_n_equal_s_prime_n_EDfin}.
		If $n\in|t|\setminus |s|$, then 
		\[
        t(n)\cap I_j=(\varphi^\prime(n)\cup\varphi_m(n))\cap I_j=\varphi_m(n)\cap I_j\subseteq \varphi_m(n)
        \]
        has size $\leq k\leq|s|\leq n$. 
		
		To see $r\leq q_m$, it is enough to check that for $n<|s|$:
		\begin{enumerate}
			\setcounter{enumi}{2}
			\item \label{item_r_leq_q_m_third_EDfin} $t(n)\cap M=s(n)$, and:
			\item \label{item_r_leq_q_m_fourth_EDfin} $\text{if }i\leq j< i^{\prime\prime},\text{ then } |t(n)\cap I_j|\leq n$.
		\end{enumerate}
		\eqref{item_r_leq_q_m_third_EDfin} follows from $t(n)\cap M=s^\prime(n)\cap M=s(n)$ by \eqref{eq_t_n_equal_s_prime_n_EDfin} and $q^\prime\leq q_\infty$.
		
		To show \eqref{item_r_leq_q_m_fourth_EDfin}, let $i\leq j< i^{\prime\prime}$ and we may assume $i\leq j< i^{\prime}$ by \eqref{item_r_leq_q_prime_second_EDfin}.
		$t(n)\cap I_j=s^\prime(n)\cap I_j$ by \eqref{eq_t_n_equal_s_prime_n_EDfin} and $|s^\prime(n)\cap I_j|\leq n$ follows from $q^\prime\leq q_\infty$.
		
		Therefore, $r$ extends $q^\prime$ and $q_m$.
	\end{proof}

    Closed-Fr-limits can control the values of $\non(M_{\finfin})$ and $\cov(M_{\finfin})$ using Cohen reals:

    \begin{lem}\label{lem:closed_Fr_limits_keep_nonMfinfin_small}
        Let $\kappa$ be an uncountable regular cardinal and $\p_\gamma$ be a ${<}\kappa$-closed-Fr-iteration such that the first $\kappa$-many iterands are Cohen forcings. Let $\{\dot{c}_\b:\b<\kappa\}$ be the added Cohen reals as members of $(\omega\times\omega)^\omega$. Then, for any $\p_\gamma$-name $\dot{\phi}$ of a function $(\omega\times\omega)^{<\omega}\to\finfin$, $\Vdash_{\p_\gamma}|\{\b<\kappa:\dot{c}_\b\in M_{\dot{\phi}} \}|<\kappa$.
    \end{lem}
    \begin{proof}
        For $x\in(\omega\times\omega)^\omega$, $\phi\colon(\omega\times\omega)^{<\omega}\to\finfin$ and $n<\omega$, let $x\in^*_n\phi$ denote $\forall m\geq n~x(m)\in\phi(x\on m)$. 
    	Assume towards contradiction that there exists a condition $p\in\p_\gamma$  such that 
    	$p\Vdash |\{\b<\kappa:\dot{c}_\b\in M_{\dot{\phi}}\}|=\kappa$.
    	For each $i<\kappa$, inductively pick $p_i\leq p$, $\b_i<\lambda$ and $n_i<\omega$ such that $\b_i\notin\{\b_{i^\prime}:i^\prime<i\}$ and $ p_i\Vdash\dot{c}_{\b_i}\in^*_{n_i}\dot{\phi}$.
    	By extending and thinning, we may assume:
    	
    	\begin{enumerate}
    		\item $\b_i\in\dom(p_i)$. (By extending $p_i$.)
    		\item $\{p_i:i<\kappa\}$ forms a uniform $\Delta$-system with root $\nabla$.
    		\item All $n_i$ are equal to $ n^*$.
    		\item All $p_i(\b_i)$ are the same Cohen condition $s\in(\omega\times\omega)^{<\omega}$. 
    		\item $|s|=n^*$. (By extending $s$ or increasing $n^*$.)
    	\end{enumerate}
    	
    	In particular, we have that:
    	\begin{equation}
    		\label{eq_property_of_refined_pi}
    		\text{For each }i<\kappa, p_i\text{ forces }\dot{c}_{\b_i}\on n^*=s\text{ and }\dot{c}_{\b_i}\in^*_{n^*}\dot{\phi}.
    	\end{equation}
    	
    	Note that $\b_i\notin\nabla$ for $i<\kappa$ since all $\b_i$ are distinct. Pick the first $\omega$ many $p_i$ and fix some bijection $i\colon\omega\times\omega\to\omega$.
    	For each $(a,b)\in\omega\times\omega$, define $q_{a,b}\leq p_{i(a,b)}$ by extending the $\b_{i(a,b)}$-th position $q_\sigma(\b_{i(a,b)}):=s^\frown(a,b)$.
    	By \eqref{eq_property_of_refined_pi}, 
    		$\text{for each }(a,b)\in\omega\times\omega,~q_{a,b}\text{ forces }\dot{c}_{\b_{i(a,b)}}\on(n^*+1)=s^\frown(a,b)\text{ and }\dot{c}_{\b_{i(a,b)}}\in^*_{n^*}\dot{\phi}$,
    	thus
    	\begin{equation}
    		\label{eq:q_a_b}
    		q_{a,b}\forces\dot{c}_{\b_{i(a,b)}}(n^*)=(a,b)\in\dot{\phi}(s)=\dot{\phi}(\dot{c}_{\b_{i(a,b)}}\on n^*).
    	\end{equation}
    	
    	Fix $a<\omega$ and we consider the sequence $\bar{q}_a\coloneq\langle q_{a,b}:b<\omega\rangle$.
    	When defining $q_{a,b}$ we changed the $\b_{i(a,b)}$-th position which is out of $\nabla$, so $\{q_{a,b}:b<\omega\}$ forms a uniform $\Delta$-system with root $\nabla$.
    	Thus we can take their limit $q^\infty_a\coloneqq\lim\bar{q}_a$.
    	By \eqref{eq:q_a_b} and \Cref{lem:Fr-limit_principle}, we obtain:
    	\begin{equation}
    		\label{eq:q_a_infty}
    		q^\infty_a\forces \exists^\infty b<\omega ~(a,b)\in\dot{\phi}(s).
    	\end{equation}
    	
    	Unfix $a$ and consider the sequence $\bar{q}\coloneq\langle q^\infty_{a}:a<\omega\rangle$.
        By \Cref{lem:closedness}, all $q^\infty_{a}$ have domain $\nabla$ and they form a uniform $\Delta$-system with root $\nabla$. Take their limit $q^\infty\coloneqq\lim\bar{q}$.
    	By \eqref{eq:q_a_infty},
    	\begin{equation}
    		q^\infty\forces \exists^\infty a<\omega\exists^\infty b<\omega ~(a,b)\in\dot{\phi}(s),
    	\end{equation}
    	which contradicts $\dot{\phi}(s)\in\finfin$.
    \end{proof}

    Now we prove the consistency of $\non(M_{\finfin})<\non(K_{\edfin})$ (and its dual).
    
    \begin{thm}
    \label{thm:Con_nonMfinfin_nonKEDfin}
        \begin{enumerate}
            \item Given $\kappa<\lambda=\lambda^{<\kappa}$ regular uncountable, there is a ccc poset forcing $\bb=\non(M_{\finfin})=\kappa$ and $\non(K_{\edfin})=\cc=\lambda$.
            In particular, $\non(M_{\finfin})<\non(K_{\edfin})$ is consistent.
            \item Given $\kappa$ regular uncountable and $\lambda=\lambda^{\omega}>\kappa$, there is a ccc poset forcing $\cov(K_{\edfin})=\kappa$ and $\cov(M_{\finfin})=\lambda=\continuum$.
        \end{enumerate}
    \end{thm}
    \begin{proof}
        \begin{enumerate}
            \item Using a bookkeeping argument, craft $\p$ a $\kappa$-closed-Fr-iteration, in such a way that the first $\kappa$ iterands are Cohen forcing, and the rest are a combination of $\posetEDfin$ or a subforcing of the Hechler forcing of size $<\kappa$ so it kills all witnesses of $\bb$ of size $< \kappa$, which is possible by \Cref{lem:poset_is_its_size_Frlinked} and \Cref{lem:pEDfin_has_Fr-limits}. It is easy to see that $\p_\gamma$ forces $\non(K_{\edfin})=\cc=\lambda$ and $\bb\geq\kappa$. \Cref{lem:closed_Fr_limits_keep_nonMfinfin_small} implies that the first $\kappa$-many Cohen reals witnesses $\non(M_{\finfin})\leq\kappa$.
            \item  Let $\gamma=\lambda+\kappa$ and $\p_\gamma$ be a $<\aleph_1$-closed-Fr-iteration such that the first $\lambda$-many iterands are Cohen forcings and each of the rest is $\posetEDfin$, which is possible by \Cref{lem:poset_is_its_size_Frlinked} and \Cref{lem:pEDfin_has_Fr-limits}. It is easy to see that $\p$ forces $\kappa\leq\covm\leq\cov(K_{\edfin})=\kappa$ and $\cc\leq\lambda$.
            To see $\cov(M_{\finfin})\geq\lambda$,
            let $\mu<\lambda$ be an infinite cardinal and $\{\dot{\phi}_\a:\a<\mu\}$ be $\p_\gamma$-names of functions $(\omega\times\omega)^{<\omega}\to\finfin$. 
            Let $\{\dot{c}_\b:\b<\lambda\}$ be Cohen reals as members of $(\omega\times\omega)^\omega$ added at the first $\lambda$ stages. By \Cref{lem:closed_Fr_limits_keep_nonMfinfin_small}, $\Vdash_{\p_\gamma}|\bigcup_{\a<\mu}\{\b<\lambda:\dot{c}_\b\in M_{\dot{\phi}_\a}\}|\leq\mu<\lambda$, which implies that $\{M_{\dot{\phi}_\a}:\a<\mu\}$ cannot cover $(\omega\times\omega)^\omega$. Thus $\cov(M_{\finfin})\geq\lambda$.
        \end{enumerate}
    \end{proof}
    
    Recall that $\max\{\bb,\eec_b(2)\}\leq\non(M_{\edfin})$ for any increasing function $b\in\oo$ by \Cref{prop:edfin_const}.
    By an argument similar to the previous theorem, it follows that equality cannot be proved. Moreover, we obtain the following stronger theorem: 
    
    \begin{thm}
    \label{thm:Con_max_nonKEDfin}
        \begin{enumerate}
            \item Given $\kappa<\lambda=\lambda^{<\kappa}$ regular uncountable, there is a ccc poset forcing that $\bb=\eec(2)=\eec_2=\kappa$ and $\non(K_{\edfin})=\cc=\lambda$.
            In particular, for any increasing function $b\in\oo$,  $\max\{\bb,\eec_b(2)\}<\non(K_{\edfin}) = \non(M_{\edfin})$ is consistent.
            \item Given $\kappa$ regular uncountable and $\lambda=\lambda^{\omega}>\kappa$, there is a ccc poset forcing  that $\cov(K_{\edfin})=\kappa$ and $\dd=\vvc_2=\lambda=\continuum$.
            In particular, for any increasing function $b\in\oo$,  $\cov(K_{\edfin}) = \cov(M_{\edfin}) <\max\{\dd,\vvc_b(2)\}$ is consistent.
        \end{enumerate}
    \end{thm}
    \begin{proof}
        \begin{enumerate}
            \item Note that $\eec(2)\leq\eec_b(2)\leq\eec_2(2)\leq\eec_2$ for any increasing function $b\in\oo$. Perform the same iteration $\p_\gamma$ as in \Cref{thm:Con_nonMfinfin_nonKEDfin} (1), but interleave subforcings of $\p^\omega$ (see \Cref{fac:p_omega}) of size $<\kappa$.
            Similarly, $\p_\gamma$ forces $\bb=\kappa$,  $\non(K_{\edfin})=\cc=\lambda$, and $\eec(2)\geq\kappa$. For the same reason as \Cref{thm:Con_max_b_eect_nonMED}, $\p_\gamma$ forces $\eec_2\leq\kappa$.
            \item Use the same $\p_\gamma$ in \Cref{thm:Con_nonMfinfin_nonKEDfin} (2). For the same reason as in \Cref{thm:Con_max_b_eect_nonMED}, $\p_\gamma$ forces $\vvc_2=\lambda$ and we are done.
        \end{enumerate}
    \end{proof}
    
    We do not know whether $\non(K_{\edfin})<\non(M_{\edfin})$ and $\cov(K_{\edfin})>\cov(M_{\edfin})$ are consistent or not.

    \subsection{Extending Cicho\'n's maximum}\label{sec:CM}
    
    In this section, we construct a model of Cicho\'n's maximum as described in \Cref{mainthm:ext_CM}. 
    Following the construction of the original model of Cicho\'n's maximum in \cite{GKS19},\cite{GKMS22}, we shall separate the left side of the diagram (\Cref{fig:CM_left}) and then separate the right (\Cref{fig:CM}). The separation of the right side is obtained by just applying the method in  \cite{GKMS22} (the method in \cite{GKS19} requires large cardinals),
    so our essential goal is to separate the left side to have \Cref{fig:CM_left}.
    
    Let us review the separation of the left side of (the not extended) Cicho\'n's diagram by Goldstern--Mej{\'\i}a--Shelah \cite{GMS16}. The main work is \textit{to keep the bounding number $\bb$ small} through the forcing iteration, since the naive bookkeeping iteration used to increase the cardinal invariants in the left side guarantees the smallness of the other numbers but not of $\bb$.
    As we have seen in \Cref{sec:Fr-limits}, Fr-limits keep $\bb$ small, but these limits do not work here. In the bookkeeping iteration, \textit{subforcings} of particular forcing notions are iterated, but subforcings of a poset which has Fr-limits might not have Fr-limits, and hence the smallness of $\bb$ is not guaranteed.
    However, we can overcome this problem by using the stronger limits, \textit{UF-limits}.
    By using this UF-limit method, we can choose suitable subforcings while keeping $\bb$ small (see \cite{GMS16} for details). This is why UF-limits are needed instead of Fr-limits.
    
    Let us go back to the separation of the left side of our extended Cicho\'n's diagram (\Cref{fig:CM_left}). The main work is to keep the bounding number $\bb$ and $\non(M_{\finfin})$ small through the forcing iteration (it is known that there is a forcing notion denoted by $\mathbb{LE}$ that keeps $\none$ small through the iteration, by \cite{Yam24}). Although we know that Fr-limits and closed-Fr-limits keep $\bb$ and $\non(M_{\finfin})$ small (\Cref{lem:closed_Fr_limits_keep_nonMfinfin_small}) respectively, we need UF-limits and closed-UF-limits as explained above.
    In \Cref{sec:Fr-limits}, we introduced the posets $\posetED$ and $\posetEDfin$, which have Fr-limits and closed-Fr-limits, and increase $\non(M_{\ed})$ and $\non(K_{\edfin})$, respectively (\Cref{table:Fr-limits}). However, it is unclear whether these posets have (closed-)UF-limits. More specifically, when proving that these posets have (closed-)Fr-limits, we showed $(\star)_1$ in \Cref{lem:chara_UF_linked}, but we need to show $(\star)_k$ for $k\geq2$ to prove that they have (closed-)UF-limits, which is unclear (see \Cref{rem:posetED_UF_limits} and \Cref{rem:posetEDfin_UF_limits}). 
    
    Instead, we introduce new posets which have (closed-)UF-limits, by slightly modifying the definitions of $\posetED$ and $\posetEDfin$. 
    Consequently, we change the ideals we deal with, and we encounter two ideals: $\J_L$ and $\I^f_L$ (\Cref{table:UF-limits}).
    These ideals themselves have interesting properties which are not directly related to (closed-)UF-limits (we study their properties in \Cref{sec:more_growth}), so our use of these ideals does not seem ad hoc.
    In this section, we introduce the ideals $\J_L$ and $\I^f_L$, and the posets $\posetJL$ and $\dif$, which increase $\non(M_{\J_L})$ and $\non(K_{\I_L^f})$ and have UF-limits and closed-UF-limits, respectively. In the end, we construct our desired model of Cicho\'n's maximum in \Cref{thm:CM} using these posets.
    
    \begin{table}[h]
    	\begin{minipage}[h]{0.49\textwidth}
    		\caption{Posets with (closed-)Fr-limits}\label{table:Fr-limits}
    		\centering
    		\begin{tabular}{c|c|c}
    			 posets & increase & limits    \\  
    			\hline
    			 $\posetED$ &      $\non(M_{\ed})$ & Fr-limits \\
    			 $\posetEDfin$ &      $\non(K_{\edfin})$ & closed-Fr-limits  
    		\end{tabular}
    		
    	\end{minipage}
    	\begin{minipage}[h]{0.49\textwidth}
    		\caption{Posets with (closed-)UF-limits}\label{table:UF-limits}
    		\centering
    		\begin{tabular}{c|c|c}
    			  posets & increase & limits   \\  
    			\hline
    			 $\posetJL$ &      $\non(M_{\J_L})$ & UF-limits  \\
    			 $\dif$ &      $\non(K_{\I_L^f})$ & closed-UF-limits  
    		\end{tabular}
    		
    	\end{minipage}
    \end{table}

    \subsubsection{The ideal $\J_L$ and the poset $\posetJL$}
    
    We change the forcing notion $\posetED$ so that it will have UF-limits. The following ideal was already in the introduction, and we calculated some of its associated cardinal invariants in \Cref{sec:omega}.
    
    \begin{dfn}\label{dfn:JL}
        The ideal $\J_L$ on $\omega\times\omega$ is defined by
        \[
        \J_L = \{A\subset\omega\times\omega: \exists k<\omega~\forall^{\infty} i<\omega~\lvert (A)_i\rvert \leq k\cdot i\}.
        \]
    \end{dfn}
    
    By definition, $\ed\subset\J_L\subset\finfin$. The ideal $\J_L$ is to the linear growth ideal $\I_L$ what the eventually different ideal $\ed$ is to $\edfin$. Namely, $\J_L\upharpoonright\Delta_{\mathrm{exp}}$, where $\Delta_{\mathrm{exp}} = \{\langle i, j\rangle:j\leq \exp(i)\}$, is isomorphic to $\I_L$. It follows that $\J_L\leq_{\mathrm{KB}}\I_L$. Ideals between $\ed$ and $\finfin$ in the sense of Kat\v{e}tov order (including $\J_L$) are extensively studied in \cite{DFGT21}.
    
    We introduce the forcing notion $\posetJL$ which generically adds an $\mathbf{M}_{\J_L}$-dominating function $\phi_G\colon(\omega\times\omega)^{<\omega}\to\J_L$.
    The poset is obtained by slightly changing the definition of $\posetED$.
    
    \begin{dfn}
    	\label{dfn:posetfin}
    	The poset $\posetJL$ is defined as follows: Its conditions are $p=(\sigma^p,s^p,n^p,F^p)$ such that 
        \begin{itemize}
        	\item $\sigma^p\colon(\omega\times\omega)^{<n^p}\to\omega$ is a finite partial function,
        	\item $s^p\colon\{(u^\frown a,v):(u,v)\in(\omega\times\omega)^{<n^p},a\in\omega\}\to[\omega]^{<\omega}$ is a finite partial function such that $|s^p(u^\frown a,v)|\leq a$ for $(u^\frown a,v)\in\dom(s^p)$,
        	\item $n^p\in\omega$, 
        	\item $F^p\subseteq\oo$ is finite.   
        \end{itemize}
        The order $q\leq p$ is given by  
        \begin{itemize}
        	\item $\sigma^q\supseteq \sigma^p$, $s^q\supseteq s^p$, $n^q\geq n^p$ and $F^q\supseteq F^p$,  
        	\item $\forall i\in\left[n^p,n^q\right)\forall x,y\in F^p$, either:  
        	\begin{itemize}
        		\item $(x\on i, y\on i)\in\dom(\sigma^q)$ and $x(i)\leq\sigma^q(x\on i, y\on i)$, or  
        		\item $(x\on (i+1), y\on i)\in\dom(s^q)$ and $y(i)\in s^q(x\on (i+1), y\on i)$.
        	\end{itemize}
        \end{itemize}
        Let $G$ be a generic filter. In $V[G]$,  $\phi_G\colon(\omega\times\omega)^{<\omega}\to\omega\times\omega$ is given by
        \[
        \phi_G(u,v)\coloneqq\bigcup_{p\in G}\{(a,b)\in\omega\times\omega:a\leq\sigma^p(u,v)\text{ or }b\in s^p(u^\frown a,v)\}.
        \]
    \end{dfn}
    
    As in the case of \Cref{lem:posetED_basics},
    
    \begin{lem}\label{lem:posetJL_basics}
    	\begin{enumerate}
    	   \item For $n<\omega$, $\{p:n_p\geq n\}$ is dense. Thus $\posetJL\Vdash\dom(\phi_G)=(\omega\times\omega)^{<\omega}$.
           \item For $(u,v)\in(\omega\times\omega)^{<\omega}$,  $\posetJL\Vdash\forall^\infty n<\omega~|(\phi_G(u,v))_n|\leq n$, so $\ran(\phi_G)\subseteq\J_L$.
    	   \item For $f,g\in \oo$, $\{p:f,g\in F^p\}$ is dense. 
    		Thus $\posetJL\Vdash(\omega\times\omega)^\omega\cap V\subseteq M_{\phi_G}$.
    		\end{enumerate}
    \end{lem}
    
    $\posetJL$ is $\sigma$-centered and has UF-limits:
    
    \begin{lem}\label{lem:posetfin_centered_Frlinked}
    	Fix $n\in\omega$, $\sigma\colon(\omega\times\omega)^{<n}\to\omega$ a finite partial function, and $s\colon\{(u^\frown a,v):(u,v)\in(\omega\times\omega)^{<n},a\in\omega\}\to\omega$ a finite partial function. Let $L<\omega$ and $\bar{x}^*=\{x^*_l:l<L\}\subseteq\omega^n$ such that all $x^*_l$ are pairwise different. 
    	Then, $Q\coloneqq\{p:\sigma^p=\sigma, s^p=s,n^p=n, \{x\on n:x\in F^p\}=\bar{x}^*\}$
    	is centered and $D$-lim-linked for any ultrafilter $D$. 
    \end{lem}
    \begin{proof}
        We prove $(\star)_k$ for $k\geq1$ in \Cref{lem:chara_UF_linked}.
        We follow the proof of \Cref{lem:posetED_centered_Frlinked}, so we:
        \begin{enumerate}
            \item define a limit function $\lim^D\colon Q^\omega\to\posetJL$, 
            \item\label{item:X_JL} take $X\in D$ for a given $\bar{q}\in Q^\omega$ and $r\leq \lim^D \bar{q}$, and 
            \item $X$ witnesses that $(\star)_1$ is satisfied.
        \end{enumerate}
    	We will show $(\star)_k$ for $k\geq2$ assuming $(\star)_1$.
        Let $X(\bar{q},r)$ denote the set $X\in D$ constructed in \eqref{item:X_JL} according to $\bar{q},r$.
    	Suppose that:
    	\begin{itemize}
    		\item $\bar{q}^j=\langle q_m^j:m<\omega\rangle\in Q^\omega$ for $j<k$,
    		\item $r\leq\lim^D\bar{q}^j$ for all $j<k$,
    	\end{itemize}
    	By $(\star)_1$, $X^\prime\coloneqq \bigcap_{j<k}X(\bar{q}^j, r)\in D$.
        For $m<\omega$ and $j<k$, put $q^j_m=(\sigma,s,n,F^j_m)$ and $F^j_m=\{x^{m,j}_l:l<L\}$. For each $j<k$, let $B_j\subseteq L$ and $n_{j,l}<\omega$ for $l\in B_j$ be $B$ and $n_l$ respectively in the proof of \Cref{lem:posetED_centered_Frlinked} by replacing $\bar{q}$ with $\bar{q_j}$. For $j<k$ and $l\in B_l$, let $X^{j,l}_4\coloneqq\{m<\omega:x^{m,j}_l(n_{j,l})\geq k\cdot L\}\in D$. Put $X_4\coloneqq\bigcap\{X^{j,l}_4:j<k,l\in B_l\}\in D$.
        Put $X\coloneqq X^\prime\cap X_4$.
        To see that $X$ witnesses $(\star)_k$ is true,
    	let $m\in X$, and we will find $q^\prime$ as a common extension of $r$ and all $q_m^j$. Put $r=(\sigma^r,s^r,n^r,F^r)$ and define $q^\prime\coloneqq (\sigma^\prime,s^\prime,n^r,F^r\cup \bigcup_{j<k}F^j_m)$ as follows:
        
        $\sigma^\prime\supseteq \sigma^r$ and for all $(u,v)\in(\omega\times\omega)^{<\omega}\setminus \dom (\sigma^r)$ of the form $(x^{m,j}_{l}\on i, x^{m,j}_{l^\prime}\on i)$ for some $j<k$, $l,l^\prime<L$, $i<n_r$, let $(u,v)\in\dom(\sigma^\prime)$ and:
        \begin{equation}
            \sigma^\prime(u,v)\coloneqq\max\{x^{m,j^*}_{l^*}(i):j^*<k,l^*<L\}. 
        \end{equation}
        $s^\prime\supseteq s^r$ and for all $(u^\prime,v)\in\fsq\times\fsq\setminus \dom (s^r)$ of the form $(x^{m,j}_{l}\on (n_{j,l}+1), x^{m,j}_{l^\prime}\on n_{j,l})$ for some $j<k$, $l\in B_j$ and $l^\prime<L$, let $(u^\prime,v)\in\dom(s^\prime)$ and:
        \begin{equation}
            \label{equation:s_prime_JL}
            s^\prime(u^\prime,v)\coloneqq\bigcup\{x^{m,j^*}_{l^*}(n_{j,l}):j^*<k,l^*<L\}.
        \end{equation}
        For such $(u^\prime,v)=(x^{m,j}_{l}\on (n_{j,l}+1), x^{m,j}_{l^\prime}\on n_{j,l})$, by the choice of $X_4$, we have $x^{m,j}_{l} (n_{j,l})\geq k\cdot L\geq |s^\prime(u^\prime,v)|$ and hence \eqref{equation:s_prime_JL} is a valid definition of $s^\prime$.

        Then, the argument of the case $(\star)_1$ in the proof of \Cref{lem:posetED_centered_Frlinked} implies that $q^\prime$ extends $r$ and all $q^j_m$. The only difference appears in the final case of that proof, in which case we have $x^{m,j}_{l^\prime}(i)\leq s^\prime(x^{m,j}_l\on(i+1),x^{m,j}_{l^\prime}\on i)$ by our current choice of $s^\prime$ and hence $q^\prime\leq q^j_m$ by the definition of $\leq_{\posetJL}$.
    \end{proof}
    
    \begin{rem}\label{rem:posetED_UF_limits}
        In the previous argument, when we define $s^\prime(u^\prime,v)$ in  \eqref{equation:s_prime_JL}, even in the case that $(x^{m,j}_{l_0}\on (n_{j,l}+1), x^{m,j}_{l_1}\on n_{j,l})=(x^{m,j^*}_{l_2}\on (n_{j^*,l}+1), x^{m,j^*}_{l_3}\on n_{j^*,l})$ for some $j,j^*<k$, $l_0\in B_j$, $l_2\in B_{j^*}$ and $l_1,l_3<L$, it could be the case that $x^{m,j}_{l_1}(n_{j,l})$ and $ x^{m,j^*}_{l_3}(n_{j^*,l})$ are different. For this reason, we are not sure if $(\star)_k$ for $k\geq2$ is true for $\posetED$.
    \end{rem}

    \subsubsection{The $f$-linear growth ideal $\ILf$ and the poset $\dif$.}
    
    We will now modify the forcing $\posetEDfin$ so that it will have closed-UF-limits. We introduce the \emph{$f$-linear growth ideal} $\I^f_L$ as well as its associated poset to increase $\non(K_{\I^f_L})$.
    
    \begin{dfn}
        For $g,h\in\oo$, we write
        $g\ll h$ if $\lim_{n\to\infty}\frac{g(n)}{h(n)}=0$. Let $\exp\in\oo$ denote the exponential function, i.e. $\exp(n)=2^n$. 
    \end{dfn}
    
    \begin{dfn}
    	Fix an interval partition $\bar{P}^\mathrm{exp}=(P_i^\mathrm{exp})_{i<\omega}$ of $\omega$ such that $|P^\mathrm{exp}_i|=\exp(i)$. Let $f\in\oo$ be a function such that $f\ll\exp$.
        The \emph{$f$-linear growth ideal} $\IL^f$ is given by:
    	\begin{equation*}
    		\IL^f\coloneqq\set{A\subseteq\omega}{\exists k<\omega~\forall^\infty i<\omega~\card{A\cap P_i^\mathrm{exp}}\leq k\cdot f(i)}.
    	\end{equation*}
    \end{dfn}

    In particular, when $f$ is the identity function, $\IL^f$ is the linear growth ideal (cf.\ {\cite[Page 56]{Hru11}, \cite[Page 3]{BM14}}), from which the subscript $L$ comes.
    Note that $f\leq^* g$ implies $\IL^f\subseteq\IL^g$ and if $f$ is bounded then $\IL^f\cong_{\mathrm{KB}}\edfin$. Moreover, if $f\ll\exp$ then $\IL^f\leq_{\mathrm{KB}}\Z$; this follows almost directly from the definition.
    
    Fix $f\in\oo$ with $f\ll\exp$. 
    We introduce the forcing notion $\dif$ which generically adds a $\mathbf{K}_{\I_L^f}$-dominating function $\phi_G\colon\omega\to\IL^f$.
    We may assume that $\posetEDfin$ is defined on $(P_i^\mathrm{exp})_{i<\omega}$, instead of the interval partition $(P_i)_{i<\omega}$ with $|P_i|=i$. Under this assumption, the conditions of $\dif$ are exactly the same as $\posetEDfin$ but the order is weaker. Let $P_i^\mathrm{exp}\coloneqq I_i\coloneqq\left[M_i,M_{i+1}\right)$ for each $i<\omega$.

    \begin{dfn} 
        The poset $\dif$ is defined as follows:
        \begin{itemize}
		    \item Conditions are $p=(i,k,s,\varphi)=(i_p,k_p,s_p,\varphi_p)$ where $i,k<\omega$, $s\in([M_i]^{\leq k})^{<\omega}$ and $\varphi\colon \omega\to [\omega]^{\leq k}$ such that $\varphi\on |s|=s$.
		    \item The order $(i^\prime,k^\prime,s^\prime,\varphi^\prime)\leq(i,k,s,\varphi)$ is defined by: $i^\prime\geq i$, $k^\prime\geq k$, $\varphi^\prime(n)\supseteq\varphi(n)$ for all $n<\omega$, $|s^\prime|\geq|s|$ and for $n<|s|$,
		      \begin{gather}
			    \label{gath_end_ext} s^\prime(n)\cap M_i=s(n),\text{ and}\\ 
			    \label{gath_density_small} \text{if }i\leq j< i^\prime,\text{ then } |s^\prime(n)\cap I_j|\leq n \cdot f(j).
		    \end{gather}
	        \item For a generic filter $G$, define $\phi_G\colon\omega\to\mathcal{P}(\omega)$ by:
            \[\phi_G(n)\coloneq\bigcup \{s_p(n):p\in G\text{ and }n<|s|\}.\]
	   \end{itemize}
    \end{dfn}

    One can easily prove:
    \begin{lem}
    	\label{lem:DI_basics}
    	The following sets are open dense:
    	\begin{enumerate}
    		\item\label{item_lem_DI_basics_x} For $x\in\oo$, $\{p:x(n)\in\varphi_p(n)\text{ for all }n\geq |s_p|\}$, 
    		\item for $i<\omega$, $\{p:i_p\geq i\}$, 
    		\item\label{item_lem_DI_basics_s} for $n<\omega$, $\{p:|s_p|\geq n\}$.
    	\end{enumerate}
    \end{lem}

    By \eqref{gath_density_small} and by Lemma \ref{lem:DI_basics} \eqref{item_lem_DI_basics_x}, we have:
    
    \begin{lem}\label{lem:posetILf_basics}
    	In $V[G]$,  $\ran(\phi_G)\subseteq\IL^f$ and $\oo\cap V\subseteq K_{\phi_G}$.
    \end{lem}

    \begin{dfn}\label{dfn:Q_iks}
    	For $i,k<\omega$ and $s\in([M_i]^{\leq k})^{<\omega}$,
    	define $Q=Q_{i,k,s}\coloneq\{p\in\dif:i_p=i, k_p=k, s_p=s\}$.
    \end{dfn}
    
    \begin{lem}\label{lem:DI_sigma_centered}
    	$Q$ is centered. In particular, $\dif$ is $\sigma$-centered.
    \end{lem}
    
    When $f$ is unbounded, $\dif$ has closed-UF-limits.
    
    \begin{lem}
    	\label{lem:DI_has_UF-limits}
        Assume $f\gg1$. Let $i,k<\omega$, $s\in([M_i]^{\leq k})^{<\omega}$ with $|s|\geq 1$ and $D$ be an ultrafilter. Then, $Q=Q_{i,k,s}$ is closed-$D$-lim-linked. In particular, $\dif$ has closed-UF-limits.
    \end{lem}
    \begin{proof}
        The argument is based on the proof of \Cref{lem:pEDfin_has_Fr-limits}. 
        Under the assumption that $\posetEDfin$ is defined on $(P_i^\mathrm{exp})_{i<\omega}$, we make use of the same limit function $\lim^D$ as in the proof of \Cref{lem:pEDfin_has_Fr-limits}. 
    	Let $0<L<\omega$, we will show $(\star)_L$ as required in Lemma \ref{lem:chara_UF_linked}: Suppose that  $\bar{q}^l=\langle(i,k,s,\varphi_m^l)\rangle_{m<\omega}\in Q^\omega$ and $q_\infty^l=\lim^D\bar{q}_l=(i,k,s,\varphi_\infty^l)$ for $l<L$ and $q^\prime=(i^\prime,k^\prime,s^\prime,\varphi^\prime)\leq q_\infty^l$ for all $l<L$.
    	As $f\gg1$, we may assume that: 
        \begin{equation}
        \label{eq:f_gg}
            kL\leq |s|\cdot f(j)\text{ for any }j\geq i^\prime,
        \end{equation}
        by increasing $i^\prime$.
        Take $X\in D$ such that:
    	\begin{equation}
    		\text{For all }m\in X, n\in|s^\prime|\setminus|s|\text{ and }l<L, ~\varphi_m^l(n)\cap  M_{i^\prime}=\varphi_\infty^l(n).
    	\end{equation}
    	Fix $m\in X$ and we will define a common extension $r$ of $q^\prime$ and all $q^l_m$ for $l<L$. Define $\psi\colon\omega\to [\omega]^{\leq k^\prime+kL}$ by $\psi(n)\coloneq\varphi^\prime(n)\cup\bigcup_{l<L}\varphi^l_m(n)$ and $t\coloneq\psi\on|s^\prime|$. Put $r\coloneq(i^{\prime\prime},k^\prime+kL,t,\psi)$, where $i^{\prime\prime}\geq i^\prime$ large enough so $r$ is a valid condition. 
        Then, the argument of the proof of \Cref{lem:pEDfin_has_Fr-limits} implies that $r\leq q^\prime,q^l_m$ for $l<L$.
        The only difference appears when, in order to see $r\leq q^\prime$, we prove the following:
        \[\text{ for }n\in |t|\setminus|s|,\text{ if }i^\prime\leq j\leq i^{\prime\prime}, \text{then }|t(n)\cap I_j|\leq n \cdot f(j).\]
    	It follows from $|t(n)\cap I_j|\leq|\bigcup_{l<L}\varphi^l_m(n)|\leq kL\leq |s|\cdot f(j)\leq n \cdot f(j)$ by \eqref{eq:f_gg}.
    \end{proof}
    
    \begin{rem}\label{rem:posetEDfin_UF_limits}
        When proving $(\star)_L$ for $L>1$, it seems that $f\gg 1$ is necessary to have \eqref{eq:f_gg}. Because of this, we are not sure if the assumption $f\gg 1$ can be dropped. 
        For this reason, we are not sure if $(\star)_k$ for $k\geq2$ is true for $\posetEDfin$.
    \end{rem}

    \subsubsection{Extended Cicho\'n's maximum}

    Now we are ready to construct a model of extended Cicho\'n's maximum. As explained in the beginning of this section, the forcing construction consists of two steps: the first one is to separate the left side of the diagram (\Cref{thm:CM_left}, \Cref{fig:CM_left}) with some additional properties ($\R_\ind\cong_T [\lambda_7]^{<\lambda_\ind}$ in \Cref{thm:CM_left}) and the second one is to separate the right side (\Cref{thm:CM}, \Cref{fig:CM}) using these properties. The framework of the forcing construction is based on \cite{Yam25}, \cite{Yam24} (as well as the notation). In this paper, we do not describe the framework in detail (see \cite{Yam25}, \cite{Yam24} for the details), but we list necessary items to state and prove \Cref{thm:CM_left}, \ref{thm:CM}.

    The locally eventually different forcing $\lebb$ was introduced in \cite{Yam24} to increase $\nonm$, while keeping $\none$ small.
    
    \begin{dfn}
        The poset $\lebb$ is defined as follows:
		\begin{enumerate}
			\item The conditions are triples $(d,s,\varphi)$ where $d\in\sq$, $s\in\prod_{n<|d|} \exp(n)$ and $\varphi\in\prod_{n<\omega}\mathcal{P}(\exp(n))$ such that, for some $k<\omega$:
			\begin{equation*}
				\dfrac{|\varphi(n)|}{\exp(n)}\leq\exp\left(-\dfrac{n}{2^k}\right)\text{ for all }n\geq|d|.
			\end{equation*} 
			
            \item $(d^\prime,s^\prime,\varphi^\prime)\leq(d,s,\varphi)$ if $d^\prime\supseteq d$, $s^\prime\supseteq s, \varphi^\prime(n)\supseteq\varphi(n)$ for $n<\omega$ and:
			\begin{equation*}
		          \text{for all }n\in(d^\prime)^{-1}(\{1\})\setminus d^{-1}(\{1\}), s^\prime(n)\notin\varphi(n).
			\end{equation*}
		\end{enumerate}	
    \end{dfn}

    \begin{lem}[{\cite[Lemma 4.27,4.29]{Yam24}}]
    \label{lem:lebb_sigma_centered_cUF}
        $\lebb$ is $\sigma$-centered and $\sigma$-closed-UF-lim-linked (witnessed by the same countable components).
    \end{lem}

    \begin{dfn}
		Put the following relational systems and forcing notions:
		\begin{itemize}
			\item $\R_1\coloneq\n$ and $\br_1\coloneq\mathbb{A}$, the Amoeba forcing.
			\item $\R_2\coloneq C_\n^\bot$ and $\br_2\coloneq\mathbb{B}$, the random forcing.
			\item $\R_3\coloneq\mathbf{D}$ and $\br_3\coloneq\mathbb{D}$, the Hechler forcing.
			\item $\R_4\coloneq\mathbf{M}_{\J_L}$, $\R_4^*\coloneq\mathbf{M}_{\finfin}$ and $\br_4\coloneq\mathbb{\posetJL}$.
			\item $\R_5\coloneq C_\mathcal{E}$, $\R_5^{\Z}\coloneqq \mathbf{M}_{\Z}$, $\R_5^f\coloneq \mathbf{K}_{\I_L^f}$, and $\br_5^f\coloneq\dif$ for $f\in\oo$ with $1\ll f\ll\exp$.
			\item $\R_6\coloneq C_\m$, and $\br_6\coloneq\lebb$.
		\end{itemize}
		Let $I^+\coloneq \{1,2,\ldots,6\}$ be the index set.
	\end{dfn}
        
    \begin{thm}\label{thm:CM_left}
        Assume:
		\begin{itemize}
			\item $\lambda_1\leq\cdots\leq\lambda_6$ are uncountable regular cardinals and $\lambda_7\geq\lambda_6$ is a cardinal.
			\item $\lambda_3=\mu_3^+$, $\lambda_4=\mu_4^+$ and $\lambda_5=\mu_5^+$ are successor cardinals and $\mu_3$ is regular.
            \item \label{item_aleph1_inacc}$\kappa<\lambda_\ind$ implies $\kappa^{\aleph_0}<\lambda_\ind$ for all $\ind\in I^+$.
			\item \label{item_ca_3}
			$\lambda_7^{<\lambda_6}=\lambda_7$, hence $\lambda_7^{<\lambda_\ind}=\lambda_7$ for all $\ind\in I^+$.
			\item $\lambda_7\leq2^{\mu_3}$.
		\end{itemize}
        Then, there is a ccc poset which forces that for each $\ind\in I^+$, $\R_\ind\cong_T C_{[\lambda_7]^{<\lambda_\ind}}\cong [\lambda_7]^{<\lambda_\ind}$, in particular, $\bb(\R_\ind)=\lambda_\ind$ and $\dd(\R_\ind)=\cc=\lambda_7$ (see \Cref{fig:CM_left}, and the same also holds for $\R_4^*$, $\R_5^{\Z}$ and $\R_5^f$ for $f\in\oo$ with $1\ll f\ll\exp$).
    \end{thm}
    \begin{proof}
        The following construction of the forcing iteration is based on \cite{Yam24} and it is obtained by replacing $\mathbb{PR}$ with $\posetJL$ and $\mathbb{PR}_g$ with $\dif$ in \cite[Construction 5.6]{Yam24}. More specifically, perform a finite support iteration $\langle\p_\a,\dot{\q}_\xi:\a\leq\gamma, \xi<\gamma \rangle$ of length $\gamma\coloneqq\lambda_7+\lambda_7$ such that for each $\xi<\lambda_7$, $\qd_\xi$ is Cohen forcing and for each $\xi\in\gamma\setminus\lambda_7$, $\qd_\xi$ is a subforcing of $\br_i$ of size $<\lambda_i$ where $\xi\equiv i\in I^+$ modulo $6$. By bookkeeping, we can have $\R_\ind\lq C_{[\lambda_7]^{<\lambda_\ind}}$ for $\ind\in I^+$ (use \Cref{lem:posetJL_basics} (3) for $\ind=4$ and \Cref{lem:posetILf_basics} for $\ind=5$). By \Cref{fac:smallness_for_addn_and_covn_and_nonm} and \Cref{lem:posetfin_centered_Frlinked}, \ref{lem:DI_sigma_centered}, \ref{lem:lebb_sigma_centered_cUF}, we have $C_{[\lambda_7]^{<\lambda_\ind}}\lq\R_\ind$ for $\ind=1,2,6$. As in \cite[Construction 5.6]{Yam24}, we can inductively define names of ultrafilters for UF-limits and closed-UF-limits through the iteration construction, which are possible by \Cref{lem:posetfin_centered_Frlinked}, \ref{lem:DI_has_UF-limits}, \ref{lem:lebb_sigma_centered_cUF} (and \Cref{lem:DI_sigma_centered}). $C_{[\lambda_7]^{<\lambda_\ind}}\lq\R_\ind$ for $\ind=3,5$ follows as in \cite[Construction 5.6]{Yam24}. By modifying the proof of \Cref{lem:closed_Fr_limits_keep_nonMfinfin_small},  we have $C_{[\lambda_7]^{<\lambda_\ind}}\lq\R_4$ (see also \cite[Main Lemma 3.26]{Yam25}).
    \end{proof}

    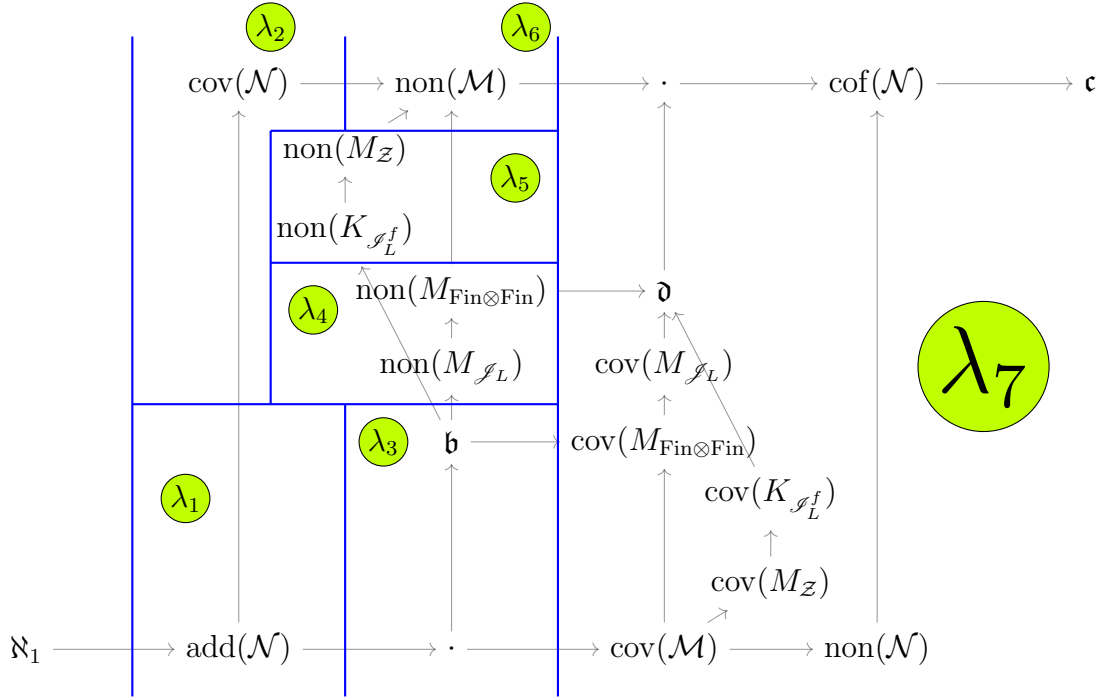
\begin{figure}[h!]
        \centering	
		\begin{tikzpicture}
			\tikzset{
				textnode/.style={text=black}, 
			}
			\tikzset{
				edge/.style={color=black, thin, opacity=0.4}, 
			}
			\newcommand{\w}{2.8}
			\newcommand{\h}{2.5}
			
			\node[textnode] (addN) at (0,  0) {$\addn$};
			\node (t1) [fill=lime, draw, text=black, circle,inner sep=1.0pt] at (-0.25*\w, 0.8*\h) {$\lambda_1$};
			
			\node[textnode] (covN) at (0,  \h*3) {$\covn$};
			\node (t2) [fill=lime, draw, text=black, circle,inner sep=1.0pt] at (0.15*\w, 3.3*\h) {$\lambda_2$};

			\node[textnode] (addM) at (\w,  0) {$\cdot$};
			
			\node (t3) [fill=lime, draw, text=black, circle,inner sep=1.0pt] at (0.68*\w, 1.1*\h) {$\lambda_3$};
			
			\node[textnode] (nonM) at (\w,  \h*3) {$\nonm$};
			\node (t6) [fill=lime, draw, text=black, circle,inner sep=1.0pt] at (1.35*\w, 3.3*\h) {$\lambda_6$};
			
			\node[textnode] (covM) at (\w*2,  0) {$\covm$};
			\node (t7) [fill=lime, draw, text=black, circle,inner sep=1.0pt] at (3.5*\w, 1.5*\h) {{\scalebox{3.0}{$\lambda_7$}}};
			
			\node[textnode] (d) at (\w*2,  1.9*\h) {$\dd$};
			\node[textnode] (cofM) at (\w*2,  \h*3) {$\cdot$};

			\node[textnode] (nonN) at (\w*3,  0) {$\nonn$};

			\node[textnode] (cofN) at (\w*3,  \h*3) {$\cofn$};

			\node[textnode] (aleph1) at (-\w,  0) {$\aleph_1$};
			\node[textnode] (c) at (\w*4,  \h*3) {$\continuum$};

			\node (t4) [fill=lime, draw, text=black, circle,inner sep=1.0pt] at (0.35*\w, 1.8*\h) {$\lambda_4$};

			\node (t5) [fill=lime, draw, text=black, circle,inner sep=1.0pt] at (1.3*\w, 2.5*\h) {$\lambda_5$};

            \node[textnode] (b) at (\w,  1.1*\h) {$\bb$};

            \node[textnode] (nonMFinFin) at (\w,  1.9*\h) {$\non(M_{\finfin})$}; 
            \node[textnode] (covMFinFin) at (2*\w,  1.1*\h) {$\cov(M_{\finfin})$};

            \node[textnode] (nonMZ) at (\w*0.5,  \h*2.65) {$\non(M_{\Z})$};
            
			\node[textnode] (covMZ) at (\w*2.5,  \h*0.35) {$\cov(M_{\Z})$};
            \node[textnode] (nonMILf) at (\w*0.5,  \h*2.2) {$\non(K_{\I_L^f})$};
            \node[textnode] (covMILf) at (\w*2.5,  \h*0.8) {$\cov(K_{\I_L^f})$};
			
			\node[textnode] (nonMJL) at (\w,  1.5*\h) {$\non(M_{\J_L})$};
            \node[textnode] (covMJL) at (2.0*\w,  1.5*\h) {$\cov(M_{\J_L})$};

			\draw[->, edge] (addN) to (covN);
			\draw[->, edge] (addN) to (addM);
			\draw[->, edge] (covN) to (nonM);	
			\draw[->, edge] (addM) to (b);
			\draw[->, edge] (nonMFinFin) to (nonM);
			\draw[->, edge] (addM) to (covM);
			\draw[->, edge] (nonM) to (cofM);
			\draw[->, edge] (covM) to (covMFinFin);
			\draw[->, edge] (d) to (cofM);
			\draw[->, edge] (covM) to (nonN);
			\draw[->, edge] (cofM) to (cofN);
			\draw[->, edge] (nonN) to (cofN);
			\draw[->, edge] (aleph1) to (addN);
			\draw[->, edge] (cofN) to (c);
			
			\draw[->, edge] (nonMFinFin) to (d);
			\draw[->, edge] (nonMJL) to (nonMFinFin);

			\draw[->, edge] (b) to (covMFinFin);
			\draw[->, edge] (covMFinFin) to (covMJL);
			
			\draw[->, edge] (nonMZ) to (nonM);
			
			\draw[->, edge] (covM) to (covMZ);

			\draw[->, edge] (b) to (nonMILf);
			\draw[->, edge] (nonMILf) to (nonMZ);

            \draw[->, edge] (b) to (nonMJL);

            \draw[->, edge] (covMILf) to (d);
			\draw[->, edge] (covMZ) to (covMILf);

            \draw[->, edge] (covMJL) to (d);

			\draw[blue,thick] (-0.5*\w,1.3*\h)--(1.5*\w,1.3*\h);
			\draw[blue,thick] (1.5*\w,-0.25*\h)--(1.5*\w,3.25*\h);
			
			\draw[blue,thick] (-0.5*\w,-0.25*\h)--(-0.5*\w,3.25*\h);
			
			\draw[blue,thick] (0.5*\w,-0.25*\h)--(0.5*\w,1.3*\h);
			
			\draw[blue,thick] (0.15*\w,2.05*\h)--(1.5*\w,2.05*\h);
			\draw[blue,thick] (0.15*\w,2.75*\h)--(1.5*\w,2.75*\h);
			\draw[blue,thick] (0.15*\w,1.3*\h)--(0.15*\w,2.75*\h);
			\draw[blue,thick] (0.5*\w,2.75*\h)--(0.5*\w,3.25*\h);
		\end{tikzpicture}
        \caption{Separation Constellation of \Cref{thm:CM_left}. Note that $\non(M_\Z)=\max\{\bb,\none\}$ by \Cref{thm:nonMZ_exact_value}.}\label{fig:CM_left}
\end{figure}

    \begin{thm}
	\label{thm:CM} 
			Let $\aleph_1\leq\theta_1\leq\cdots\leq\theta_{12}$ be regular cardinals and $\theta_\cc$ an infinite cardinal with $\theta_\cc\geq\theta_{12}$ and $\theta_\cc^{\aleph_0}=\theta_\cc$.
			Then, there exists a ccc poset which forces $\bb(\R_\ind)=\theta_\ind$ and $\dd(\R_\ind)=\theta_{13-\ind}$ for each $\ind\in I^+$ (the same also holds for $\R_4^*$, $\R_5^{\Z}$, $\R_5^f$ for $f\in\oo$ with $1\ll f\ll\exp$) and $\cc=\theta_\cc$ (see \Cref{fig:CM}).
            In particular, in the forcing extension, for ideals $\I_0, \I_1$ on countable sets such that $\J_L \leq_{\mathrm{K}} \I_0 \leq_{\mathrm{K}} \finfin$ and $\I_L^f \leq_{\mathrm{K}} \I_1 \leq_{\mathrm{K}} \Z$ for some $1\ll f\ll \exp$, we have:
        \[\non(M_{\I_0})=\theta_4,~\cov(M_{\I_0})=\theta_9,\]
        \[\non(K_{\I_1})=\non(M_{\I_1})=\theta_5,~\cov(K_{\I_1})=\cov(M_{\I_1})=\theta_8.\]
	\end{thm}
    \begin{proof}
        See \cite{GKMS22} (and \cite[Theorem 5.9]{Yam24} for the argument without $\mathsf{GCH}$).
    \end{proof}

    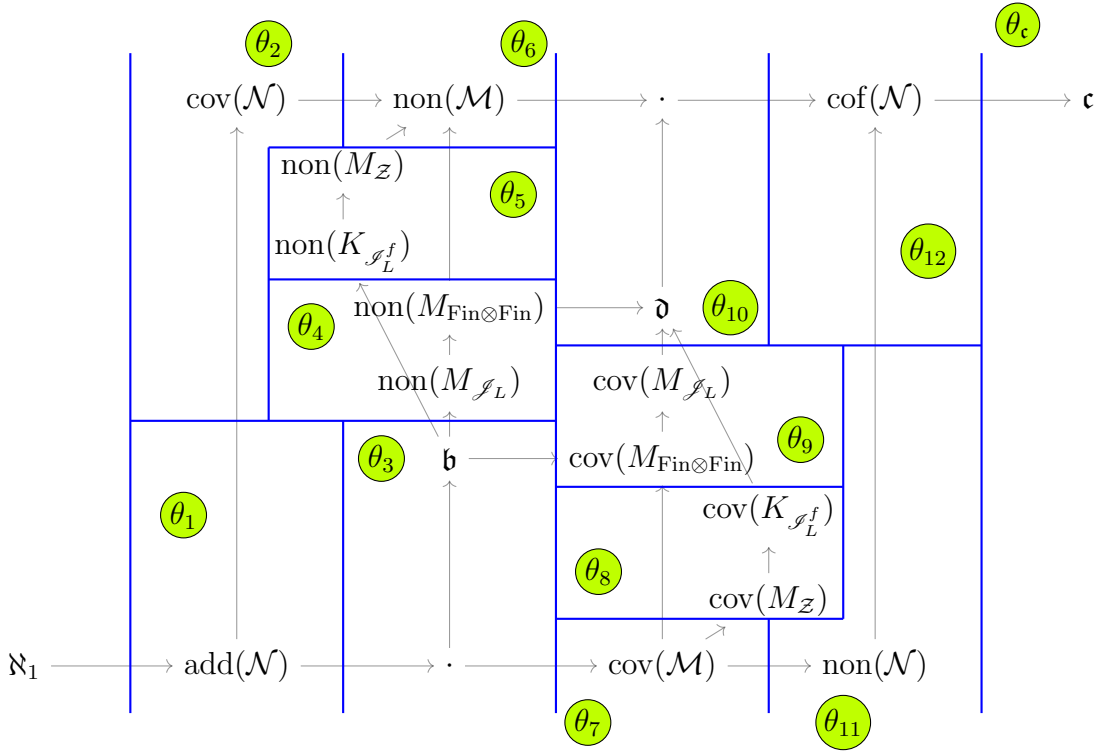
\begin{figure}[h!]
        \centering	
		\begin{tikzpicture}
			\tikzset{
				textnode/.style={text=black}, 
			}
			\tikzset{
				edge/.style={color=black, thin, opacity=0.4}, 
			}
			\newcommand{\w}{2.8}
			\newcommand{\h}{2.5}
			
			\node[textnode] (addN) at (0,  0) {$\addn$};
			\node (t1) [fill=lime, draw, text=black, circle,inner sep=1.0pt] at (-0.25*\w, 0.8*\h) {$\theta_1$};
			
			\node[textnode] (covN) at (0,  \h*3) {$\covn$};
			\node (t2) [fill=lime, draw, text=black, circle,inner sep=1.0pt] at (0.15*\w, 3.3*\h) {$\theta_2$};

			\node[textnode] (addM) at (\w,  0) {$\cdot$};
			
			\node (t3) [fill=lime, draw, text=black, circle,inner sep=1.0pt] at (0.68*\w, 1.1*\h) {$\theta_3$};
			
			\node[textnode] (nonM) at (\w,  \h*3) {$\nonm$};
			\node (t6) [fill=lime, draw, text=black, circle,inner sep=1.0pt] at (1.35*\w, 3.3*\h) {$\theta_6$};
			
			\node[textnode] (covM) at (\w*2,  0) {$\covm$};
			\node (t7) [fill=lime, draw, text=black, circle,inner sep=1.0pt] at (1.65*\w, -0.3*\h) {$\theta_7$};
			
			\node[textnode] (d) at (\w*2,  1.9*\h) {$\dd$};
			\node (t10) [fill=lime, draw, text=black, circle,inner sep=1.0pt] at (2.32*\w, 1.9*\h) {$\theta_{10}$};
			\node[textnode] (cofM) at (\w*2,  \h*3) {$\cdot$};

			\node[textnode] (nonN) at (\w*3,  0) {$\nonn$};
			\node (t11) [fill=lime, draw, text=black, circle,inner sep=1.0pt] at (2.85*\w, -0.3*\h) {$\theta_{11}$};
			
			\node[textnode] (cofN) at (\w*3,  \h*3) {$\cofn$};
			\node (t12) [fill=lime, draw, text=black, circle,inner sep=1.0pt] at (3.25*\w, 2.2*\h) {$\theta_{12}$};
			
			\node[textnode] (aleph1) at (-\w,  0) {$\aleph_1$};
			\node[textnode] (c) at (\w*4,  \h*3) {$\continuum$};
			\node (t10) [fill=lime, draw, text=black, circle,inner sep=1.0pt] at (3.67*\w, 3.4*\h) {$\theta_\cc$};

			\node (t4) [fill=lime, draw, text=black, circle,inner sep=1.0pt] at (0.35*\w, 1.8*\h) {$\theta_4$};
			
			\node (t9) [fill=lime, draw, text=black, circle,inner sep=1.0pt] at (2.65*\w, 1.2*\h) {$\theta_9$};
			
			\node (t5) [fill=lime, draw, text=black, circle,inner sep=1.0pt] at (1.3*\w, 2.5*\h) {$\theta_5$};
            
			\node (t8) [fill=lime, draw, text=black, circle,inner sep=1.0pt] at (1.7*\w, 0.5*\h) {$\theta_8$};

            \node[textnode] (b) at (\w,  1.1*\h) {$\bb$};

            \node[textnode] (nonMFinFin) at (\w,  1.9*\h) {$\non(M_{\finfin})$}; 
            \node[textnode] (covMFinFin) at (2*\w,  1.1*\h) {$\cov(M_{\finfin})$};

            \node[textnode] (nonMZ) at (\w*0.5,  \h*2.65) {$\non(M_{\Z})$};
            
			\node[textnode] (covMZ) at (\w*2.5,  \h*0.35) {$\cov(M_{\Z})$};
            \node[textnode] (nonMILf) at (\w*0.5,  \h*2.2) {$\non(K_{\I_L^f})$};
            \node[textnode] (covMILf) at (\w*2.5,  \h*0.8) {$\cov(K_{\I_L^f})$};
			
			\node[textnode] (nonMJL) at (\w,  1.5*\h) {$\non(M_{\J_L})$};
            \node[textnode] (covMJL) at (2.0*\w,  1.5*\h) {$\cov(M_{\J_L})$};

			\draw[->, edge] (addN) to (covN);
			\draw[->, edge] (addN) to (addM);
			\draw[->, edge] (covN) to (nonM);	
			\draw[->, edge] (addM) to (b);
			\draw[->, edge] (nonMFinFin) to (nonM);
			\draw[->, edge] (addM) to (covM);
			\draw[->, edge] (nonM) to (cofM);
			\draw[->, edge] (covM) to (covMFinFin);
			\draw[->, edge] (d) to (cofM);
			\draw[->, edge] (covM) to (nonN);
			\draw[->, edge] (cofM) to (cofN);
			\draw[->, edge] (nonN) to (cofN);
			\draw[->, edge] (aleph1) to (addN);
			\draw[->, edge] (cofN) to (c);

			\draw[->, edge] (nonMFinFin) to (d);
			\draw[->, edge] (nonMJL) to (nonMFinFin);

			\draw[->, edge] (b) to (covMFinFin);
			\draw[->, edge] (covMFinFin) to (covMJL);
			
			\draw[->, edge] (nonMZ) to (nonM);
			
			\draw[->, edge] (covM) to (covMZ);
            
			\draw[->, edge] (b) to (nonMILf);
			\draw[->, edge] (nonMILf) to (nonMZ);

            \draw[->, edge] (b) to (nonMJL);

            \draw[->, edge] (covMILf) to (d);
			\draw[->, edge] (covMZ) to (covMILf);

            \draw[->, edge] (covMJL) to (d);

			\draw[blue,thick] (-0.5*\w,1.3*\h)--(1.5*\w,1.3*\h);
			\draw[blue,thick] (3.5*\w,1.7*\h)--(1.5*\w,1.7*\h);
			\draw[blue,thick] (1.5*\w,-0.25*\h)--(1.5*\w,3.25*\h);
			
			\draw[blue,thick] (-0.5*\w,-0.25*\h)--(-0.5*\w,3.25*\h);
			\draw[blue,thick] (3.5*\w,-0.25*\h)--(3.5*\w,3.25*\h);
			
			\draw[blue,thick] (0.5*\w,-0.25*\h)--(0.5*\w,1.3*\h);
			\draw[blue,thick] (2.5*\w,1.7*\h)--(2.5*\w,3.25*\h);
			
			\draw[blue,thick] (0.15*\w,2.05*\h)--(1.5*\w,2.05*\h);
			\draw[blue,thick] (0.15*\w,2.75*\h)--(1.5*\w,2.75*\h);
			\draw[blue,thick] (0.15*\w,1.3*\h)--(0.15*\w,2.75*\h);
			\draw[blue,thick] (0.5*\w,2.75*\h)--(0.5*\w,3.25*\h);
			
			\draw[blue,thick] (2.85*\w,0.95*\h)--(1.5*\w,0.95*\h);
			\draw[blue,thick] (2.85*\w,0.25*\h)--(1.5*\w,0.25*\h);
			\draw[blue,thick] (2.85*\w,1.7*\h)--(2.85*\w,0.25*\h);
			\draw[blue,thick] (2.5*\w,0.25*\h)--(2.5*\w,-0.25*\h);
		\end{tikzpicture}
        \caption{Separation Constellation of \Cref{thm:CM}.}\label{fig:CM}
    \end{figure}

    \section{More on the ideals $\J_L$ and $\I^f_L$}\label{sec:more_growth}

    In this final section, we collect several results related to $\J_L$ and the $f$-linear growth ideals $\I_L^f$ introduced in \Cref{sec:CM}. We start with considering the following variants of $\J_L$.

    \begin{dfn}\label{dfn:JLf}
        For any function $f\in\omega^\omega$ with $f\gg1$, an ideal $\J_L^f$ on $\omega\times\omega$ is defined by
        \[
        \J_L^f = \{A\subset\omega\times\omega: \exists k<\omega\,\forall^{\infty} i<\omega\,(\lvert (A)_i\rvert \leq k\cdot f(i))\}.
        \]
        We simply write $\J_L$ for $\J_L^f$ in the case that $f$ is the identity function.
    \end{dfn}
    
   These ideals appeared in \cite{DFGT21}. According to their notation, $\J_L^f = I(\mathcal{F}_f)$. By definition, $\ed\subseteq\J_L^f\subseteq\finfin$ for any $f\in\omega^\omega$ with $f\gg 1$. Also, note that if $f\ll\exp$ then $\J_L^f\upharpoonright\Delta_{\mathrm{exp}}$, where $\Delta_{\mathrm{exp}} = \{\langle i, j\rangle:j\leq \exp(i)\}$, is isomorphic to $\I_L^f$. 
   
   Now recall that $\ed$ is a critical ideal for local selectivity in the following sense.

    \begin{dfn}[{\cite{BTW82}}]
    	Let $\bar{P}=(P_i)_{i<\omega}$ be a partition of $\omega$. We say that $A\subseteq\omega$ is a \emph{selector of $\bar{P}$} if $|A\cap P_i|\leq 1$ for all $i<\omega$. An ideal $\I$ on $\omega$ is called \emph{locally selective} if for every partition $\bar{P}$ of $\omega$ into sets in $\I$, there is an $\I$-positive selector of $\bar{P}$.
    \end{dfn}
    
    \begin{fac}[{\cite[Proposition 3.2]{Hru11}}]
    	\label{fac:ED_chara}
        For any ideal $\I$ on $\omega$, $\mathcal{ED}\leq_{\mathrm{K}}\I$ if and only if $\I$ is not locally selective.
    \end{fac}

    Analogously, $\J_L^f$ is critical for ``locally linear selectivity.''
    
    \begin{dfn}
    	Let $\bar{P}=(P_i)_{i<\omega}$ be a partition of $\omega$. We say that $A\subseteq\omega$ is a \emph{linear selector} of $\bar{P}$ if $|A\cap P_i|\leq i$ for all $i<\omega$. An ideal $\I$ on $\omega$ is \emph{locally linear selective} if for every partition $\bar{P}$ of $\omega$ into sets in $\I$, there is an $\I$-positive linear selector of $\bar{P}$.
    \end{dfn}

    \begin{prop}\label{prop:chara_JL}
        Let $\I$ be an ideal on $\omega$. Then the following are equivalent:
        \begin{enumerate}
            \item $\J_L^f \leq_{\mathrm{K}} \I$ for all $f\in\omega^\omega$ such that $f\gg 1$.
            \item $\J_L^f \leq_{\mathrm{K}} \I$ for some $f\in\omega^\omega$ such that $f\gg 1$.
            \item $\I$ is not locally linear selective.
        \end{enumerate}
    \end{prop}
    \begin{proof}
        (1) trivially implies (2). To show that (2) implies (3), assume that $\J_L^f \leq_{\mathrm{K}} \I$ is witnessed by $\pi\colon\omega\to\omega\times\omega$. Let
        \[
        P_i = \bigcup\{\pi^{-1}[\{j\}\times\omega]: f(j) = i\}.
        \]
        Then $\overline{P} = \langle P_i\rangle_{i<\omega}$ is a partition of $\omega$ into sets in $\I$, since each $\{j\}\times\omega$ belongs to $\J_L^f$ and for each $i$ there are only finitely many $j$ such that $f(j)=i$ by $f\gg1$. Now let $A\subset\omega$ be a linear selector of $\overline{P}$. Then for all $i<\omega$,
        \[
        \lvert (\pi[A])_i \rvert \leq \lvert \pi[A\cap P_{f(i)}]\rvert \leq \lvert A\cap P_{f(i)}\rvert \leq f(i)
        \]
        and thus $\pi[A]\in\J_L^f$. So $A\subset \pi^{-1}[\pi[A]]\in\I$.

        To show that (3) implies (1), let $f\in\omega^\omega$ be such that $f\gg 1$ and let $\overline{P} = \langle P_i\rangle_{i<\omega}$ be a partition of $\omega$ into sets in $\I$ such that all linear selectors are in $\I$. Take any function $\pi\colon\omega\to\omega\times\omega$ such that for all $i<\omega$, $\pi[P_i]\subset\{j_i\}\times\omega$ and $\pi\on P_i$ is injective, where $j_i$ is the least $j_i$ such that  $f(j_i +1) > i$.
        To show that $\pi$ witnesses $\J_L^f \leq_{\mathrm{K}} \I$, let $A\in\J_L^f$ 
        and take $k,j^*<\omega$ such that for all $j\geq j^*$,
        $\lvert (A)_{j}\rvert \leq k\cdot f(j)$.
        Let $i_0<\omega$ be such that $j_i\geq j^*$ for all $i\geq i_0$. Then, for all $i\geq i_0$,
        \[
        \lvert \pi^{-1}[A]\cap P_i\rvert = \lvert \pi[\pi^{-1}[A]\cap P_i]\rvert \leq \lvert (A)_{j_i}\rvert \leq k\cdot f(j_i) \leq k\cdot i.
        \]
        Thus $\pi^{-1}[A]$ is included by a union of $(P_i)_{i<i_0}$ and $k$-many linear selectors of $\bar{P}$, so $\pi^{-1}[A]\in\I$ by (3).
    \end{proof}
    
    In particular, $\J_L^f$ does not depend on $f\in\oo$ in the sense of  Kat\v{e}tov--Blass order:

    \begin{cor}\label{cor:JL_KB_eq}
        $\J_L^f\cong_{\mathrm{KB}}\J_L$ for any $f\gg1$.
    \end{cor}
    \begin{proof}
        In the previous proof of (3) $\Rightarrow$ (1), when $f\in\oo$ is the identity function, $j_i=i$ and hence $\pi\colon\omega\to\omega\times\omega$ is injective. This implies that if $\I$ is not locally linear selective, then $\J_L \leq_{\mathrm{KB}} \I$. Thus $\J_L \leq_{\mathrm{KB}} \J_L^f$ for any $f\gg1$ by (2) $\Rightarrow$ (3). $\J_L^f \leq_{\mathrm{KB}} \J_L$ follows from (2) $\Rightarrow$ (1).
     \end{proof} 

    The relation of $\J_L$ and other Borel ideals with respect to the Kat\v{e}tov order can be summarized as follows.
    
    \begin{prop}\label{prop:JL_KB}\leavevmode
        \begin{enumerate}
            \item $\J_L \leq_{\mathrm{KB}} \I_L^f$ for all $1\ll f\ll \exp$.
            \item $\ed\leq_{\mathrm{KB}}\J_L\leq_{\mathrm{KB}}\finfin$ and $\ed\not\geq_{\mathrm{K}}\J_L\not\geq_{\mathrm{K}}\finfin$.
            \item $\edfin\not\leq_{\mathrm{K}}\J_L$ and $\J_L\not\leq_{\mathrm{K}}\edfin$.
        \end{enumerate}
    \end{prop}
    \begin{proof}
        (1) $\J_L\cong_{\mathrm{KB}}\J^f_L \leq_{\mathrm{KB}} \I^f_L$ follows from \Cref{cor:JL_KB_eq}.

        (2) $\ed\subseteq\J_L\subseteq\finfin$ by definition, so $\ed\leq_{\mathrm{KB}}\J_L\leq_{\mathrm{KB}}\finfin$ hold. To see $\finfin\not\leq_{\mathrm{K}} \J_L$, note that any tall analytic P-ideal is Kat\v{e}tov--Blass above some $\I_L^f$, and thus above $\J_L$. Since $\finfin\not\leq_{\mathrm{K}} \I$ for any P-ideal $\I$ (cf.\ \cite[Observation 2.3]{BF17}), $\finfin\not\leq_{\mathrm{K}} \J_L$. To show that $\J_L\not\leq_{\mathrm{KB}}\ed$, it is enough to show that $\ed$ does not satisfy (3) in \cref{prop:chara_JL}. Let $\langle P_n\rangle_{n<\omega}$ be a partition of $\omega$ into sets in $\ed$.  There are two cases:
        \begin{enumerate}[label=(\roman*)]
            \item For infinitely many $i<\omega$, there is $n<\omega$ such that $\lvert P_n\cap (\{i\}\times\omega)\rvert = \omega$.
            \item For all but finitely many $i<\omega$, there are infinitely many $n<\omega$ such that $P_n\cap(\{i\}\times\omega)\neq\emptyset$.
        \end{enumerate}
        In both cases, it is easy to find a linear-selector of $\langle P_n\rangle_{n<\omega}$ that is not in $\ed$.
        
        (3) Since $\edfin\not\leq_{\mathrm{K}}\finfin$ (cf.\ \cite{Hru17}), $\edfin\not\leq_{\mathrm{K}}\J_L$. To show that $\J_L\not\leq_{\mathrm{K}}\edfin$, it is enough to show that $\edfin$ does not satisfy (3) in \cref{prop:chara_JL}. Let $\langle P_n\rangle_{n<\omega}$ be a partition of $\omega$ into sets in $\edfin$.  Let $\Delta_n = \{n\}\times (n+1)$. Then either
        \begin{enumerate}[label=(\roman*)]
            \item $\sup_{n<\omega}\lvert\{k<\omega: P_k\cap\Delta_n\neq\emptyset\}\rvert < \omega$, or
            \item $\sup_{n<\omega}\lvert\{k<\omega: P_k\cap\Delta_n\neq\emptyset\}\rvert = \omega$.
        \end{enumerate}
        In either case, it is easy to find a linear-selector of $\langle P_n\rangle_{n<\omega}$ that is not in $\edfin$.
    \end{proof}

    \begin{rem}
        Since $\J_L$ is not a P-ideal and not Kat\v{e}tov above $\edfin$,
        \[
        \add^*(\J_L) = \non^*(\J_L) = \omega.
        \]
        It is easy to see that $\J_L$ has an uncountable strongly unbounded set, so we have
        \begin{gather*}
            \add^*_\omega(\J_L) = \add(M_{\J_L}) = \add(K_{\J_L}) = \omega_1,\\
            \cof^*(\J_L) = \cof^*_\omega(\J_L) = \cof(M_{\J_L}) = \cof(K_{\J_L}) = \cc.
        \end{gather*}
        In \cite{DFGT21}, $\cov^*(\J_L) = \nonm$ is shown. It is not hard to show $\non^*_\omega(\J_L) = \covm$; since $\ed\subset\J_L$, $\covm=\non^*_\omega(\ed)\leq\non^*_\omega(\J_L)$ holds and the proof of $\non^*_\omega(\ed)\leq\covm$ (\Cref{prop:non_star_omega_ed}) can be modified to show that $\non^*_\omega(\J_L)\leq\covm$.
        Since $\ed\subset \J_L \subset \finfin$, we have
        \begin{gather*}
            \non(K_{\J_L}) = \bb, \cov(K_{\J_L}) = \dd, \\
            \ee^{\mathrm{const}}(2) \leq \non(M_{\J_L}) \leq \ee^\mathrm{const}_\leq(2),\\
            \vv^{\mathrm{const}}_{\leq}(2) \leq \cov(M_{\J_L}) \leq \vv^\mathrm{const}(2).
        \end{gather*}
    \end{rem}

    Next we consider the $f$-linear growth ideals $\I^f_L$. Unlike the $\J^f_L$'s, these ideals are not necessarily Kat\v{e}tov--Blass equivalent. Recall that $f \leq^* g$ implies $\I_L^f \subset \I_L^g$.

    \begin{prop}\label{prop:no_KB_minimal}
	   There is no $\leq_{\mathrm{KB}}$-minimal or maximal member among the family $\{\mathscr{I}_{\mathrm{L}}^f:1\ll f\ll\exp\}$.
    \end{prop}
    \begin{proof}
    	In the rest of the proof, we write $P_n = P_n^{\exp}$. 
    	Let $F\coloneqq\{f\in\oo:1\ll f\ll\exp\}$.
    	For $f,g\in F$ and $n<\omega$, define the following condition:
    	\begin{equation}
    		\label{eq:f_g_pigeon}
    		g(n)^2\leq f(n)\text{ and }g(n)^2 \cdot \left( 1+\sum_{k<n} \left\lceil \frac{2^k}{f(k)} \right\rceil \right) \leq 2^n.
    	\end{equation}
    	It can be easily seen that ``$\forall g\in F~\exists f\in F~\forall^\infty n<\omega~(\text{\eqref{eq:f_g_pigeon} holds})$'' and ``$\forall f\in F~\exists g\in F~\forall^\infty n<\omega~(\text{\eqref{eq:f_g_pigeon} holds})$''.
    	Thus, it suffices to show that if $f,g\in F$ satisfy \eqref{eq:f_g_pigeon} for almost all $n<\omega$, then $\mathscr{I}_{\mathrm{L}}^g\lneq_{\mathrm{KB}}\mathscr{I}_{\mathrm{L}}^f$.
    	Take such $f,g$ and we may assume that \eqref{eq:f_g_pigeon} is true for all $n<\omega$. $\I_L^g \leq_{\mathrm{KB}}\I_L^f$ is obvious as $g\leq^* f$.	
    	To see $\I_L^f \not\leq_{\mathrm{KB}} \I_L^g$, suppose toward a contradiction that $\pi\colon\omega\to\omega$ witnesses $\I_L^f\leq_{KB}\I_L^g$. It suffices to find $A\in\I_L^f$ such that $\pi^{-1}[A]\notin\I_L^g$.
    	To construct such a set $A$, we inductively define $n_i<\omega$ and finite sets $A_i\subset\omega$ as follows.
    	Let $i<\omega$ and assume $(n_j)_{j<i}$ and $(A_j)_{j<i}$ are defined.
    	First, pick $n_i<\omega$ so that if $A_j\cap P_n\neq\emptyset$ for some $j<i$ and $n<\omega$, then for all $m\geq n_i$,
    	\begin{equation}\label{eqn:disjointness}
    		\pi[P_m]\cap P_n = \emptyset.
    	\end{equation}
    	This is possible because $\pi$ is finite-to-one. 
    	To define $A_i$, 
    	put \[N_i \coloneqq \sum_{k<n_i} \left\lceil \frac{2^k}{f(k)} \right\rceil ,\]
    	and let $\langle I_k\rangle_{k< N_i}$ be a subpartition of $\langle P_n\rangle_{n< n_i}$ such that if $I_k \subseteq P_n$ then $|I_k|\leq f(n)$. Put $I_{N_i} = \bigcup_{m\geq n_i}P_m$. The pigeonhole principle implies that there is $k\leq N_i$ such that
    	\[
    	\lvert\pi^{-1}[I_k]\cap P_{n_i}\rvert \geq g(n_i)^2,
    	\]
    	since otherwise \eqref{eq:f_g_pigeon} implies
    	\[|P_{n_i}|=2^{n_i}<g(n_i)^2\cdot (N_i+1)\leq 2^{n_i},\]
    	a contradiction.
    	Let $k_i$ be the least such $k$. Let $\{a^i_j: j<l\}$ be the increasing enumeration of $I_{k_i}\cap\pi[P_{n_i}]$ and let $A_i = \{a^i_j:j<\min\{g(n_i)^2, l\}\}$.
    	
    	Now set $A = \bigcup_{i<\omega}A_i$. Then the following hold:
    	\begin{enumerate}
    		\item For all $i<\omega$, $\lvert\pi^{-1}[A_i]\cap P_{n_i}\rvert \geq g(n_i)^2$.
    		\item For all $i<\omega$, $\lvert A_i\rvert\leq g(n_i)^2$ and if $A_i\cap P_n\neq\emptyset$ for some $n < n_i$ then $\lvert A_i\rvert\leq \lvert I_{k_i}\rvert\leq f(n)$.
    		\item For all $n<\omega$, there is $i<\omega$ such that $A\cap P_n = A_i \cap P_n$ and such $i$ is unique if $A\cap P_n \neq\emptyset$.
    	\end{enumerate}
    	Note that the third clause follows from (\ref{eqn:disjointness}). Then $\pi^{-1}[A]\notin\I_L^g$ because by (1),
    	\[
    	\lvert\pi^{-1}[A]\cap P_{n_i}\rvert \geq \lvert\pi^{-1}[A_i]\cap P_{n_i}\rvert\geq g(n_i)^2
    	\]
    	for all $i<\omega$. To see $A\in\I_L^f$, it suffices to show that $\lvert A\cap P_n\rvert\leq f(n)$ for all $n>0$. By (3), for each $n>0$, there is $i<\omega$ such that $A\cap P_n = A_i\cap P_n$. Then by (2), if $n < n_i$, then $\lvert A_i\rvert\leq f(n)$, or else
    	\[
    	\lvert A\cap P_n\rvert = \lvert A_i \cap P_n\rvert \leq g(n_i)^2 \leq f(n_i) \leq f(n).
    	\]
    	This completes the proof.
    \end{proof}
        
    \begin{fac}[{\cite[Theorem 2.2]{HMTU17}}]
    	\label{fac:EDfin_chara}
    	$\mathcal{ED}_\mathrm{fin}\leq_{\mathrm{KB}}\I$ if and only if there is a partition of $\omega$ into finite sets such that all the selectors are in $\I$.
    \end{fac}

    Recall that an ultrafilter $\mathcal{U}$ on $\omega$ is \emph{rapid} if for every partition $\bar{P}$ of $\omega$ into finite sets, there is a linear selector of $\bar{P}$ in $\mathcal{U}$. According to the naming convention in \cite{BTW82}, we introduce the following notion.
    
    \begin{dfn}
        An ideal $\I$ on $\omega$ is \emph{locally rapid} if for every partition $\bar{P}$ of $\omega$ into finite sets, there is an $\I$-positive linear selector of $\bar{P}$.
    \end{dfn}

    \begin{prop}\label{prop:chara_f_growth}
        The following are equivalent:
        \begin{enumerate}
            \item There is $f\in\oo$ such that $1\ll f\ll\exp$ and $\IL^f\leq_{\mathrm{KB}} \I$. \label{item:If_KB}
            \item $\I$ is not locally rapid. 
            \label{item:non_locally_rapid}
        \end{enumerate}
    \end{prop}
    \begin{proof}
        To show \eqref{item:If_KB} implies \eqref{item:non_locally_rapid}, assume that $f\in\oo$ is such that $\IL^f\leq_{\mathrm{KB}} \I$, witnessed by the function $g\colon\omega\to\omega$. For $j\in\omega$, put $Q_j\coloneqq\bigcup\{P_i^\mathrm{exp}:i<\omega,f(i)=j\}$ and $P_j\coloneqq g^{-1}(Q_j)$. Note that $\bar{Q}\coloneqq(Q_j)_{j<\omega}$ is a partition of $\omega$ (possibly including $\emptyset$), so is $\bar{P}\coloneqq(P_j)_{j<\omega}$. To see that $\bar{P}$ witnesses \eqref{item:non_locally_rapid}, let $A\subseteq\omega$ be a linear selector of $\bar{P}$.
        Let $i<\omega$ be arbitrary and $j=f(i)$.
        Then, $g(A)\cap P_i^\mathrm{exp}\subseteq g(A)\cap Q_j=g(A\cap g^{-1}(Q_j))=g(A\cap P_j)$ has size $\leq|A\cap P_j|\leq j=f(i)$ since $A$ is a linear selector of $\bar{P}$ and we shall show $A\in\I$. Since $i$ was arbitrary, $g(A)\in\IL^f$ (witnessed by $k=1$). $g$ witnesses $\IL^f\leq_{\mathrm{KB}} \I$, so $A\subseteq g^{-1}(g(A))\in\I$.
        
    	To show \eqref{item:non_locally_rapid} implies \eqref{item:If_KB}, let $\bar{P}=(P_j)_{j<\omega}$ be a partition of $\omega$ into finite sets such that all the linear selectors are in $\I$. Take $i_0<i_1<\cdots<\omega$ such that $|P_j|\leq|P_{i_j}^\mathrm{exp}|$ and choose some function $f\in\oo$ such that $f(i_j)=j$ for all $j$. We may assume $f\ll\exp$, by taking a subsequence of $(i_j)_{j<\omega}$ if necessary. Let $g\colon\omega\to\omega$ be an injection such that $g(P_j)\subseteq P^\mathrm{exp}_{i_j}$. To see that $g$ witnesses $\IL^f\leq_{\mathrm{KB}} \I$, let us assume $A\in \IL^f$ and we shall show $B\coloneqq g^{-1}(A)\in\I$. Take $k<\omega$ such that for each $i<\omega$, $|A\cap P_i^\mathrm{exp}|\leq k\cdot f(i)$. For each $j<\omega$, $|B\cap P_j |=|g(B\cap P_j)|\leq |A\cap P_{i_j}^\mathrm{exp}|\leq k\cdot f(i_j)=k\cdot j$, so $B$ can be partitioned into $k$-many sets $(B_l)_{l<k}$ such that for each $j<\omega$, $|B_l\cap P_j |\leq j$. By \eqref{item:non_locally_rapid}, each $B_l$ is in $\I$ and hence $B=\bigcup_{l<k}B_l\in\I$.
    \end{proof}
    
    \begin{cor}\label{cor:chara_rapid_ult}
        Let $\mathscr{U}$ be an ultrafilter on $\omega$. Then $\mathscr{U}$ is rapid if and only if $\I_L^f\not\leq_{\mathrm{KB}}\mathscr{U}^*$ for all $f\in\omega^\omega$ such that $1\ll f\ll \exp$.
    \end{cor}
    
    For the following results, we will use Solecki's characterization for analytic P-ideals $\I=\Exh(\phi)$ in \Cref{fac:submeasure_chara}. When $\I$ is tall, the submeasure $\phi$ has the following property:
    
    \begin{lem}[{\cite[Lemma 1.4]{HH07}}]
    	\label{lem:tall_HH07}
    	Let $\I=\Exh(\phi)$ where $\phi$ is a lower semi-continuous submeasure on $\omega$. Then $\I$ is tall if and only if $\lim_{n\to\omega}\phi(\{n\})=0$.
    \end{lem}
    
    \begin{prop}\label{prop:tallAP_non_locally_rapid}
    	Every tall analytic P-ideal is not locally rapid.
    \end{prop}
    \begin{proof}
     Let $\I=\Exh(\phi)$ where $\phi$ is a lower semi-continuous submeasure as in \Cref{fac:submeasure_chara}. Inductively take $M_0<M_1<\cdots<\omega$ as follows:
    	$M_0\coloneq 0$ and for $i>0$, 
    	\begin{equation*}
    		M_i\coloneq\min\left\{M>M_{i-1}:\text{ for }n\geq M,~\phi(\{n\})\leq\dfrac{1}{i\cdot2^i}\right\},
    	\end{equation*}
    	which is possible by Lemma \ref{lem:tall_HH07}.
    	Let $P_i\coloneq\left[M_i,M_{i+1}\right)$. To see the interval partition $\bar{P}=(P_i)_{i<\omega}$ witnesses \eqref{item:non_locally_rapid} in \Cref{prop:chara_f_growth} , assume $A\subseteq \omega$ is a linear-selector of $\bar{P}$. For $i>i_0$:
    	\begin{align*}
    		\phi(A\cap P_i)&\leq \sum_{n\in A\cap P_i}\phi(\{n\})\\
    					   &\leq |A\cap P_i|\cdot\dfrac{1}{i\cdot2^i}\leq\dfrac{1}{2^i}.
    	\end{align*}
    	Thus, 
    	\begin{align*}
    		\phi(A\setminus M_i)&=\phi\left(A\cap \bigcup_{k\geq i} P_k\right) = \sup_{i^\prime\geq i}\phi\left(A\cap \bigcup_{i\leq k \leq i^\prime} P_k\right)\\
    		&\leq\sup_{i^\prime\geq i}\sum_{i\leq k \leq i^\prime}\phi\left(A\cap P_k\right)\leq\sum_{i^\prime\geq i}\dfrac{1}{2^{i^\prime}}=\dfrac{1}{2^{i-1}}. 
    	\end{align*}
    	Therefore, $\lim_{n\to\infty}\phi(A\setminus n)=0$, so $A\in \I=\Exh(\phi)$. Hence $\I$ is not locally rapid.
    \end{proof}

    We end this section with providing a characterization of locally rapid Borel ideals. Such a characterization was suggested by Laflamme in \cite[Section 4]{Laf96}, but he did not write a precise statement in the paper, so we will include it with a proof for the record.\footnote{Note that an ideal is locally rapid if and only if its dual filter is weakly rapid in Laflamme's sense.} Following \cite{Laf96} and \cite{KZ17}, we say that an ideal $\I$ on $\omega$ is \emph{$\omega$-diagonalizable by elements of $\mathcal{A}\subset[\omega]^{\omega}$} if there is a sequence $\langle A_n\rangle_{n<\omega}$ of elements of $\mathcal{A}$ such that
    \[
    \forall X\in\I~\exists n<\omega~\lvert A_n\cap X\rvert<\omega.
    \]
    By definition, $\non^*(\I) > \omega$ if and only if $\I$ is not $\omega$-diagonalizable by elements of $[\omega]^\omega$. \cref{lem:addstar_nonstar} (2) states that if $\I$ is Borel, these conditions are also equivalent to $\edfin\leq_{\mathrm{KB}}\I$. We introduce the following variant of $\omega$-diagonalizability.

    \begin{dfn}
        We say that $A\subset\Fin$ is \emph{unbounded}\footnote{This terminology was used in \cite{HMM10}. Unboundedness is also called $\mathrm{Fr}$-universality, where $\mathrm{Fr}$ is the Fr\'echet filter, in \cite{Laf96}.} if for any $k<\omega$, there is $a\in A$ such that $a\cap k = \emptyset$. For an ideal $\I$ on $\omega$ and $\mathcal{A}\subset [\Fin]^{\omega}$, we say that $\I$ is \emph{$\omega$-diagonalizable by unbounded elements of $\mathcal{A}$} if there is a sequence $\langle A_n\rangle_{n<\omega}$ of unbounded elements of $\mathcal{A}$ such that
        \[
        \forall X\in\I~\exists n<\omega~\lvert\{a\in A_n: a\subset X\}\rvert < \omega.
        \]
    \end{dfn}
    
    Next we introduce the locally rapid ideal game $\mathcal{G}_{\textrm{local-rapid}}(\I)$ according to Laflamme's suggestion. The game consists of $\omega$ rounds and a typical run of this game looks as follows.
    \[
    \begin{array}{c|ccccc}
        \mathrm{I} & X_0 & & X_1 & & \cdots \\ \hline
        \mathrm{II} & & s_0 & & s_1 & \cdots
    \end{array}
    \]
    In this game, Player I must play finite subsets $X_n$ of $\omega$ and Player II must respond by choosing $s_n \in [\omega\setminus X_n]^{\leq n}$. Player I wins if and only if $\bigcup_{n<\omega}s_n\in\I$.

    \begin{prop}[Laflamme \cite{Laf96}]\label{weakly_rapid_game}
        Let $\I$ be an ideal on $\omega$. Then the following holds.
        \begin{enumerate}
            \item Player I has a winning strategy in $\mathcal{G}_{\textrm{local-rapid}}(\I)$ if and only if $\I$ is not locally rapid.
            \item Player II has a winning strategy in $\mathcal{G}_{\textrm{local-rapid}}(\I)$ if and only if $\I$ is $\omega$-diagonalizable by unbounded elements of $\bigcup_{n<\omega}[[\omega]^{\leq n}]^{\omega}$.
        \end{enumerate}
    \end{prop}

    Borel determinacy immediately implies the desired characterization.
    
    \begin{cor}\label{char_of_weakly_rapid_for_Borel}
        For any Borel ideal $\I$ on $\omega$, $\I$ is locally rapid if and only if it is $\omega$-diagonalizable by unbounded elements of $\bigcup_{n<\omega}[[\omega]^{\leq n}]^{\omega}$.
    \end{cor}

    In the proof of \cref{weakly_rapid_game}, we use the following notation. For any strategy $\sigma$ of Player I and any move $\overline{s}$ of Player II, $\sigma * \overline{s}$ denotes the next move of Player I according to $\sigma$ against $\overline{s}$. Analogously, for any strategy $\tau$ of Player II and any move $\overline{X}$ of Player I, $\overline{X}*\tau$ denotes the next move of Player II according to $\tau$ against $\overline{X}$.

    \begin{proof}[Proof of \cref{weakly_rapid_game}]
        To show (1), suppose that $\I$ is not locally rapid witnessed by a partition $\langle P_n\rangle_{n<\omega}$ of $\omega$ into finite sets. We define a strategy $\sigma$ for I by
        \[
        \sigma*\langle s_0, \ldots, s_{n-1}\rangle = \bigcup\{P_k: k < n \land  P_k\cap\bigcup_{i<n}s_i \neq \emptyset\}
        \]
        To show that $\sigma$ is a winning strategy for Player I, suppose that II plays $\langle s_i\rangle_{i<\omega}$ against $\sigma$. Let $S = \bigcup_{i<\omega} s_i$. Then for all $n < \omega$, there is $k\leq n$ such that $S\cap P_n \subset s_k$ and thus $\lvert S \cap P_n\rvert \leq k \leq n$. By the choice of $\langle P_n\rangle_{n<\omega}$, $S\in\I$. So Player I wins the game.

        Now suppose that Player I has a winning strategy $\sigma$. Let $g\in\omega^\omega$ such that $\sigma(\emptyset)\subset g(0)$ and for each $n<\omega$,
        \[
        \bigcup\{\sigma*\langle s_0, \ldots, s_n\rangle:\forall i \leq n\,(s_i\subset g(n))\}\subset g(n+1).
        \]
        Set $P_0 = [0, g(1))$ and $P_{n+1} = [g(n+1), g(n+2))$ for all $n<\omega$. Suppose toward a contradiction that $\I$ is locally rapid. Then there is $S\in\I^+$ such that for all $n<\omega$, $\lvert S \cap P_n \rvert \leq n$. Take such an $S$ and for each $i<2$, let $S_i = S \cap \bigcup_{k < \omega} P_{2k+i+1}$. Since $S = (S\cap P_0) \cup S_0 \cup S_1$, we have $S_i\in\I^+$ for some $i < 2$. Assuming $S_0\in\I^+$, Player II can defeat $\sigma$ by playing $s_{2k} = \emptyset$ and $s_{2k+1} = S_0\cap P_{2k+1}$ for all $k<\omega$. Player II can defeat $\sigma$ in an analogous way in the other case too, which is a contradiction.

        To show (2), suppose that $\I$ is $\omega$-diagonalized by a sequence $\langle A_n\rangle_{n<\omega}$ of unbounded elements of $\bigcup_{n<\omega}[[\omega]^{\leq n}]^{\omega}$. Adding $[\omega]^n$ into the sequence if necessary, we may assume that $A_n \subset [\omega]^{\leq n}$ for all $n<\omega$. Fix a surjection $\pi\colon\omega\to\omega$ such that $\pi(n) \leq n$ and $\lvert\pi^{-1}(n)\rvert=\omega$ for all $n<\omega$. Then we define a strategy $\tau$ for Player II by 
        \[
        \langle X_0, \ldots, X_n\rangle*\tau = \text{the lexicographically least element of } A_{\pi(n)} \cap [\omega\setminus (X_n\cup n)]^{\leq\pi(n)}.
        \]
        This is well-defined since all $A_n$'s are unbounded. To show that $\tau$ is a winning strategy for Player II, let $\langle s_n\rangle_{n<\omega}$ be Player II's move following $\tau$. Then the set $S:=\bigcup_{n<\omega}s_n$ includes infinitely many members of each $A_n$. Since $\I$ is $\omega$-diagonalized by $\langle A_n\rangle_{n<\omega}$, $S\in\I^+$.

        Finally, suppose that Player II has a winning strategy $\tau$ in the game.  We define $X_{\overline{s}}$ for any possible move $\overline{s}$ of player II when she follows $\tau$. This is done by induction on the length of $\overline{s}$ as follows: Let $X_{\emptyset} = \{\langle X \rangle*\tau: X\in\Fin\}$. For each $s\in X_{\emptyset}$, choose $X^s_{\emptyset}$ such that $\langle X^s_{\emptyset}\rangle * \tau = s$. For any $\overline{s}$ with $X_{\overline{s}}$ well-defined and any $t\in X_{\vec{s}}$, define
        \[
        X_{\overline{s}^{\frown}t} = \{\langle X_{\emptyset}^{s_0}, \ldots, X_{\overline{s}}^t, X\rangle*\tau : X\in\Fin\},
        \]
        and for each $u\in X_{\overline{s}^{\frown}t}$, choose $X_{\overline{s}^{\frown}t}^u$ such that $\langle X_{\emptyset}^{s_0}, \ldots, X_{\overline{s}}^t, X_{\overline{s}^{\frown}t}^u\rangle * \tau = u$. We claim that $\I$ is $\omega$-diagonalized by the family of all the sets $X_{\overline{s}}$. Suppose otherwise. Then there is $Y\in\I$ such that for any $\overline{s}$ with $X_{\overline{s}}$ well-defined, for infinitely many $t \in X_{\overline{s}}$,
        \[
        \langle X_{\emptyset}^{s_0}, \ldots, X_{\overline{s}}^t\rangle*\tau \subset Y.
        \]
        Then one can easily produce a run that is a win for Player I, but Player II follows $\tau$. This is a contradiction.
    \end{proof}

    \printbibliography

\end{document}